\documentclass[12pt]{amsart}
\usepackage{geometry} 
\usepackage{amscd}
\usepackage{amssymb}
\usepackage{amsmath}
\usepackage{amsthm}

\newtheorem{thm}{Theorem}[section]
\newtheorem{lemma}[thm]{Lemma}
\newtheorem{corollary}[thm]{Corollary}
\newtheorem{prop}[thm]{Proposition}

\newtheorem{conjecture}[thm]{Conjecture}
\theoremstyle{definition}
\newtheorem{notation}[thm]{Notation}
\newtheorem{rem}[thm]{Remark}

\newtheorem{defn}[thm]{Definition}

\makeatletter

\newcommand{\isom}{\overset{\sim}{\rightarrow}}

\title{Local $\varepsilon$-isomorphisms for rank two $p$-adic representations of $\mathrm{Gal}(\overline{\mathbb{Q}}_p/\mathbb{Q}_p)$ and a functional equation of Kato's Euler system.}
\author{Kentaro Nakamura}
\date{} 

\begin{document}

\maketitle
\pagestyle{plain}
\footnote{2010 Mathematical Subject Classification 11F80 (primary), 11F85, 11S25 (secondary).
}
\begin{abstract}
In this article, we prove (many parts of) the rank two case of the 
Kato's local $\varepsilon$-conjecture using the Colmez's $p$-adic local Langlands correspondence for $\mathrm{GL}_2(\mathbb{Q}_p)$. We show that a Colmez's pairing defined in his study of locally algebraic vectors gives us the conjectural $\varepsilon$-isomorphisms for (almost) all the families of $p$-adic representations of $\mathrm{Gal}(\overline{\mathbb{Q}}_p/\mathbb{Q}_p)$ of rank two, which satisfy the desired interpolation property for the de Rham and trianguline case. For the de Rham and non-trianguline case, we also show this interpolation property for the ``critical" range of Hodge-Tate weights using the Emerton's theorem on the compatibility of 
classical and $p$-adic local Langlands correspondence. As an application, we prove that the Kato's Euler system associated to any Hecke eigen new form satisfies a functional equation which has the same form as predicted by the Kato's global $\varepsilon$-conjecture.

\end{abstract}
\setcounter{tocdepth}{2}
\tableofcontents

\section{Introduction.}
By the ground breaking work of Colmez \cite{Co10b} and many other important works by Berger, Breuil, Dospinescu, Kisin and Paskunas, the $p$-adic local 
Langlands correspondence for $\mathrm{GL}_2(\mathbb{Q}_p)$ is now a theorem (\cite{Pa13}, \cite{CDP14b}). This gives us a correspondence between absolutely irreducible
two dimensional $p$-adic representations of $\mathrm{Gal}(\overline{\mathbb{Q}}_p/\mathbb{Q}_p)$ and absolutely irreducible unitary Banach admissible non-ordinary representations of $\mathrm{GL}_2(\mathbb{Q}_p)$ via the so called the Montreal functor. An important feature of this functor is that 
it also gives us a correspondence between representations with torsion coefficients. From this property, the Colmez's theory is expected to have many applications to problems in number theory concerning the relationship between the $p$-adic variations of the Galois side and those of the automorphic side. For example, Emerton \cite{Em} and Kisin \cite{Ki09} independently applied his theory to the Fontaine-Mazur 
conjecture on the modularity of two dimensional geometric $p$-adic representation of $\mathrm{Gal}(\overline{\mathbb{Q}}/\mathbb{Q})$. 

In the present article, we give another application to the rank two case of a series of Kato's conjectures in \cite{Ka93a}, \cite{Ka93b} on the $p$-adic interpolations of 
special values of $L$-functions and local ($L$-and $\varepsilon$-) constants. There are two main theorems in the article, the first one concerns with the $p$-adic 
local $\varepsilon$-conjecture, where the Colmez's theory crucially enters in, and the second concerns with the  global $\varepsilon$-conjecture, which we roughly explain now (see \S2.1, \S3.1, \S4.1 for more details). In this introduction, we assume $p\not=2$ (for simplicity), fix an isomorphism $\iota:\overline{\mathbb{C}}\isom \overline{\mathbb{Q}}_p$, and set $\zeta^{(l)}:=(\iota(\mathrm{exp}(\frac{2\pi i}{l^n}))_{n\geqq 1}\in \mathbb{Z}_l(1):=\Gamma(\overline{\mathbb{Q}}_p, \mathbb{Z}_l(1))$
 for each prime $l$. Set $\Gamma:=\mathrm{Gal}(\mathbb{Q}(\zeta_{p^{\infty}})/\mathbb{Q})\isom \mathrm{Gal}(\mathbb{Q}_p(\zeta_{p^{\infty}})/\mathbb{Q}_p)$.

 Fix a prime $l$. Let $R$ be a commutative noetherian semi-local $\mathbb{Z}_p$-algebra such that $R/\mathrm{Jac}(R)$ is a finite ring with a $p$-power order, or a finite extension of $\mathbb{Q}_p$. To any $R$-representation $T$ of $G_{\mathbb{Q}_l}$, one can functorially attach a (graded) invertible $R$-module $\Delta_R(T)$ using the determinant of the perfect complex $C^{\bullet}_{\mathrm{cont}}(G_{\mathbb{Q}_l}, T)$ of continuous cochains of $G_{\mathbb{Q}_l}$ with values in $T$. For a pair 
$(R, T)=(L, V)$ such that $L$ is a finite extension of $\mathbb{Q}_p$ and $V$ is any (resp. de Rham)$L$-representation of $G_{\mathbb{Q}_l}$ when $l\not=p$ (resp. $l=p$), one can define a representation $W(V)$  of the Weil-Deligne group $'W_{\mathbb{Q}_l}$ of $\mathbb{Q}_l$ by the Grothendieck local monodromy theorem (resp. the $p$-adic monodromy theorem) when $l\not=p$ (resp. $l=p$). Using the local ($L$-and $\varepsilon$-) constants associated to $W(V)$ (and the Bloch-Kato fundamental exact sequence when $l=p$), one can define a canonical $L$-linear isomorphism which we call the de Rham 
$\varepsilon$-isomorphism 
$$\varepsilon_{L,\zeta^{(l)}}^{\mathrm{dR}}(V):\bold{1}_L\isom \Delta_L(V),$$ where, for any $R$, $\bold{1}_R:=(R,0)$ is the trivial graded invertible $R$-module of degree zero.
The $l$-adic local $\varepsilon$-conjecture \cite{Ka93b} predicts the existence of a canonical isomorphism 
$$\varepsilon_{R,\zeta^{(l)}}(T):\bold{1}_R\isom \Delta_R(T)$$ 
for any pair $(R,T)$ as above which interpolates the de Rham $\varepsilon$-isomorphisms (see Conjecture \ref{2.2} for the precise formulation).  The $l\not=p$ case of this conjecture has been already proved by Yasuda \cite{Ya09}. 
The first main theorem of the present article concerns with the rank two case of the $p$-adic local $\varepsilon$-conjecture (see Theorem \ref{3.0} for more details).

\begin{thm}\label{TheoremA}
Assume $l=p$. For (almost) all the pairs $(R,T)$ as above such that $T$ are of rank one or two, one can canonically define 
$R$-linear isomorphisms 
$$\varepsilon_{R,\zeta^{(p)}}(T):\bold{1}_R\isom \Delta_R(T)$$
which are compatible with base changes, exact sequences and the Tate duality, and satisfy
the following: for any pair $(L, V)$ such that $V$ is de Rham of rank one or two, 
\begin{itemize}
\item[(i)]if $V$ is trianguline, then we have 
$$\varepsilon_{L,\zeta^{(p)}}(V)=\varepsilon_{L,\zeta^{(p)}}^{\mathrm{dR}}(V),$$
\item[(ii)]if $V$ is non-trianguline with distinct Hodge-Tate weights $\{0, k\}$ for $k\geqq 1$, then we have 
$$\varepsilon_{\overline{\mathbb{Q}}_p,\zeta^{(p)}}(V(-r)(\delta))=\varepsilon_{\overline{\mathbb{Q}}_p,\zeta^{(p)}}^{\mathrm{dR}}(V(-r)(\delta))$$
for any pair $(r,\delta)$ such that $0\leqq r\leqq k-1$ and $\delta:\Gamma\rightarrow \overline{\mathbb{Q}}_p^{\times}$ is a homomorphism with finite image.
\end{itemize}

\end{thm}

\begin{rem}\label{2.5}
For $l=p$, this conjecture is much more difficult than that for $l\not=p$, and has been proved only in some special cases before the present article. For the rank one case, it is proved by Kato \cite{Ka93b} (see also \cite{Ve13}). For the cyclotomic deformation (or a more general twists of) crystalline representations, it is proved by Benois-Berger \cite{BB08} and Loeffler-Venjakob-Zerbes \cite{LVZ13}. For the trianguline case, it is proved by the author \cite{Na14b}, more precisely, where we generalized the conjecture for rigid analytic families of $(\varphi,\Gamma)$-modules over the Robba ring, and proved the generalized version of the conjecture for trianguline case which includes all the known cases (except the integrality of 
the local $\varepsilon$-isomorphisms). Hence, the condition (ii) of the theorem above seems to be 
 an essentially new result in the literatures on the local $\varepsilon$-conjecture.

\end{rem}

We next explain the second main result of the article. Let $S$ be a finite set of primes containing $p$. 
Let $\mathbb{Q}_S$ be the maximal Galois extension of $\mathbb{Q}$ which is unramified outside 
$S\cup\{\infty\}$, and set $G_{\mathbb{Q},S}:=\mathrm{Gal}(\mathbb{Q}_S/\mathbb{Q})$. For an $R$-representation $T$ of $G_{\mathbb{Q},S}$, one can also define a graded invertible $R$-module $\Delta^S_R(T)$ using $C^{\bullet}_{\mathrm{cont}}(G_{\mathbb{Q},S},T)$. Set $\Delta_R^{(l)}(T):=\Delta_R(T|_{G_{\mathbb{Q}_l}})$ for each $l\in S$. 
The generalized Iwasawa main conjecture \cite{Ka93a} predicts the existence of the canonical isomorphism 
$$z^S_R(T):\bold{1}_R\isom \Delta_R^S(T)$$ for any pair $(R,T)$ as above which interpolates the special values of the $L$-functions of the motives over $\mathbb{Q}$ with good reduction outside $S$ 
(see Conjecture \ref{GIMC} and the references there for the precise formulation). Ser $T^*:=\mathrm{Hom}_R(T,R)(1)$ be the Tate dual of $T$. By the Poitou-Tate duality, one has a canonical isomorphism $\Delta_R^S(T^*)\isom \boxtimes_{l\in S}\Delta_R^{(l)}(T)\boxtimes \Delta_R(T)$. Then, the global $\varepsilon$-conjecture \cite{Ka93b} asserts that one has the equality 
$z^S_R(T^*)=\boxtimes_{l\in S}\varepsilon_{R,\zeta^{(l)}}(T|_{G_{\mathbb{Q}_l}})\boxtimes z^S_R(T)$, where $\varepsilon_{R,\zeta^{(l)}}(T|_{G_{\mathbb{Q}_l}}):\bold{1}_R\isom \Delta_R^{(l)}(T)$ is the local $\varepsilon$-isomorphism 
defined by \cite{Ya09} for $l\not=p$ and the conjectural $\varepsilon$-isomorphism for $l=p$. 
Set $\Lambda_R(\Gamma):=R[[\Gamma]]$, and define an $R[[\Gamma]]$-representation 
$\bold{Dfm}(T):=T\otimes_R\Lambda_R(\Gamma)$ on which $G_{\mathbb{Q},S}$ acts by 
$g(x\otimes \lambda):=g(x)\otimes [\bar{g}]^{-1}\cdot \lambda$. We set 
$\Delta_{R}^{\mathrm{Iw},*}(T):=\Delta_{\Lambda_R(\Gamma)}^{*}(\bold{Dfm}(T))$ for $*=S, (l)$, and 
set (conjecturally) $z_R^{S,\mathrm{Iw}}(T):=z_{\Lambda_R(\Gamma)}^{S}(\bold{Dfm}(T))$, $\varepsilon_{R,\zeta^{(l)}}^{\mathrm{Iw}}(T):=\varepsilon_{\Lambda_R(\Gamma),\zeta^{(l)}}^{(l)}(\bold{Dfm}(T))$. Define an involution $\iota:\Lambda_{R}(\Gamma)\isom \Lambda_R(\Gamma):[\gamma]\mapsto [\gamma]^{-1}$, and denote $M^{\iota}$ for the base change by $\iota$ for any 
$\Lambda_R(\Gamma)$-module $M$. Then, one has a canonical isomorphism $\Delta_R^{\mathrm{Iw}}(T^*)^{\iota}\isom \Delta^S_{\Lambda_R(\Gamma)}((\bold{Dfm}(T))^*)$.

The second main theorem concerns with the global $\varepsilon$-conjecture for $\bold{Dfm}(T_f)$ for Hecke eigen new forms $f$. Let $N,k\in \mathbb{Z}_{\geqq 1}$. 
Let $f(\tau):=\sum_{n\geqq 1}a_n(f)q^n\in S_{k+1}(\Gamma_1(N))^{\mathrm{new}}$ be 
a normalized Hecke eigen new form of level $N$ and of weight $k+1$. Set 
$f^*(\tau):=\sum_{n\geqq 1}\overline{a_n(f)}q^n$ where $\overline{(-)}$ is the complex conjugation, then $f^*(\tau)$ is also a Hecke eigen new form in $S_{k+1}(\Gamma_1(N))^{\mathrm{new}}$ by the theory of new form. Set $S:=\{l|N\}\cup\{p\}$, and 
$L:=\mathbb{Q}_p(\{\iota(a_n(f))\}_{n\geqq 1})$. For $f_0=f,f^*$, let $T_{f_0}$ be the $\mathcal{O}_L$-representation of 
$G_{\mathbb{Q},S}$ of rank two associated to $f_0$ defined by Deligne \cite{De69}. In \cite{Ka04}, Kato 
defined an Euler system associated to $f_0$ which interpolates the critical values of the twisted 
$L$-functions associated to $f_0^*$.  Set $Q(\Lambda_{\mathcal{O}_L}(\Gamma))$ the total fraction ring of $\Lambda_{\mathcal{O}_L}(\Gamma)$. As a consequence of Kato's theorem proved in \S 12 of \cite{Ka04}, we define in \S 4.2 a canonical $Q(\Lambda_{\mathcal{O}_L}(\Gamma))$-linear isomorphism 
$$\widetilde{z}_{\mathcal{O}_L}^{\mathrm{Iw},S}(T_{f_0}(r)):\bold{1}_{Q(\Lambda_{\mathcal{O}_L}(\Gamma))}\isom \Delta_{\mathcal{O}_L}^{S,\mathrm{Iw}}(T_{f_0}(r))\otimes_{\Lambda_R(\Gamma)}Q(\Lambda_{\mathcal{O}_L}(\Gamma))$$ for any $r\in \mathbb{Z}$,
which should be the base change to $Q(\Lambda_{\mathcal{O}_L}(\Gamma))$ of the conjectural zeta isomorphism $z^{S,\mathrm{Iw}}_{\mathcal{O}_L}(T_{f_0}(r)):\bold{1}_{\Lambda_{\mathcal{O}_L}(\Gamma)}\isom \Delta_{\mathcal{O}_L}^{\mathrm{Iw},S}(T_{f_0}(r))$. We remark that one has a canonical isomorphism $T_{f^*}(1)[1/p]\isom T_f(k)[1/p]$ and one has a canonical isomorphism 
$\Delta^{\mathrm{Iw},S}_{\mathcal{O}_L}(T_{f^*}(1))^{\iota}\isom 
\Delta_{\mathcal{O}_L(\Gamma)}^{S}((\bold{Dfm}(T_f(k)))^*)$. The second main theorem of the present article concerns with the global $\varepsilon$-conjecture for the pair $(\Lambda_{\mathcal{O}_L}(\Gamma), \bold{Dfm}(T_{f_0}(r))$(see Theorem \ref{4.5}).

\begin{thm}\label{TheoremB}
Assume $V:=T_f[1/p]|_{G_{\mathbb{Q}_p}}$ is absolutely irreducible and 
$\bold{D}_{\mathrm{cris}}(V(-r)(\delta))^{\varphi=1}\allowbreak=0$ for any $0\leqq r\leqq k-1$ and 
$\delta:\Gamma\rightarrow \overline{\mathbb{Q}}_p^{\times}$ with finite image. Then one has the following equality 
$$\widetilde{z}^{\mathrm{Iw},S}_{\mathcal{O}_L}(T_{f^*}(1))^{\iota}=\boxtimes_{l\in S}
(\varepsilon_{\mathcal{O}_L,\zeta^{(l)}}^{\mathrm{Iw},(l)}(T_f(k))\otimes \mathrm{id}_{Q(\Lambda_{\mathcal{O}_L}(\Gamma))})\boxtimes \widetilde{z}_{\mathcal{O}_L}^{\mathrm{Iw},S}(T_f(k))$$ 
under the base change to $Q(\Lambda_{\mathcal{O}_L}(\Gamma))$ of the canonical isomorphism 
$$\Delta_{\mathcal{O}_L}^{\mathrm{Iw},S}(T_{f^*}(1))^{\iota}\isom \boxtimes_{l\in S}\Delta_{\mathcal{O}_L}^{(l)}(T_f(k))\boxtimes \Delta_{\mathcal{O}_L}^{\mathrm{Iw}, S}(T_f(k))$$ 
defined by the Poitou-Tate duality, where the isomorphism $$\varepsilon_{\mathcal{O}_L,\zeta^{(l)}}^{\mathrm{Iw},(l)}(T_f(k)):\bold{1}_{\Lambda_{\mathcal{O}_L(\Gamma)}}\isom \Delta_{\mathcal{O}_L}^{\mathrm{Iw},(l)}(T_f(k))$$ is the local $\varepsilon$-isomorphism for the pair $(\Lambda_{\mathcal{O}_L}(\Gamma), \bold{Dfm}(T_f(k))|_{G_{\mathbb{Q}_l}})$ defined by \cite{Ya09} for $l\not=p$, by 
Theorem \ref{TheoremA} for $l=p$. 
\end{thm}

The cases for which the assumptions do not hold correspond to the ordinary case or the exceptional zero case, which we will treat in the next article \cite{NY}.

The following conjecture is a part of the generalized Iwasawa main conjecture for the pair $(\Lambda_{\mathcal{O}_L}(\Gamma), \bold{Dfm}(T_{f_0}(r))$. 

 \begin{conjecture}\label{4.6.1}
 For any $r\in \mathbb{Z}$, the isomorphism $\widetilde{z}_{\mathcal{O}_L}^{\mathrm{Iw},S}(T_{f_0}(r))$ uniquely 
 descends to 
 a $\Lambda_{\mathcal{O}_L}(\Gamma)$-linear isomorphism 
 $$z^{\mathrm{Iw},S}_{\mathcal{O}_L}(T_{f_0}(r)):\bold{1}_{\Lambda_{\mathcal{O}_L(\Gamma)}}\isom 
 \Delta^{\mathrm{Iw},S}_{\mathcal{O}_L}(T_{f_0}(r)),$$
 i.e. one has  $\widetilde{z}_{\mathcal{O}_L}^{\mathrm{Iw},S}(T_{f_0}(r))=z^{\mathrm{Iw},S}_{\mathcal{O}_L}(T_{f_0}(r))
 \otimes \mathrm{id}_{Q(\Lambda_{\mathcal{O}_L}(\Gamma))}.$
 
 \end{conjecture}
 
  We remark that the uniqueness in the conjecture is trivial since the natural map 
 $\Lambda_{\mathcal{O}_L}(\Gamma)\rightarrow Q(\Lambda_{\mathcal{O}_L}(\Gamma))$ is injective, and, if the conjecture is true for an $r\in \mathbb{Z}$, then it is true for all $r\in \mathbb{Z}$.  As an immediate corollary of the theorem, we obtain the following.
 
 \begin{corollary}
 The conjecture $\ref{4.6.1}$ is true for $f$ 
 if and only if it is true for $f^*$.
 \end{corollary}

Now, we briefly describe the contents of different sections. In \S2, \S3, we study the $p$-adic local $\varepsilon$-conjecture. We first remark that many results in these sections heavily depend on many deep results in the theory of the $p$-adic local Langlands correspondence for $\mathrm{GL}_2(\mathbb{Q}_p)$ (\cite{Co10a},\cite{Co10b},\cite{Do11}, \cite{Em}). In particular, our local $\varepsilon$-isomorphism 
defined in Theorem \ref{TheoremA} is nothing else but the Colmez's pairing defined in VI.6 of \cite{Co10b}. Our contribution is (to find the relation between the Colmez's pairing and the local $\varepsilon$-isomorphism and) to show that this pairing satisfies the interpolation property (i.e. the condition (i), (ii) in the theorem). Section 2 is mainly for preliminaries. In \S2.1, we first recall the $l$-adic and the $p$-adic local $\varepsilon$-conjecture. In \S2.2, we recall the theory of $(\varphi,\Gamma)$-modules and restate the $p$-adic local $\varepsilon$-conjecture in terms of $(\varphi,\Gamma)$-modules. In \S2.3, we propose a conjecture (Conjecture \ref{2.11}) on a conjectural definition of the local $\varepsilon$-isomorphism for any 
$(\varphi,\Gamma)$-modules of any rank using the Colmez's convolution pairing defined in \cite{Co10a}. Section 3 is devoted to the proof of Theorem \ref{TheoremA}, in particular, we prove  Conjecture \ref{2.11} for the rank two case. In \S3.1, we state the main theorem. In \S3.2, we define our local $\varepsilon$-isomorphism using the Colmez's pairing defined in VI.6 of \cite{Co10b}, and prove Conjecture \ref{2.11} for the rank two case, which is essentially a consequence of the $\mathrm{GL}_2(\mathbb{Q}_p)$-compatibility (a notion defined in \S III of \cite{CD14}) of the $(\varphi,\Gamma)$-modules of rank two. The subsections \S3.3 and \S3.4 are the technical hearts of this article, where we show that our $\varepsilon$-isomorphisms satisfy the conditions (i) and (ii) in Conjecture \ref{TheoremA}. In \S3.3, we show the interpolation property for the trianguline case (i.e. the condition (i) in the theorem) by comparing the local $\varepsilon$-isomorphism defined in \S3.2 with that defined in our previous work \cite{Na14b}, where we use a result of Dospinescu \cite{Do11} on the explicit description of the action of $w:=\begin{pmatrix}0& 1\\ 1& 0\end{pmatrix}\in \mathrm{GL}_2(\mathbb{Q}_p)$ on the locally analytic vectors (see Theorem \ref{3.1}). In \S3.4, we show the interpolation property for the non-trianguline case (i.e. the condition (ii) in the theorem). For an absolutely irreducible $(\varphi,\Gamma)$-module $D$ of rank two such that $V(D)$ is de Rham with distinct Hodge-Tate weights $\{0,k\}$ ($k\geqq 1$), using the Colmez's theory of Kirillov model of locally algebraic vectors in VI of \cite{Co10b}, we prove two explicit formulas formulas (Proposition \ref{formula1}, Proposition \ref{formula2}) which respectively (essentially) describe $\varepsilon_L(V(-r)(\delta))$ and 
$\varepsilon_L^{\mathrm{dR}}(V(-r)(\delta))$ for  any  pair $(r,\delta)$ such that $0\leqq r\leqq k-1$ and $\delta:\Gamma\rightarrow L^{\times}$ with finite image. As a corollary of these formula, we prove a result (Corollary \ref{3.12.123}) which enables us to characterize an element in the Iwasawa cohomology by the dual exponential maps. This corollary is a crucial in the proof of Theorem \ref{TheoremB}. Finally, using the Emerton's theorem \cite{Em} on the compatibility of the $p$-adic and the classical local Langlands correspondence and the classical explicit formula of the action of  $w$ on the (classical) Kirillov model, we prove the condition (ii) in Theorem \ref{TheoremA} for the non-trianguline case.

The final section \S4 is devoted to the proof of Theorem \ref{TheoremB}. In \S4.1, we (roughly) recall the generalized Iwasawa main conjecture and the global $\varepsilon$-conjecture. In \S4.2, we (re-)state our second main theorem (Theorem \ref{4.5}) and define our (candidate of) zeta isomorphism $\widetilde{z}_{\mathcal{O}_L}^{\mathrm{Iw},S}(T_f(r))$ using the ($p$-th layer of) Kato's Euler system \cite{Ka04} associated to $f$. In the final subsection \S4.3, we prove Theorem \ref{TheoremB}  (Theorem \ref{4.5}), where we reduce the theorem to the classical functional equation of the (twisted) $L$-function of $f$ using the Kato's explicit reciprocity law, Theorem \ref{TheoremA} and Corollary \ref{3.12.123}.

\begin{notation}
For a field $F$, set $G_F:=\mathrm{Gal}(F^{\mathrm{sep}}/F)$ the absolute Galois group of $F$. For each $m\in \mathbb{Z}_{\geqq 0}$, 
set $\mu_{m}(F)$ the group of 
the $m$-th roots of unity in $F$. 
For each prime $p$, set $\Gamma(F, \mathbb{Z}_p(1)):=\varprojlim_{n\geqq 1}\mu_{p^n}(F)$. For $F=\mathbb{Q}_p$, let $W_{\mathbb{Q}_p}\subseteq G_{\mathbb{Q}_p}$ be the Weil group of $\mathbb{Q}_p$, $I_p\subseteq W_{\mathbb{Q}_p}$ be the inertia subgroup. Let $\mathrm{rec}_{\mathbb{Q}_p}: \mathbb{Q}_p^{\times} 
\isom W_{\mathbb{Q}_p}^{\mathrm{ab}}$ be the reciprocity map of the local class field theory normalized so that $\mathrm{rec}_{\mathbb{Q}_p}(p)$ is a lift of the geometric Frobenius $\mathrm{Fr}_p\in G_{\mathbb{F}_p}$. Set $\Gamma:=\mathrm{Gal}(\mathbb{Q}(\mu_{p^{\infty}})/\mathbb{Q})\isom \mathrm{Gal}(\mathbb{Q}_p(\mu_{p^{\infty}})/\mathbb{Q}_p)$, and let $\chi:\Gamma\isom \mathbb{Z}_p^{\times}$ be the $p$-adic cyclotomic character which we also see as a character of $G_{\mathbb{Q}_p}$. Set $H_{\mathbb{Q}_p}:=\mathrm{Ker}(\chi)\subseteq G_{\mathbb{Q}_p}$. For each $b\in \mathbb{Z}_p^{\times}$, define 
$\sigma_b\in \Gamma$ such that $\chi(\sigma_b)=b$. For a perfect field $k$ of characteristic $p$, 
we denote $W(k)$ for the ring of Witt vectors, on which the lift $\varphi$ of the $p$-th power Frobenius on $k$ acts. Let $[-]:k\rightarrow W(k)$ be the Teichm\"ulcer lift. 
Set $\widetilde{\bold{E}}^+:=\varprojlim_{n\geqq 0}\mathcal{O}_{\mathbb{C}_p}/p$ where the projective limit with respect to $p$-th power map, $\widetilde{\bold{E}}:=\mathrm{Frac}(\widetilde{\bold{E}}^+)$, $\widetilde{\bold{A}}^+:=W(\widetilde{\bold{E}}^+)$, $\widetilde{\bold{A}}:=W(\widetilde{\bold{E}})$, 
$\widetilde{\bold{B}}^+:=\widetilde{\bold{A}}^+[1/p]$ and $\widetilde{\bold{B}}:=\widetilde{\bold{A}}[1/p]$. Let $\theta: \widetilde{\bold{A}}^+\rightarrow \mathcal{O}_{\mathbb{C}_p}$ be the continuous $\mathbb{Z}_p$-algebra homomorphism 
defined by $\theta([(\bar{x}_n)_{n\geqq 0}]):=\mathrm{lim}_{n\rightarrow \infty} x_n^{p^n}$ for 
any $(\bar{x}_n)_{n\geqq 0}\in \widetilde{\bold{E}}^+$, where
$x_n\in \mathcal{O}_{\mathbb{C}_p}$ is a lift of $\bar{x}_n\in \mathcal{O}_{\mathbb{C}_p}/p$. Set
$\bold{B}^+_{\mathrm{dR}}:=\varprojlim_{n\geqq 1} \widetilde{\bold{A}}^+[1/p]/\mathrm{Ker}(\theta)^n[1/p]$. For each $\mathbb{Z}_p$-basis $\zeta=(\zeta_{p^n})_{n\geqq 0}\in \Gamma(\overline{\mathbb{Q}}_p, \mathbb{Z}_p(1))$, define $t_{\zeta}:=\mathrm{log}([(\bar{\zeta}_{p^n})_{n\geqq 0}])\in \bold{B}^+_{\mathrm{dR}}$, which is a uniformizer of $\bold{B}^+_{\mathrm{dR}}$. Set $\bold{B}_{\mathrm{dR}}:=\bold{B}^+_{\mathrm{dR}}[1/t_{\zeta}]$.  For a commutative ring $R$, 
we denote by $\bold{D}^-(R)$ the derived category of bounded below complex of $R$-modules, 
by $\bold{D}_{\mathrm{perf}}(R)$ the full subcategory of perfect complexes of $R$-modules. 
We denote by $\bold{P}_{\mathrm{fg}}(R)$ the category of finite projective $R$-modules. For any $P\in \bold{P}_{\mathrm{fg}}(R)$, we denote by $r_P$ its $R$-rank, by
$P^{\vee}:=\mathrm{Hom}_R(P, R)$ its dual. For $P_1, P_2\in\bold{P}_{\mathrm{fg}}(R)$ and 
$\langle,\rangle:P_1\times P_2\rightarrow R$ a perfect pairing of $R$-modules, we always identify $P_2$ with $P_1^{\vee}$ by the isomorphism $P_2\isom P_1^{\vee}: x\mapsto [y\mapsto \langle y,x\rangle]$.
\end{notation}
\section{Preliminaries and conjectures}
In this section, we first recall  the ($l$-adic and the $p$-adic) local $\varepsilon$-conjecture.
Then, after reviewing the (Iwasawa) cohomology theory of $(\varphi,\Gamma)$-modules, 
we formulate a conjecture on a conjectural definition of the $p$-adic local 
$\varepsilon$-isomorphism using (a multivariable version of) the Colmez's convolution pairing.

\subsection{Review of the local $\varepsilon$-conjecture}
In this subsection, we quickly recall the local $\varepsilon$-conjecture. See the original articles 
\cite{Ka93b}, \cite{FK06} (the latter one includes the non-commutative version) or \cite{Na14b} for more details.

The local $\varepsilon$-conjecture is formulated using the theory of the determinant functor, for which 
we use the Knudsen-Mumford's one \cite{KM76}, which we briefly recall here (see also \S3.1 of \cite{Na14b}).

Let $R$ be a commutative ring. We define a category $\mathcal{P}_{R}$, whose objects are the 
pairs $(L,r)$ where $L$ is an invertible $R$-module and $r:\mathrm{Spec}(R)\rightarrow \mathbb{Z}$ is a
 locally constant function, whose morphisms are defined by $\mathrm{Mor}_{\mathcal{P}_R}((L,r), (M,s)):=
 \mathrm{Isom}_R(L, M)$ if $r=s$, or empty otherwise. We call the objects of this category graded invertible 
 $R$-modules. For $(L,r), (M,s)$, define its product by $(L,r)\boxtimes (M,s):=(L\otimes_R M, r+s)$ with the natural associativity constraint and the commutativity constraint $(L,r)\boxtimes (M,s)\isom (M,s)\boxtimes (L, r):l\otimes m\mapsto (-1)^{r s}m\otimes l$. We always identify $(L,r)\boxtimes (M,s)=(M,s)\boxtimes (L, r)$ by this constraint isomorphism. 
  The unit object for the product 
 is $\bold{1}_R:=(R, 0)$. For each $(L,r)$, we set $(L,r)^{-1}:=(L^{\vee},-r)$, which is the inverse of $(L, r)$ by the isomorphism 
 $i_{(L,r)}:(L,r)\boxtimes(L^{\vee},-r)\isom \bold{1}_R$ induced by the evaluation map 
 $L\otimes_RL^{\vee}\isom R:x\otimes f\mapsto f(x)$. 
 For a ring homomorphism $f:R\rightarrow R'$, one has a 
 base change functor $(-)\otimes_RR':\mathcal{P}_R\rightarrow\mathcal{P}_{R'}$ defined by $(L,r)\mapsto (L,r)\otimes_{R}R':=(L\otimes_RR',r\circ f^*)$ where 
 $f^*:\mathrm{Spec}(R')\rightarrow \mathrm{Spec}(R)$. For a category $\mathcal{C}$, denote by $(\mathcal{C},\mathrm{is})$ the category such that the objects are the same as $\mathcal{C}$ and the morphisms are all the isomorphisms in $\mathcal{C}$. Define a functor
 $$\mathrm{Det}_R:(\bold{P}_{\mathrm{fg}}(R),\mathrm{is})\rightarrow \mathcal{P}_R:
 P\mapsto (\mathrm{det}_RP,r_P)$$
 where we set $\mathrm{det}_RP:=\wedge_R^{r_P}P$. Note that $\mathrm{Det}_R(0)=\bold{1}_R$ is the unit object. For a short exact sequence $0\rightarrow P_1\rightarrow P_2\rightarrow P_3\rightarrow 0$ 
 in $\bold{P}_{\mathrm{fg}}(R)$, we always identify $\mathrm{Det}_R(P_1)\boxtimes \mathrm{Det}_R(P_3)$ with $\mathrm{Det}_R(P_2)$ by the following functorial isomorphism (put $r_i:=r_{P_i}$)
 \begin{equation}\label{17e}
  \mathrm{Det}_R(P_1)\boxtimes \mathrm{Det}_R(P_3)\isom \mathrm{Det}_R(P_2)
  \end{equation}
 induced by $$ (x_1\wedge\cdots\wedge x_{r_1})\otimes(\overline{x_{r_1+1}}\wedge\cdots\wedge \overline{x_{r_2}})\mapsto x_1\wedge\cdots\wedge x_{r_1}\wedge x_{r_1+1}\wedge\cdots \wedge x_{r_2}
 $$ where $x_1,\cdots, x_{r_1}$ (resp. $\overline{x_{r_1+1}},\cdots, \overline{x_{r_2}}$) 
 are local sections of $P_1$ (resp. $P_3$) and $x_i\in P_2$ ($i=r_1+1, \cdots,  r_2$) is a lift of $\overline{{x}_i}\in P_3$. For a bounded complex $P^{\bullet}$ in $\bold{P}_{\mathrm{fg}}(R)$, define $\mathrm{Det}_R(P^{\bullet})\in 
  \mathcal{P}_R$ by 
  $$\mathrm{Det}_R(P^{\bullet}):=\boxtimes_{i\in \mathbb{Z}}\mathrm{Det}_R(P^i)^{(-1)^i}.$$ 
    By the result of \cite{KM76}, $\mathrm{Det}_R$  naturally extends to a functor
  $$\mathrm{Det}_R:(\bold{D}_{\mathrm{perf}}(R),\mathrm{is})\rightarrow \mathcal{P}_R$$
  such that the isomorphism (\ref{17e}) extends to the following situation: 
  for any exact sequence $0\rightarrow P_1^{\bullet}
  \rightarrow P_2^{\bullet}\rightarrow P_3^{\bullet}\rightarrow 0$ of bounded below complexes of $R$-modules such that each $P_i^{\bullet}$ is a perfect complex, then there exists a canonical isomorphism 
   \begin{equation*}
    \mathrm{Det}_R(P^{\bullet}_1)\boxtimes \mathrm{Det}_R(P^{\bullet}_3)\isom  \mathrm{Det}_R(P^{\bullet}_2).
   \end{equation*}
  If 
  $P^{\bullet}\in \bold{D}_{\mathrm{perf}}(R)$ satisfies that $\mathrm{H}^i(P^{\bullet})[0]\in \bold{D}_{\mathrm{perf}}(R)$ for 
  any $i$, there exists a canonical isomorphism 
  $$\mathrm{Det}_R(P^{\bullet})\isom \boxtimes_{i\in \mathbb{Z}}\mathrm{Det}_R(
  \mathrm{H}^i(P^{\bullet})[0])^{(-1)^i}.$$

  For $(L,r)\in \mathcal{P}_R$, define $(L, r)^{\vee}:=(L^{\vee}, r)\in \mathcal{P}_R$, which induces an anti-equivalence $(-)^{\vee}:\mathcal{P}_R\isom \mathcal{P}_R$.
  For $P\in \bold{P}_{\mathrm{fg}}(R)$ and an $R$-basis $\{e_1,\cdots, e_{r_P}\}$, we denote by 
  $\{e_1^{\vee},\cdots, e_{r_P}^{\vee}\}$ its dual basis of $P^{\vee}$. Then 
  one has a canonical isomorphism $\mathrm{Det}_R(P^{\vee})\isom \mathrm{Det}_R(P)^{\vee}$ defined by the isomorphism 
  $$\mathrm{det}_R(P^{\vee})\isom (\mathrm{det}_RP)^{\vee}:e_1^{\vee}\wedge\cdots\wedge e_{r_P}^{\vee}
  \mapsto (e_1\wedge\cdots\wedge e_{r_P})^{\vee}.$$
  This isomorphism naturally extends to $(\bold{D}_{\mathrm{perf}}(R), \mathrm{is})$, i.e. for any $P^{\bullet}\in \bold{D}_{\mathrm{perf}}(R)$, 
  there exists a canonical isomorphism 
  \begin{equation}
\mathrm{Det}_R(\bold{R}\mathrm{Hom}_R(P^{\bullet}, R))\isom \mathrm{Det}_R(P^{\bullet})^{\vee}. \end{equation}

Now, we start to recall the local $\varepsilon$-conjecture. Fix a prime $p$. From now on until the end of the article, we use the notation $R$ to represent a commutative topological $\mathbb{Z}_p$-algebra satisfying one of the following conditions (i) or (ii).
\begin{itemize}
\item[(i)] $R$ is a $\mathrm{Jac}(R)$-adically complete noetherian semi-local ring such that $R/\mathrm{Jac}(R)$ is a finite ring (equipped with the $\mathrm{Jac}(R)$-adic topology), 
where $\mathrm{Jac}(R)$ is the Jacobson radical of $R$,
\item[(ii)] $R$ is a finite extension of $\mathbb{Q}_p$ (equipped with the $p$-adic topology). \end{itemize}
We note that a ring $R$ (satisfying (i) or (ii)) satisfies (i) if and only if $p\not\in R^{\times}$.
We use the notation $L$ instead of $R$ if we consider only the case 
(ii).

In this article, we mainly treat representations (of $G_{\mathbb{Q}_p}$ or $\mathrm{GL}_2(\mathbb{Q}_p)$, etc.) defined over such a ring $R$. Let $G$ be a topological group. We say that $T$ is an $R$-representation of $G$ if $T$ is a 
finite projective $R$-module with a continuous $R$-linear $G$-action. 
For a continuous homomorphism $\delta:G\rightarrow R^{\times}$, we set $R(\delta):=R\bold{e}_{\delta}$ the $R$-representation of rank one with a fixed basis $\bold{e}_{\delta}$ on which $G$ acts by 
 $g(\bold{e}_{\delta}):=\delta(g)\bold{e}_{\delta}$. We always identify $R(\delta^{-1})$ with the $R$-dual $R(\delta)^{\vee}$ by 
 $R(\delta^{-1})\isom R(\delta)^{\vee}:\bold{e}_{\delta^{-1}}\mapsto \bold{e}_{\delta}^{\vee}$, 
 and identify $R(\delta_1)\otimes_R R(\delta_2)\isom 
 R(\delta_1\cdot \delta_2)$ by $\bold{e}_{\delta_1}\otimes \bold{e}_{\delta_2}\isom 
 \bold{e}_{\delta_1\cdot \delta_2}$ for any $\delta_1,\delta_2:G\rightarrow R^{\times}$. We set $T(\delta):=T\otimes_{R}R(\delta)$.
 For an $R$-representation $T$ of $G$, we set $T(\delta):=T\otimes_{R}R(\delta)$ and denote
by $\mathrm{C}_{\mathrm{cont}}^{\bullet}(G, T)$ the complex of continuous cochains of $G$ with values in $T$, i.e. defined by $\mathrm{C}_{\mathrm{cont}}^i(G, T):=\{c:G^{\times i}\rightarrow T:\text{ continuous maps }\}$ for each $i\geqq 0$ with the usual boundary map. We also regard
$\mathrm{C}_{\mathrm{cont}}^{\bullet}(G, T)$ as an object of $\bold{D}^-(R)$.

Now, we fix another prime $l$ (we don't assume $l\not=p$). Let $T$ be an $R$-representation of $G_{\mathbb{Q}_l}$. We set $\mathrm{H}^i(\mathbb{Q}_l, T):=\mathrm{H}^i(\mathrm{C}^{\bullet}_{\mathrm{cont}}(G_{\mathbb{Q}_l}, T))$. For each 
$r\in \mathbb{Z}$, we set $T(r):=T\otimes_{\mathbb{Z}_p}\Gamma(\overline{\mathbb{Q}}_l, \mathbb{Z}_p(1))^{\otimes r}$. We denote by $T^*:=T^{\vee}(1)$  the Tate dual of $T$. By the classical theory of the Galois cohomology of local fields, it is known that one has $\mathrm{C}_{\mathrm{cont}}^{\bullet}(G_{\mathbb{Q}_l},T)\in \bold{D}_{\mathrm{perf}}(R)$. Using the determinant functor, we define the following graded invertible $R$-module 
$$\Delta_{R,1}(T):=\mathrm{Det}_R(C_{\mathrm{cont}}^{\bullet}(G_{\mathbb{Q}_l},T)),$$
which is of degree $-r_T$  (resp. of degree $0$) when $l=p$ (resp. when $l\not=p$) by the Euler-Poincar\'e formula. 

For $a\in R^{\times}$ ($a\in \mathcal{O}_L^{\times}$ if $R=L$), we set 
$$R_a:=\{x\in W(\overline{\mathbb{F}}_p)\hat{\otimes}_{\mathbb{Z}_p}R|(\varphi\otimes \mathrm{id}_R)(x)=(1\otimes a)\cdot x\},$$ which is an invertible $R$-module. For $T$ as above, we freely regard $\mathrm{det}_RT$ as a continuous homomorphism $\mathrm{det}_RT:G_{\mathbb{Q}_l}^{\mathrm{ab}}\rightarrow R^{\times}$. 
Define a constant 
$$a_l(T):=\mathrm{det}_RT(\mathrm{rec}_{\mathbb{Q}_l}(p))\in R^{\times},$$ and define another graded invertible $R$-module 
$$\Delta_{R,2}(T):=\begin{cases}(R_{a_l(T)}, 0) & (l\not=p)\\
(\mathrm{det}_RT\otimes_RR_{a_p(T)}, r_T) & (l=p).\end{cases}$$
Finally, we set 
$$\Delta_{R}(T):=\Delta_{R,1}(T)\boxtimes\Delta_{R,2}(T)$$ 
which we call the local fundamental line.

The local fundamental line is compatible with the functorial operations, i.e. for any 
$R\rightarrow R'$, one has a canonical isomorphism 
$$\Delta_{R}(T)\otimes_RR'\isom \Delta_{R'}(T\otimes_RR'),$$
for any exact sequence $0\rightarrow T_1\rightarrow T_2\rightarrow T_3\rightarrow 0$ of 
$R$-representations of $G_{\mathbb{Q}_l}$, one has a canonical isomorphism 
$$\Delta_R(T_2)\isom \Delta_R(T_1)\boxtimes\Delta_R(T_3),$$
and one has the following canonical isomorphism 
$$\Delta_R(T)\isom \begin{cases} \Delta_R(T^*)^{\vee} & (l\not=p) \\
\Delta_R(T^*)^{\vee}\boxtimes(L(r_T), 0)& (l=p)\end{cases}$$
defined as the product of the following two isomorphisms
$$\Delta_{R,1}(T)\isom \Delta_{R,1}(T^*)^{\vee},$$ 
which is induced by the Tate duality $C^{\bullet}_{\mathrm{cont}}(G_{\mathbb{Q}_l}, T)\isom 
\bold{R}\mathrm{Hom}_R(C^{\bullet}_{\mathrm{cont}}(G_{\mathbb{Q}_l}, T^*), R),$ and 
$$\Delta_{R,2}(T)\isom \begin{cases} \Delta_{R,2}(T^*)^{\vee} & (l\not=p) \\
\Delta_{R,2}(T^*)^{\vee}\boxtimes(L(r_T), 0)& (l=p)\end{cases}$$
which is defined by $x\mapsto [y\mapsto x\otimes y]$ for $x\in R_{a_l(T)}, y\in R_{a_l(T^*)}$ when $l\not=p$ , by $x\otimes y\mapsto [z\otimes w\mapsto y\otimes w]\otimes z(x)$ for 
$x\in \mathrm{det}_RT, y\in R_{a_p(T)}, z\in \mathrm{det}_R(T^*)=(\mathrm{det}_RT)^{\vee}(r_T), 
w\in R_{a_p(T^*)}$ when $l=p$
(remark that one has $R_{a_l(T)}\otimes_RR_{a_l(T^*)}=R$ since one has $a_l(T)\cdot a_{l}(T^*)=1$ for any $l$).

The local $\varepsilon$-conjecture concerns with the existence of a compatible family of trivializations $\varepsilon_{R,\zeta}(T):\bold{1}_R\isom\Delta_{R}(T)$, which we call the local $\varepsilon$-isomorphisms, for all the pairs $(R, T)$ as above which interpolate the trivializations 
$\varepsilon^{\mathrm{dR}}_{L,\zeta}(V):\bold{1}_L\isom\Delta_{L}(V)$, which we call the de Rham $\varepsilon$-isomorphisms, for all the pairs $(L,V)=(R,T)$ such that $V$ is de Rham (resp. any) if $l=p$ (resp. if $l\not=p$), whose definition we briefly recall now. 

We first recall the $\varepsilon$-constants defined for the representations of the Weil-Deligne group $'W_{\mathbb{Q}_l}$ of $\mathbb{Q}_l$. Let $K$ be a field of characteristic zero which contains all the $l$-power roots of unity. 
For a $\mathbb{Z}_l$-basis $\zeta=\{\zeta_{l^n}\}_{n\geqq 0}\in \Gamma(K, \mathbb{Z}_l(1)):=\varprojlim_{n\geqq 0}\mu_{l^n}(K)$, define an additive character 
$$\psi_{\zeta}:\mathbb{Q}_l\rightarrow K^{\times} \text{ by }
\psi_{\zeta}(\frac{1}{l^n})=\zeta_{l^n}$$ for any $n\geqq 0$. By the theory of local constants \cite{De73}, one can attach a constant $$\varepsilon(\rho, \psi, dx)\in K^{\times}$$ to 
any smooth $K$-representation $\rho=(M,\rho)$ of $W_{\mathbb{Q}_l}$ (i.e. $M$ is a finite dimensional $K$-vector space with a $K$-linear smooth action $\rho$ of $W_{\mathbb{Q}_l}$), which depends on the choices of an additive character $\psi:\mathbb{Q}_l^{\times}\rightarrow K^{\times}$ and a ($K$-valued)
Haar measure $dx$ on $\mathbb{Q}_l$. In this article, we consider this constant only for the pair $(\psi_{\zeta},dx)$ such that $\int_{\mathbb{Z}_l}dx=1$, which we denote by $$\varepsilon(\rho,\zeta)
:=\varepsilon(\rho,\psi_{\zeta},dx)$$ for simplicity. For a $K$-representation $M=(M,\rho,N)$ of the Weil-Deligne group $'W_{\mathbb{Q}_l}$ (i.e. $\rho:=(M,\rho)$ is a smooth $K$-representation of $W_{\mathbb{Q}_l}$ with a $K$-linear endomorphism $N:M\rightarrow M$ such that 
$\widetilde{\mathrm{Fr}_l}\circ N=l^{-1}\cdot N\circ\widetilde{\mathrm{Fr}_l}$ for any lift $\widetilde{\mathrm{Fr}_l}\in W_{\mathbb{Q}_l}$ of the geometric 
Frobenius $\mathrm{Fr}_l\in G_{\mathbb{F}_l}$), its $\varepsilon$-constant is defined by 
$$\varepsilon(M,\zeta):=\varepsilon(\rho,\zeta)\cdot \mathrm{det}_K(-\mathrm{Fr}_l|M^{I_l}/(M^{N=0})^{I_l}).$$

Now we recall the definition of de Rham $\varepsilon$-isomorphism 
$\varepsilon^{\mathrm{dR}}_{L,\zeta}(V): \bold{1}_L\isom \Delta_{L}(V)$ for any (resp. de Rham) 
$L$-representation $V$ of $G_{\mathbb{Q}_l}$ when $l\not=p$ (resp. $l=p$). Fix a $\mathbb{Z}_l$-basis $\zeta=\{\zeta_{l^n}\}_{n\geqq 0}\in \Gamma(\overline{\mathbb{Q}}_p, \mathbb{Z}_l(1))$. 
By the Grothendieck's local monodromy theorem (resp. the $p$-adic local monodromy theorem and the Fontaine's functor $\bold{D}_{\mathrm{pst}}(-)$) when 
$l\not=p$ (resp. $l=p$), one can functorially define 
an $L$-representation $$W(V)=(W(V), \rho, N)$$ of $'W_{\mathbb{Q}_l}$. 
Set $L_{\infty}:=L\otimes_{\mathbb{Q}_p}\mathbb{Q}_p(\mu_{l^{\infty}})$, and decompose it $L_{\infty}=\prod_{\tau}L_{\tau}$ into the product of fields $L_{\tau}$. Then, we define a constant $$\varepsilon_{L}(W(V), \zeta)\in L_{\infty}^{\times}$$ as the product of the $\varepsilon$-constants $\varepsilon(W(V)_{\tau},\zeta_{\tau})\in L_{\tau}^{\times}$ of $W(V)_{\tau}
:=W(V)\otimes_LL_{\tau}$ 
for all $\tau$, where $\zeta_{\tau}\in\Gamma(L_{\tau}, \mathbb{Z}_l(1))$ is the image 
of $\zeta$ in $L_{\tau}$. Set $$\bold{D}_{\mathrm{st}}(V):=W(V)^{I_l}, \bold{D}_{\mathrm{cris}}(V):=\bold{D}_{\mathrm{st}}(V)^{N=0}$$ on which the Frobenius $\varphi_l:=\mathrm{Fr}_l$ naturally acts. 
Remark that one has $\bold{D}_{\mathrm{cris}}(V)=V^{I_l}$ if $l\not=p$. Set $$\bold{D}_{\mathrm{dR}}(V):=(\bold{B}_{\mathrm{dR}}\otimes_{\mathbb{Q}_p}V)^{G_{\mathbb{Q}_p}}, \bold{D}^i_{\mathrm{dR}}(V):=(t^i\bold{B}^+_{\mathrm{dR}}\otimes_{\mathbb{Q}_p}V)^{G_{\mathbb{Q}_p}}\text{ and }t_V:=\bold{D}_{\mathrm{dR}}(V)/\bold{D}^0_{\mathrm{dR}}(V)$$ (resp. $\bold{D}_{\mathrm{dR}}(V)=\bold{D}_{\mathrm{dR}}^i(V)=t_V:=0$) when $l=p$ (resp. $l\not=p$). 

Using these preliminaries, we first define an isomorphism 
$$\theta_{L}(V):\bold{1}_L\isom \Delta_{L,1}(V)\boxtimes\mathrm{Det}_L(\bold{D}_{\mathrm{dR}}(V))$$
which is naturally induced by the following exact sequence of $L$-vector spaces
\begin{multline}\label{exp}
0\rightarrow \mathrm{H}^0(\mathbb{Q}_l, V)\rightarrow \bold{D}_{\mathrm{cris}}(V)\xrightarrow{(a)}
\bold{D}_{\mathrm{cris}}(V)\oplus t_V\xrightarrow{(b)}\mathrm{H}^1(\mathbb{Q}_l, V)\\
\xrightarrow{(c)}\bold{D}_{\mathrm{cris}}(V^*)^{\vee}\oplus \bold{D}_{\mathrm{dR}}^0(V)\xrightarrow{(d)}\bold{D}_{\mathrm{cris}}(V^{*})^{\vee}\rightarrow \mathrm{H}^2(\mathbb{Q}_l, V)\rightarrow 0,
\end{multline}
where the map (a) is the sum of $1-\varphi_l:\bold{D}_{\mathrm{cris}}(V)\rightarrow \bold{D}_{\mathrm{cris}}(V)$ and  the canonical map $\bold{D}_{\mathrm{cris}}(V)\rightarrow t_V$, and the maps 
(b) and (c) are defined by using the Bloch-Kato's exponential and its dual when $l=p$, and the map (d) is the dual of (a) for $V^*$ (see \cite{Ka93b}, \cite{FK06} and \cite{Na14b} for the precise definition). Define a constant $\Gamma_L(V)\in \mathbb{Q}^{\times}$ by 
$$\Gamma_L(V):=\begin{cases} 1& (l\not=p) \\
\prod_{r\in\mathbb{Z}}\Gamma^*(r)^{-\mathrm{dim}_L\mathrm{gr}^{-r}\bold{D}_{\mathrm{dR}}(V)}& (l=p),\end{cases}$$
where we set 
$$\Gamma^*(r):=\begin{cases} (r-1)! & (r\geqq 1)\\ \frac{(-1)^r}{(-r)!} & (r\leqq 0).\end{cases}$$

We next define an isomorphism $$\theta_{\mathrm{dR},L}(V,\zeta):\mathrm{Det}_L(\bold{D}_{\mathrm{dR}}(V))\isom \Delta_{L,2}(V)$$ which is induced by the isomorphism 
$$ \mathrm{det}_L\bold{D}_{\mathrm{dR}}(V)=L\isom L_{a_l(V)}:x\mapsto \varepsilon_{L}(W(V), \zeta)\cdot x$$
when $l\not=p$ (remark that one has $\varepsilon_{L}(W(V), \zeta)\in L_{a_l(V)}$ when $l\not=p$), by the inverse of the isomorphism 
$$L_{a_p(T)}\otimes_L\mathrm{det}_LV\isom \mathrm{det}_L\bold{D}_{\mathrm{dR}}(V)(\subseteq \bold{B}_{\mathrm{dR}}\otimes_{\mathbb{Q}_p}\mathrm{det}_LV): x\mapsto \frac{1}{\varepsilon_L(W(V),\zeta)}\cdot\frac{1}{t_{\zeta}^{h_V}}\cdot x$$
when $l=p$, where we set $h_M:=\sum_{r\in \mathbb{Z}}r\cdot \mathrm{dim}_L\mathrm{gr}^{-r}\bold{D}_{\mathrm{dR}}(V)$.
Finally, we define the de Rham $\varepsilon$-isomorphism
$$\varepsilon_{L,\zeta}^{\mathrm{dR}}(V):\bold{1}_L\isom\Delta_L(V)$$ as the following composites
\begin{multline*}
\varepsilon_{L,\zeta}^{\mathrm{dR}}(V):\bold{1}_L\xrightarrow{\Gamma_L(V)\cdot\theta_L(V)}
\Delta_{L,1}(V)\boxtimes \mathrm{Det}_L(\bold{D}_{\mathrm{dR}}(V))\\
\xrightarrow{\mathrm{id}\boxtimes \theta_{\mathrm{dR},L}(V,\zeta)}
\Delta_{L,1}(V)\boxtimes\Delta_{L,2}(V)=\Delta_L(V).
\end{multline*}

The local $\varepsilon$-conjecture (Conjecture 1.8 \cite{Ka93b}, Conjecture 3.4.3 \cite{FK06}, and Conjecture 3.8 \cite{Na14b}) is the following, which is now a theorem when $l\not=p$ by \cite{Ya09}.

\begin{conjecture}\label{2.2}
Fix a prime $l$. Then, there exists a unique compatible family of isomorphisms
$$\varepsilon_{R,\zeta}(T):\bold{1}_R\isom\Delta_R(T)$$ 
for all the triples $(R,T,\zeta)$ such that $T$ is an $R$-representation of $G_{\mathbb{Q}_l}$ and $\zeta$ is a $\mathbb{Z}_l$-basis of $\Gamma(\overline{\mathbb{Q}}_p,\mathbb{Z}_l(1))$, which satisfies the following properties.
\begin{itemize}
\item[(1)]For any continuous $\mathbb{Z}_p$-algebra homomorphism $R\rightarrow R'$, one has
$$\varepsilon_{R,\zeta}(T)\otimes\mathrm{id}_{R'}=\varepsilon_{R',\zeta}(T\otimes_{R}R')$$ under the canonical isomorphism $$\Delta_{R}(T)\otimes_RR'\isom\Delta_{R'}(T\otimes_RR').$$
\item[(2)]For any exact sequence $0\rightarrow T_1\rightarrow T_2\rightarrow T_3\rightarrow 0$ of $R$-representations of $G_{\mathbb{Q}_l}$, one has
$$\varepsilon_{R,\zeta}(T_2)=\varepsilon_{R,\zeta}(T_1)\boxtimes\varepsilon_{R,\zeta}(T_3)$$ under the canonical isomorphism $$\Delta_R(T_2)\isom\Delta_R(T_1)\boxtimes\Delta_R(T_3).$$
\item[(3)]For each $a\in \mathbb{Z}_l^{\times}$, one has
$$\varepsilon_{R,\zeta^a}(T)=\mathrm{det}_RT(\mathrm{rec}_{\mathbb{Q}_l}(a))\cdot \varepsilon_{R,\zeta}(T).$$
\item[(4)]One has the following commutative diagrams 
\begin{equation*}
\begin{CD}
\Delta_{R}(T)@>\mathrm{can}>> \Delta_{R}(T^*)^{\vee} \\
@A \varepsilon_{R,\zeta^{-1}}(T) AA  @ VV\varepsilon_{R,\zeta}(T^*)^{\vee} V \\
\bold{1}_R@> \mathrm{id} >> \bold{1}_R
\end{CD}
\end{equation*}
when $l\not=p$, and 
\begin{equation*}
\begin{CD}
\Delta_{R}(T)@>\mathrm{can}>> \Delta_{R}(T^*)^{\vee}\boxtimes(R(r_T),0) \\
@A \varepsilon_{R,\zeta^{-1}}(T) AA  @ VV\varepsilon_{R,\zeta}(T^*)^{\vee}\boxtimes [\bold{e}_{r_T}\mapsto 1]V \\
\bold{1}_R@> \mathrm{can} >> \bold{1}_R\boxtimes \bold{1}_R
\end{CD}
\end{equation*}
when $l=p$.
\item[(5)]For any triple $(L,V,\zeta)$ such that $V$ is any $($resp. de Rham$)$ if $l\not=p$ $($resp. if $l=p$$)$, one has
$$\varepsilon_{L,\zeta}(V)=\varepsilon_{L,\zeta}^{\mathrm{dR}}(V).$$
\end{itemize}

\end{conjecture}
\begin{rem}\label{2.3}
There exists a non-commutative version of this conjecture, but we only consider the commutative case in this article. See \cite{FK06} for the non-commutative version.

\end{rem}
\begin{rem}\label{Ya}
When $l\not=p$, this conjecture has been already proved by Yasuda \cite{Ya09}. More precisely,  he proved that the correspondence $$(L,V,\zeta)\mapsto \varepsilon_{0,L}(V,\zeta):=\mathrm{det}_L(-\varphi_l|V^{I_l})\cdot \varepsilon_{L}(W(V), \zeta)\in L_{a_l(V)}$$ defined for all the triples $(L,V,\zeta)$ as in the condition (5) (for $l\not=p$) in Conjecture \ref{2.2} uniquely  extends to a correspondence $$(R,T,\zeta)\mapsto \varepsilon_{0,R}(T,\zeta)\in R_{a_l(T)}$$ for all the triples $(R,T,\zeta)$ as in the conjecture, which satisfies the similar properties (1)-(5) in the conjecture. Then, the isomorphism $\varepsilon_{R,\zeta}(T):\bold{1}_R\isom \Delta_R(T)$ is defined as the product of the isomorphism 
$\bold{1}_R\isom\Delta_{R,2}(T)$ induced by the isomorphism $$R\isom R_{a_l(T)}:x\mapsto \varepsilon_{0,R}(T,\zeta)\cdot x$$ with the isomorphism $\bold{1}_R\isom \Delta_{R,1}(T)$ defined by 
$$\bold{1}_R\isom \mathrm{Det}_R(C_{\mathrm{cont}}^{\bullet}(I_l, T))\boxtimes\mathrm{Det}_R(C_{\mathrm{cont}}^{\bullet}(I_l, T))^{-1}\isom 
\mathrm{Det}_R(C_{\mathrm{cont}}^{\bullet}(G_{\mathbb{Q}_l}, T)), $$ 
where the first isomorphism is the canonical one and the second isomorphism is induced by the canonical quasi-isomorphism 
$$C_{\mathrm{cont}}^{\bullet}(G_{\mathbb{Q}_l}, T)\isom [C_{\mathrm{cont}}^{\bullet}(I_l, T)\xrightarrow{1-\varphi_l^{-1}}C_{\mathrm{cont}}^{\bullet}(I_l, T)].$$

\end{rem}

Before proceeding to the next subsection, we recall here the notion of the cyclotomic deformations of $R$-representations, which will play an important role in this article. 
For any $R$ such that $p\not\in R^{\times}$, we set $\Lambda_R(\Gamma):=R[[\Gamma]]$ the Iwasawa algebra of $\Gamma$ with coefficients in $R$, and set $\Lambda_L(\Gamma):=\Lambda_{\mathcal{O}_L}(\Gamma)[1/p]$ for any $L$. For an $R$-representation $T$ of $G_{\mathbb{Q}_l}$, we define a $\Lambda_R(\Gamma)$-representation $\bold{Dfm}(T)$ which we call the cyclotomic deformation of $T$ by 
$$\bold{Dfm}(T):=T\otimes_R\Lambda_R(\Gamma)$$ on which $G_{\mathbb{Q}_l}$ acts by 
$$g(x\otimes \lambda):=g(x)\otimes [\bar{g}]^{-1}\cdot\lambda$$ for $g\in G_{\mathbb{Q}_l}$, $x\in T$, $\lambda\in \Lambda_R(\Gamma)$. We set
$$\Delta_{R,*}^{\mathrm{Iw}}(T):=\Delta_{\Lambda_R(\Gamma),*}(\bold{Dfm}(T)) \text{ and } \mathrm{H}^i_{\mathrm{Iw}}(\mathbb{Q}_l, T):=\mathrm{H}^i(\mathbb{Q}_l, \bold{Dfm}(T))$$
for for $*=1,2,$ or $*=\phi$ (the empty set).
For a continuous homomorphism $\delta:\Gamma\rightarrow R^{\times}$, define a continuous $R$-algebra homomorphism $$f_{\delta}:\Lambda_R(\Gamma)\rightarrow R:[\gamma]\mapsto \delta(\gamma)^{-1}$$ for any $\gamma\in \Gamma$. Then, one has a canonical isomorphism of $R$-representations of $G_{\mathbb{Q}_l}$
$$\bold{Dfm}(T)\otimes_{\Lambda_R(\Gamma),f_{\delta}}R\isom T(\delta):(x\otimes \lambda)\otimes a\mapsto a\cdot f_{\delta}(\lambda)\cdot x\otimes\bold{e}_{\delta}$$ for $x\in T, \lambda\in \Lambda_R(\Gamma), a\in R$. By the compatibility with base changes, this isomorphism induces a canonical isomorphism $$\Delta_{R}^{\mathrm{Iw}}(T)\otimes_{\Lambda_R(\Gamma),f_{\delta}}R\isom \Delta_R(T(\delta)).$$ 
Let $\iota:\Lambda_R(\Gamma)\isom \Lambda_R(\Gamma)$ be the involution of the topological 
$R$-algebra defined by 
$\iota([\gamma])=[\gamma]^{-1}$ for any $\gamma\in \Gamma$. For any $\Lambda_R(\Gamma)$-module $M$, we set $M^{\iota}:=M\otimes_{\Lambda_R(\Gamma),\iota}\Lambda_R(\Gamma)$, i.e. $M^{\iota}=M$ as $R$-module on which $\Lambda_R(\Gamma)$ acts by $\lambda\cdot_{\iota} x:=\iota(\lambda)\cdot x$ for $\lambda\in \Lambda_R(\Gamma), x\in M$, where $\cdot_{\iota}$ is the action on $M^{\iota}$ and $\cdot$ is the usual action on $M$. One has a canonical isomorphism of 
$\Lambda_R(\Gamma)$-representations of $G_{\mathbb{Q}_l}$
$$\bold{Dfm}(T^*)^{\iota}\isom \bold{Dfm}(T)^*:x\otimes \lambda\mapsto [y\otimes \lambda'\mapsto \iota(\lambda)\cdot \lambda'\otimes x(y)]$$
for $x\in T^*$, $y\in T$, $\lambda,\lambda'\in \Lambda_R(\Gamma)$, which naturally induces a canonical isomorphism 
\begin{equation}\label{dual}
\Delta^{\mathrm{Iw}}_{R,*}(T^*)^{\iota}\isom \Delta_{\Lambda_R(\Gamma),*}(\bold{Dfm}(T)^*)
\end{equation}
for $*=1,2,$ or $*=\phi$.


\subsection{Review of the theory of \'etale $(\varphi,\Gamma)$-modules}
From now on until the end of $\S 3$, we concentrate on the case where $l=p$, and fix a basis $\bold{e}_1:=\zeta\in \mathbb{Z}_p(1):=\Gamma(\overline{\mathbb{Q}}_p, \mathbb{Z}_p(1))$. Set $\bold{e}_r:=\bold{e}_1^{\otimes r}\in \mathbb{Z}_p(r)$ for $r\in \mathbb{Z}$. 
For $R$ as in $\S2.1$ such that $p\not \in R^{\times}$, we set $\mathcal{E}_R:=\varprojlim_{n\geqq 1}(R/\mathrm{Jac}(R)^n[[X]][1/X])$, and set $\mathcal{E}_L:=\mathcal{E}_{\mathcal{O}_L}[1/p]$ for $R=L$, on which $\varphi$ and $\Gamma$ acts as continuous $R$-algebra homomorphism by $\varphi(X):=(1+X)^p-1$, $\gamma(X):=(1+X)^{\chi(\gamma)}-1$ for $\gamma\in \Gamma$. 

For $R$ such that $p\not\in R^{\times}$, we say that $D$ is an \'etale $(\varphi,\Gamma)$-module over $\mathcal{E}_R$ if $D$ is a finite projective $\mathcal{E}_R$-module equipped with a Frobenius structure $\varphi:\varphi^*D:=D\otimes_{\mathcal{E}_R,\varphi}\mathcal{E}_R\isom D$ and a commuting continuous semi-linear action $\Gamma\times D\rightarrow D:(\gamma,x)\mapsto \gamma(x)$. For $R=L$, we say that $D$ is an \'etale $(\varphi,\Gamma)$-module over $\mathcal{E}_L$ if $D$ is the base change to $\mathcal{E}_L$ of an \'etale $(\varphi,\Gamma)$-module over $\mathcal{E}_{\mathcal{O}_L}$. 
We denote by $D^{\vee}:=\mathrm{Hom}_{\mathcal{E}_R}(D, \mathcal{E}_R)$ the dual $(\varphi,\Gamma)$-module of $D$, by $D(r):=D\otimes_{\mathbb{Z}_p}\mathbb{Z}_p(r)$ the $r$-th Tate twist of $D$ (for $r\in \mathbb{Z}$), by $D^*:=D^{\vee}(1)$ the Tate dual of $D$. One has the Fontaine's equivalence $T\mapsto D(T)$ between the category of $R$-representations of $G_{\mathbb{Q}_p}$ and that of \'etale $(\varphi,\Gamma)$-modules over $\mathcal{E}_R$. In the construction of 
this equivalence, we need to embed the ring $\mathcal{E}_{\mathbb{Z}_p}$ ($\mathcal{E}_R$ for $R=\mathbb{Z}_p$) into the Fontaine ring $\widetilde{\bold{A}}^+(:=W(\widetilde{\bold{E}}^+))$ by 
$X\mapsto X_{\zeta}:=[(\bar{\zeta}_{p^n})_{n\geqq 0}]-1\in \widetilde{\bold{A}}^+$, which depends on the choice of the fixed basis $\zeta=(\zeta_{p^n})_{n\geqq 0}\in \mathbb{Z}_p(1)$. We remark that, for a different choice $\zeta^a$ of a basis for $a\in \mathbb{Z}_p^{\times}$, one has $X_{\zeta^a}=
(X_{\zeta}+1)^a-1$. Define a left inverse $\psi:D\rightarrow D$ of $\varphi$ by 
$$\psi:D=\sum_{i=1}^{p-1}(1+X)^i\varphi(D)\rightarrow D:\sum_{i=0}^{p-1}(1+X)^i\varphi(x_i)\mapsto x_0.$$ 

We next recall the cohomology theory of $(\varphi,\Gamma)$-modules. 
Let $\Gamma_{\mathrm{tor}}\subseteq \Gamma$ denote the torsion subgroup of $\Gamma$. 
Define a finite subgroup $\Delta\subseteq \Gamma_{\mathrm{tor}}$ by 
$\Delta:=\{1\}$ when $p>2$ and $\Delta:=\Gamma_{\mathrm{tor}}$ when $p=2$. 
Then $\Gamma/\Delta$ has a topological generator $\bar{\gamma}$,  which we fix. We also fix a lift $\gamma\in \Gamma$ of 
$\bar{\gamma}$.

\begin{rem}\label{2.6}
When $p=2$ and $p\not\in R^{\times}$, the cohomology theory of $(\varphi,\Gamma)$-modules over 
$\mathcal{E}_R$ is a little more subtle than that in other cases since one has $|\Gamma_{\mathrm{tor}}|=p$ in this case. 
To avoid this subtlety, we treat $(\varphi,\Gamma)$-modules over the rings of the form $R=R_0[1/p]$, where $R_0$ is a topological $\mathbb{Z}_p$-algebra satisfying the condition (i) in \S2.1. 
For such $R$, we set $\mathcal{E}_{R}:=\mathcal{E}_{R_0}[1/p]$, and say that an $\mathcal{E}_{R}$-module $D$ is an \'etale $(\varphi,\Gamma)$-module over $\mathcal{E}_{R}$ if it is the base change to $\mathcal{E}_R$ of an \'etale $(\varphi,\Gamma)$-module over $\mathcal{E}_{R_0}$. 
From now on until the end of this article, we use the notation $R$ to represent topological $\mathbb{Z}_p$-algebras of the form $R_0$ or $R_0[1/p]$ (resp. $R_0[1/2]$) as above when 
$p\geqq 3$ (resp. $p=2$), and we only consider the $R$-representations of $G_{\mathbb{Q}_l}$ or $G_{\mathbb{Q},S}$, and \'etale $(\varphi,\Gamma)$-modules over $\mathcal{E}_{R}$ for such $R$. 

\end{rem}
\begin{defn}\label{2.11}
For  any \'etale $(\varphi,\Gamma)$-module $D$ over $\mathcal{E}_{R}$, define complexes 
$C^{\bullet}_{\varphi,\gamma}(D)$ and $C^{\bullet}_{\psi,\gamma}(D)$ of $R$-modules
concentrated in degree $[0,2]$, and define a morphism $\Psi_D$ between them as follows:
\begin{equation}\label{a}
\begin{CD}
C^{\bullet}_{\varphi,\gamma}(D)@.= [D^{\Delta} @> (\gamma-1,\varphi-1)>> 
D^{\Delta}\oplus D^{\Delta} @> (\varphi-1)\oplus(1-\gamma) >> D^{\Delta}] \\
@V \Psi_D VV  @ VV\mathrm{id} V @ VV \mathrm{id}\oplus -\psi V @ VV -\psi V \\
C^{\bullet}_{\psi,\gamma}(D)@.= [D^{\Delta} @> (\gamma-1,\psi-1)>> 
D^{\Delta}\oplus D^{\Delta} @> (\psi-1)\oplus(1-\gamma) >> D^{\Delta}].
\end{CD}
\end{equation}
\end{defn}

For $i\in \mathbb{Z}_{\geqq 0}$, we denote by $\mathrm{H}^i_{\varphi,\gamma}(D)$ (resp. $\mathrm{H}^i_{\psi,\gamma}(D)$) the $i$-th cohomology of 
$C^{\bullet}_{\varphi,\gamma}(D)$ (resp.$C^{\bullet}_{\psi,\gamma}(D)$). It is known that 
the map $\Psi_D:C^{\bullet}_{\varphi,\gamma}(D)\rightarrow C^{\bullet}_{\psi,\gamma}(D)$ is quasi-isomorphism by (for example) Proposition I.5.1 and Lemme I.5.2 of \cite{CC99}. In this article, we freely identify 
$C^{\bullet}_{\varphi,\gamma}(D)$  (resp. $\mathrm{H}^i_{\varphi,\gamma}(D)$) 
with $C^{\bullet}_{\psi,\gamma}(D)$ in $\bold{D}^-(R)$ (resp. $\mathrm{H}^i_{\psi,\gamma}(D)$) via the 
quasi-isomorphism $\Psi_D$.

For \'etale $(\varphi,\Gamma)$-modules $D_1, D_2$ over $\mathcal{E}_R$, one has an $R$-bilinear cup product pairing 
$$C^{\bullet}_{\varphi,\gamma}(D_1)\times C^{\bullet}_{\varphi,\gamma}(D_2)
\rightarrow C^{\bullet}_{\varphi,\gamma}(D_1\otimes D_2),$$
which induces the cup product pairing 
$$\cup:\mathrm{H}^i_{\varphi,\gamma}(D_1)
\times \mathrm{H}^{j}_{\varphi,\gamma}(D_2)\rightarrow \mathrm{H}^{i+j}_{\varphi,\gamma}(D_1\otimes D_2).$$ 
For example, this pairing is explicitly defined by the formulae
$$x\cup [y]:=[x\otimes y] \text{ for } i=0, j=2,$$
$$[x_1,y_1]\cup [x_2,y_2]:=[x_1\otimes \gamma(y_2)-y_1\otimes \varphi(x_2)] \text{ for } i=j=1.$$

\begin{defn}\label{dual}
 Using the cup product,
the evaluation map $\mathrm{ev}:D^*\otimes D\rightarrow \mathcal{E}_R(1):f\otimes x \mapsto f(x)$, the comparison isomorphism $\mathrm{H}^{2}(\mathbb{Q}_p, R(1))
\isom \mathrm{H}^2_{\varphi,\gamma}(\mathcal{E}_R(1))$ (see below), and the Tate's trace map $\mathrm{H}^2(\mathbb{Q}_p, R(1))\isom R$, one gets the Tate duality pairings 
$$C^{\bullet}_{\varphi,\gamma}(D^*)\times C^{\bullet}_{\varphi,\gamma}(D)
\rightarrow  R[-2]$$
and 
$$\langle-, -\rangle_{\mathrm{Tate}}:\mathrm{H}^i_{\varphi,\gamma}(D^*)\times \mathrm{H}^{2-i}_{\varphi,\gamma}(D)\rightarrow R.$$
\end{defn}


 Let $T$ be an $R$-representation of $G_{\mathbb{Q}_p}$. 
 By the result of \cite{He98}, one has a canonical functorial isomorphism 
$$C^{\bullet}_{\mathrm{cont}}(G_{\mathbb{Q}_p}, T)\isom C^{\bullet}_{\varphi,\gamma}(D(T))$$ 
in $\bold{D}^-(R)$ and a canonical functorial $R$-linear isomorphism 
$$\mathrm{H}^i(\mathbb{Q}_p, T)\isom \mathrm{H}^i_{\varphi,\gamma}(D(T)).$$ 
 In particular, we obtain a canonical isomorphism 
$$\Delta_{R,1}(T)\isom\mathrm{Det}_R(C^{\bullet}_{\varphi,\gamma}(D(T)))=:\Delta_{R,1}(D(T)).$$

For an \'etale $(\varphi,\Gamma)$-module $D$. We freely regard the rank one $(\varphi,\Gamma)$-module $\mathrm{det}_{\mathcal{E}_R}D$ as a character $\mathrm{det}_{\mathcal{E}_R}D:G_{\mathbb{Q}_p}^{\mathrm{ab}}\rightarrow R^{\times}$ by the Fontaine's equivalence. Then, the $(\varphi,\Gamma)$-module $\mathrm{det}_{\mathcal{E}_R}D$ has a basis $\bold{e}$ on which $\varphi$ and $\Gamma$ act by 
$$\varphi(\bold{e})=\mathrm{det}_{\mathcal{E}_R}D(\mathrm{rec}_{\mathbb{Q}_p}(p))\cdot \bold{e},\,\, \gamma'(\bold{e})=\mathrm{det}_{\mathcal{E}_R}D(\gamma')\cdot \bold{e}$$ for $\gamma'\in \Gamma$, where we regard $\Gamma$ as a subgroup of $G_{\mathbb{Q}_p}^{\mathrm{ab}}$ by the canonical isomorphism 
$\mathrm{Gal}(\mathbb{Q}_p^{\mathrm{ab}}/\mathbb{Q}_p^{\mathrm{ur}})\isom \Gamma$. Using the character $\mathrm{det}_{\mathcal{E}_R}D$, we define 
$$\mathcal{L}_R(D):=\{x\in \mathrm{det}_{\mathcal{E}_R}D|\varphi(x)=\mathrm{det}_{\mathcal{E}_R}D(\mathrm{rec}_{\mathbb{Q}_p}(p))\cdot x, \gamma'(x)=\mathrm{det}_{\mathcal{E}_R}D(\gamma')\cdot x \text{ for  }\gamma'\in \Gamma\},$$
which is a free $R$-module of rank one, 
and define the following graded invertible $R$-modules 
$$\Delta_{R,2}(D):=(\mathcal{L}_R(D), r_D)\text{ and }\Delta_R(D):=\Delta_{R,1}(D)\boxtimes \Delta_{R,2}(D).$$
By (the proof of) Lemma 3.1 of \cite{Na14b}, there exists a canonical isomorphism 
$$\Delta_{R,2}(T)\isom \Delta_{R,2}(D(T))$$ 
for any $R$-representation $T$ of $G_{\mathbb{Q}_p}$. Therefore, 
we obtain a canonical isomorphism 
$$\Delta_R(T)\isom \Delta_R(D(T)),$$
by which we identify the both sides. 

We next recall the theory of the Iwasawa cohomology of $(\varphi,\Gamma)$-modules. 
For an \'etale $(\varphi,\Gamma)$-module $D$ over $\mathcal{E}_R$, we define the cyclotomic deformation 
$\bold{Dfm}(D)$ which is an \'etale $(\varphi,\Gamma)$-module over 
$\mathcal{E}_{\Lambda_R(\Gamma)}$ by 
$$\bold{Dfm}(D):=D\otimes_{\mathcal{E}_R}\mathcal{E}_{\Lambda_R(\Gamma)}$$ as $\mathcal{E}_{\Lambda_R(\Gamma)}$-module on which $\varphi$ and $\Gamma$ act by 
$$\varphi(x\otimes y):=\varphi(x)\otimes \varphi(y),\,\, \gamma'(x\otimes y):=\gamma'(x)\otimes [\gamma']^{-1}\cdot \gamma'(y)$$ for $x\in D, y\in \mathcal{E}_{\Lambda_R(\Gamma)}$, $\gamma'\in \Gamma$. Then, one has a canonical isomorphism $$D(\bold{Dfm}(T))\isom 
\bold{Dfm}(D(T)).$$ Hence, if we set 
$$\mathrm{H}^i_{\mathrm{Iw}, \varphi,\gamma}(D):=\mathrm{H}^i_{\varphi,\gamma}(\bold{Dfm}(D)) \text{ and }\Delta_{R}^{\mathrm{Iw}}(D):=\Delta_{\Lambda_R(\Gamma)}(\bold{Dfm}(D)), \text{ etc}.,$$ then we obtain the following canonical isomorphisms 
$$\mathrm{H}^i_{\mathrm{Iw}}(\mathbb{Q}_p, T)\isom \mathrm{H}^i_{\mathrm{Iw},\varphi,\gamma}(D(T))\text{ and } \Delta_R^{\mathrm{Iw}}(T)\isom \Delta_R^{\mathrm{Iw}}(D(T)), \text{ etc. }$$ for any $R$-representation $T$ of 
$G_{\mathbb{Q}_p}$. For any continuous homomorphism $\delta:\Gamma\rightarrow R^{\times}$, the base change with respect to $f_{\delta}:\Lambda_R(\Gamma)\rightarrow R:[\gamma']\mapsto \delta(\gamma')^{-1}$  induces canonical isomorphisms 
$$\bold{Dfm}(D)\otimes_{\Lambda_R(\Gamma),f_{\delta}}R\isom 
D(\delta):(x\otimes y)\otimes 1\mapsto f_{\delta}(y)x\otimes\bold{e}_{\delta}$$ and 
$$\Delta_R^{\mathrm{Iw}}(D)\otimes_{\Lambda_R(\Gamma),f_{\delta}}R\isom \Delta_R(D(\delta)),$$ and induces a canonical specialization map
$$\mathrm{sp}_{\delta}:\mathrm{H}^i_{\mathrm{Iw},\varphi,\gamma}(D)\rightarrow \mathrm{H}^i_{\varphi,\gamma}(D(\delta)).$$ 

We remark that the continuous action of $\Gamma$ on $D$ uniquely extends to a $\Lambda_R(\Gamma)$-module structure on $D$. We define a complex $C^{\bullet}_{\psi}(D)$ of $\Lambda_R(\Gamma)$-modules which concentrated in degree $[1,2]$ by 
$$C^{\bullet}_{\psi}(D):[D\xrightarrow{\psi-1}D].$$
By the result of \cite{CC99}, there exists a canonical isomorphism
$$C^{\bullet}_{\psi}(D)\isom C^{\bullet}_{\psi,\gamma}(\bold{Dfm}(D))$$
in $\bold{D}^-(\Lambda_R(\Gamma))$. In particular, there exists a canonical isomorphism 
$$\mathrm{can}:D^{\psi=1}\isom \mathrm{H}^1_{\mathrm{Iw}, \psi,\gamma}(D)$$ of $\Lambda_R(\Gamma)$-modules which is explicitly defined by 
$$x\mapsto \left[\left(\frac{p-1}{p}\cdot \mathrm{log}(\chi(\gamma))p_{\Delta}(x\otimes 1), 0\right)\right],$$ where $p_{\Delta}:=\frac{1}{|\Delta|}\sum_{\sigma\in \Delta} [\sigma]\in \mathbb{Z}[1/2][\Delta]$ (remark that we have $\frac{p-1}{p}\cdot \mathrm{log}(\chi(\gamma))\cdot p_{\Delta}\in \mathbb{Z}_p[\Delta]$ for any $p$). Hence, if we define a specialization map
$$\iota_{\delta}:D^{\psi=1}\rightarrow \mathrm{H}^1_{\psi,\gamma}(D(\delta)):x\mapsto x_{\delta}:=\left[\left(\frac{p-1}{p}\cdot \mathrm{log}(\chi(\gamma))\cdot p_{\Delta}(x\otimes\bold{e}_{\delta}),0\right)\right]$$for any continuous homomorphism 
$\delta:\Gamma\rightarrow R^{\times}$,  then it makes the 
diagram
\begin{equation}
\begin{CD}
D^{\psi=1}@> \mathrm{can} >> \mathrm{H}^1_{\mathrm{Iw},\psi,\gamma}(D)\\
@VV \iota_{\delta} V  @ VV \mathrm{sp}_{\delta} V \\
\mathrm{H}^1_{\psi,\gamma}(D(\delta)) @> \mathrm{id} > >
\mathrm{H}^1_{\psi,\gamma}(D(\delta))
\end{CD}
\end{equation}
commutative.

We next consider the Tate dual of $\bold{Dfm}(D)$. For this, we first remark that 
the involution $\iota:\Lambda_R(\Gamma)\isom \Lambda_R(\Gamma):[\gamma']\mapsto [\gamma']^{-1}$ naturally induces an $\mathcal{E}_R$-linear  involution $\iota:\mathcal{E}_{\Lambda_R(\Gamma)}\isom 
\mathcal{E}_{\Lambda_R(\Gamma)}$. For an \'etale $(\varphi,\Gamma)$-module $D$ over $\mathcal{E}_R$, define an \'etale 
$(\varphi,\Gamma)$-module $D\otimes_{\mathcal{E}_R}\widetilde{\mathcal{E}_{\Lambda_R(\Gamma)}}$ over $\mathcal{E}_{\Lambda_R(\Gamma)}$ by
$$D\otimes_{\mathcal{E}_R}\widetilde{\mathcal{E}_{\Lambda_R(\Gamma)}}=D\otimes_{\mathcal{E}_R}\mathcal{E}_{\Lambda_R(\Gamma)}$$ as $\mathcal{E}_{\Lambda_R(\Gamma)}$-module on which 
$\varphi$ and $\Gamma$ act by $$\varphi(x\otimes y)=\varphi(x)\otimes\varphi(y)\text{ and 
 }\gamma'(x\otimes y)=\gamma'(x)\otimes [\gamma']\cdot \gamma'(y)$$ for $x\in D$, $y\in \mathcal{E}_{\Lambda_R(\Gamma)}, \gamma'\in \Gamma$. Then, the isomorphism 
 $$D\otimes_{\mathcal{E}_R}\mathcal{E}_{\Lambda_R(\Gamma)}\isom D\otimes_{\mathcal{E}_R}\mathcal{E}_{\Lambda_R(\Gamma)}:x\otimes y\mapsto x\otimes \iota(y)$$ induces an isomorphism 
 $$D\otimes_{\mathcal{E}_R}\widetilde{\mathcal{E}_{\Lambda_R(\Gamma)}}\isom \bold{Dfm}(D)^{\iota}$$ of $(\varphi,\Gamma)$-modules over $\mathcal{E}_{\Lambda_R(\Gamma)}$. Since one has a canonical isomorphism 
$$(D\otimes_{\mathcal{E}_R}\widetilde{\mathcal{E}_{\Lambda_R(\Gamma)}})\otimes_{\Lambda_R(\Gamma),f_{\delta}}R\isom D(\delta^{-1}):(x\otimes y)\otimes 1\mapsto f_{\delta}(y)x\otimes \bold{e}_{\delta^{-1}}$$ for any $\delta:\Gamma\rightarrow R^{\times}$, we obtain a canonical specialization map 
\begin{equation*}
\tilde{\iota}_{\delta}:D^{\psi=1,\iota}\isom \mathrm{H}^1_{\psi,\gamma}(\bold{Dfm}(D)^{\iota})
\isom \mathrm{H}^1_{\psi,\gamma}(D\otimes_{\mathcal{E}_R}\widetilde{\mathcal{E}_{\Lambda_R(\Gamma)}})\rightarrow \mathrm{H}^1_{\psi,\gamma}(D(\delta^{-1}))
\end{equation*}
which is explicitly defined by 
$$x\mapsto \left[\left(\frac{p-1}{p}\cdot \mathrm{log}(\chi(\gamma))\cdot p_{\Delta}(x\otimes\bold{e}_{\delta^{-1}}),0\right)\right].$$

We apply this to the Tate dual $D^*$ of $D$. Since one has a canonical isomorphism 
$$\bold{Dfm}(D^*)^{\iota}\isom D^*\otimes_{\mathcal{E}_R}\widetilde{\mathcal{E}_{\Lambda_R(\Gamma)}}
\isom \bold{Dfm}(D)^{*},$$ we obtain canonical isomorphisms
$$\Delta_{R}^{\mathrm{Iw}}(D^*)^{\iota}\isom \Delta_{\Lambda_R(\Gamma)}(\bold{Dfm}(D)^*)$$
 and 
$$\mathrm{can}:(D^*)^{\psi=1,\iota}\isom \mathrm{H}^1_{\psi,\gamma}(\bold{Dfm}(D^*)^{\iota})\isom 
\mathrm{H}^1_{\psi,\gamma}(\bold{Dfm}(D)^*),$$ which makes
 the diagram 

\begin{equation}
\begin{CD}
(D^*)^{\psi=1,\iota}@> \mathrm{can} >> \mathrm{H}^1_{\psi,\gamma}(\bold{Dfm}(D)^*)\\
@VV \tilde{\iota}_{\delta} V  @ VV \mathrm{sp}_{\delta} V \\
\mathrm{H}^1_{\psi,\gamma}(D(\delta)^*) @> \mathrm{id} > >
\mathrm{H}^1_{\psi,\gamma}(D(\delta)^{*})
\end{CD}
\end{equation}
for any $\delta:\Gamma\rightarrow R^{\times}$ commutative, where the right vertical arrow is the specialization map with respect to the base change 
$$\bold{Dfm}(D)^{*}\otimes_{\Lambda_R(\Gamma),f_{\delta}}R\isom (\bold{Dfm}(D)\otimes_{\Lambda_R(\Gamma),f_{\delta}}R)^*\isom D(\delta)^{*}.$$ 

Using these preliminaries, we define a $\Lambda_R(\Gamma)$-bilinear pairing 
$$\{-,-\}_{\mathrm{Iw}}:(D^*)^{\psi=1,\iota}\times D^{\psi=1}\rightarrow \Lambda_R(\Gamma)$$
which we call the Iwasawa pairing by the following commutative diagram 
\begin{equation}
\begin{CD}
 (D^*)^{\psi=1,\iota}\times D^{\psi=1}@> \mathrm{can}\times \mathrm{can} >> \mathrm{H}^1_{\psi,\gamma}(\bold{Dfm}(D)^*)\times \mathrm{H}^1_{\psi,\gamma}(\bold{Dfm}(D))\\
@VV\{-,-\}_{\mathrm{Iw}}  V  @ VV \langle-,-\rangle_{\mathrm{Tate}}  V \\
\Lambda_R(\Gamma)@> \mathrm{id}> >
\Lambda_R(\Gamma).
\end{CD}
\end{equation}
 From the arguments above, we obtain the commutative diagram
\begin{equation}\label{CD6}
\begin{CD}
(D^*)^{\psi=1,\iota}\times D^{\psi=1}@> \tilde{\iota}_{\delta}\times \iota_{\delta} >>\mathrm{H}^1_{\psi,\gamma}(D(\delta)^*)\times \mathrm{H}^1_{\psi,\gamma}(D(\delta))\\
@VV\{-,-\}_{\mathrm{Iw}}  V  @ VV \langle-,-\rangle_{\mathrm{Tate}}  V \\
\Lambda_R(\Gamma)@> f_{\delta}> >
R
\end{CD}
\end{equation}
for any $\delta:\Gamma\rightarrow R^{\times}$.
\begin{rem}\label{2.9.1}
We remark that the pairing $\{-,-\}_{\mathrm{Iw}}$ coincides with the Colmez's Iwasawa pairing 
which is defined in \S VI.1 of \cite{Co10a} in a different way.
\end{rem}

 For a continuous homomorphism $\delta:\Gamma\rightarrow R^{\times}$, we define a continuous 
 $R$-algebra automorphism 
 $g_{\delta}:\Lambda_R(\Gamma)\isom \Lambda_R(\Gamma)$ by $g_{\delta}([\gamma'])=\delta(\gamma')^{-1}\cdot[\gamma']$ for $\gamma'\in \Gamma$.
 \begin{lemma}\label{2.10.1}For any $\delta:\Gamma\rightarrow R^{\times}$, 
 one has the following commutative diagram
  \begin{equation}
\begin{CD}
(D^*)^{\psi=1,\iota}\times D^{\psi=1}@>  \{-,-\}_{\mathrm{Iw}}>> \Lambda_R(\Gamma) 
\\
@V (x,y)\mapsto (x\otimes \bold{e}_{\delta^{-1}}, y\otimes \bold{e}_{\delta}) VV  @ VV g_{\delta} V \\
(D(\delta)^*)^{\psi=1,\iota}\times D(\delta)^{\psi=1}@>  \{-,-\}_{\mathrm{Iw}}>> \Lambda_R(\Gamma).
\end{CD}
\end{equation}

 \end{lemma}
 \begin{proof}
For a $\Lambda_R(\Gamma)$-module $M$, we define a $\Lambda_R(\Gamma)$-module $g_{\delta,*}(M):=M$ on which $\Lambda_R(\Gamma)$-acts by $g_{\delta}$. Then, 
we have isomorphisms $\bold{Dfm}(D)\isom g_{\delta,*}(\bold{Dfm}(D(\delta))):x\otimes y\mapsto 
(x\otimes \bold{e}_{\delta})\otimes g_{\delta}(y)$ and $\bold{Dfm}(D^*)^{\iota}\isom g_{\delta,*}(\bold{Dfm}(D(\delta)^*)^{\iota})
:x\otimes y\mapsto (x\otimes\bold{e}_{\delta^{-1}})\otimes g_{\delta^{-1}}(y)$, and these induce the following commutative diagram 

\begin{equation}
\begin{CD}
 \mathrm{H}^1_{\psi,\gamma}(\bold{Dfm}(D)^*)\times \mathrm{H}^1_{\psi,\gamma}(\bold{Dfm}(D))@> \isom >> g_{\delta, *}(\mathrm{H}^1_{\psi,\gamma}(\bold{Dfm}(D(\delta)^*)))\times g_{\delta,*}(\mathrm{H}^1_{\psi,\gamma}(\bold{Dfm}(D(\delta))))\\
@VV\langle-,-\rangle_{\mathrm{Tate}} V  @ VV \langle-,-\rangle_{\mathrm{Tate}}  V \\
\Lambda_R(\Gamma)@> g_{\delta}> >
g_{\delta, *}(\Lambda_R(\Gamma)).
\end{CD}
\end{equation}
By definition of $\{-,-\}_{\mathrm{Iw}}$, the lemma follows from this commutative diagram.

 \end{proof}
 
 Take an isomorphism $\Gamma\isom \Gamma_{\mathrm{tor}}\times \mathbb{Z}_p$ of topological groups. Let $\gamma_0\in \Gamma$ be the element corresponding to $(e,1)$ by this isomorphism, 
 where $e\in \Gamma_{\mathrm{tor}}$ is the identity element. 
 For $R$ such that $p\not\in R^{\times}$, define $\mathcal{E}_R(\Gamma):=\Lambda_R(\Gamma)[
 \frac{1}{[\gamma_0]-1}]^{\wedge}$ the $\mathrm{Jac}(R)$-adic completion of $\Lambda_R(\Gamma)[
 \frac{1}{[\gamma_0]-1}]$, which does not depend on the choice of the decomposition 
 $\Gamma\isom \Gamma_{\mathrm{tor}}\times \mathbb{Z}_p$. For $R=R_0[1/p]$ such that $p\not\in R_0^{\times}$, define $\mathcal{E}_R(\Gamma):=\mathcal{E}_{R_0}(\Gamma)[1/p]$. Here, we recall some properties of the base changes to $\mathcal{E}_R(\Gamma)$ of $\Delta_R(D)$ and $\{-,-\}_{\mathrm{Iw}}$, which are proved in \cite{Co10a}. By III.4 of \cite{Co10a}, $\gamma_0-1$ acts on $D^{\psi=0}$ as a topological automorphism and the induced action of $\Lambda_R(\Gamma)[\frac{1}{[\gamma_0]-1}]$ on $D^{\psi=0}$ uniquely extends to an action of $\mathcal{E}_R(\Gamma)$, which makes $D^{\psi=0}$ a 
 finite projective $\mathcal{E}_R(\Gamma)$-module of rank $r_D$. By VI.1 of \cite{Co10a}, the $\Lambda_R(\Gamma)$-linear homomorphism $D^{\psi=1}\xrightarrow{1-\varphi} D^{\psi=0}$ induces an isomorphism 
 $$D^{\psi=1}\otimes_{\Lambda_R(\Gamma)}\mathcal{E}_R(\Gamma)\isom D^{\psi=0}$$ of 
 $\mathcal{E}_R(\Gamma)$-modules, and one has
 $$D^{\varphi=1}\otimes_{\Lambda_R(\Gamma)}\mathcal{E}_R(\Gamma)=(D/(\psi-1)D)\otimes_{\Lambda_R(\Gamma)}\mathcal{E}_R(\Gamma)=0$$ (since $D^{\varphi=1}$ and $D/(\psi-1)D$ are finite generated $R$-modules). In particular, we obtain a canonical isomorphism
  $$C^{\bullet}_{\psi}(D)\otimes^{\bold{L}}_{\Lambda_R(\Gamma)}\mathcal{E}_R(\Gamma)\isom 
 D^{\psi=0}[-1]$$
 in $\bold{D}^-(\mathcal{E}_R(\Gamma))$, and this induces a canonical isomorphism 
 $$\Delta^{\mathrm{Iw}}_{R,1}(D)\otimes_{\Lambda_R(\Gamma)}\mathcal{E}_R(\Gamma)\isom (\mathrm{det}_{\mathcal{E}_R(\Gamma)}D^{\psi=0}, r_D)^{-1}.$$
 Moreover, since we have $\mathcal{L}_{\Lambda_R(\Gamma)}(\bold{Dfm}(D))=\mathcal{L}_R(D)\otimes_R\Lambda_R(\Gamma)$, we obtain the following canonical isomorphism 
 \begin{equation}\label{9.11}
 \Delta_R^{\mathrm{Iw}}(D)\otimes_{\Lambda_R(\Gamma)}\mathcal{E}_R(\Gamma)
 \isom (\mathrm{det}_{\mathcal{E}_R(\Gamma)}D^{\psi=0}\otimes_R\mathcal{L}_R(D)^{\vee}, 0)^{-1}.
 \end{equation}

 Using the isomorphism $\Delta_{R}^{\mathrm{Iw}}(D^*)^{\iota}\isom 
 \Delta_{\Lambda_R(\Gamma)}(\bold{Dfm}(D)^*)$, we similarly obtain the following canonical isomorphism 
$$\Delta_{\Lambda_R(\Gamma)}(\bold{Dfm}(D)^*)\otimes_{\Lambda_R(\Gamma)}\mathcal{E}_R(\Gamma)\isom (\mathrm{det}_{\mathcal{E}_R(\Gamma)}(D^*)^{\psi=0, \iota}\otimes_R\mathcal{L}_R(D^*)^{\vee}, 0)^{-1}.$$ Finally, by Proposition VI.1.2 of \cite{Co10a}, the Iwasawa pairing 
$\{-,-\}_{\mathrm{Iw}}:(D^*)^{\psi=1,\iota}\times D^{\psi=1}\rightarrow \Lambda_R(\Gamma)$ uniquely extends to an $\mathcal{E}_R(\Gamma)$-bilinear perfect pairing 
$$\{-,- \}_{0,\mathrm{Iw}}:(D^*)^{\psi=0,\iota}\times D^{\psi=0}\rightarrow \mathcal{E}_R(\Gamma)$$
such that $\{(1-\varphi)x, (1-\varphi)y\}_{0,\mathrm{Iw}}=\{x,y\}_{\mathrm{Iw}}$ for any 
$x\in (D^*)^{\psi=1}, y\in D^{\psi=1}$. 

 \begin{rem}\label{rankone}
 As we mentioned in Remark \ref{2.5}, Conjecture \ref{2.2} is known for the rank one case by \cite{Ka93b}. 
 Using the isomorphism (\ref{9.11}), the base change to $\mathcal{E}_R(\Gamma)$ of the 
 local $\varepsilon$-isomorphism $\varepsilon_{R,\zeta}^{\mathrm{Iw}}(D(T)):=
 \varepsilon_{R,\zeta}^{\mathrm{Iw}}(T)$ for any $R$-representation $T$ of $G_{\mathbb{Q}_p}$ of rank one defined in \cite{Ka93b} is explicitly described as follows, which will play an important role in this article.
 Let $\delta:\mathbb{Q}_p^{\times}\rightarrow R^{\times}$ be a continuous homomorphism corresponding to a character $\delta:G_{\mathbb{Q}_p}^{\mathrm{ab}}\rightarrow 
 R^{\times}$ by the local class field theory. Then, the $(\varphi,\Gamma)$-module $D(R(\delta))$ corresponding to $R(\delta)$ is isomorphic to $\mathcal{E}_R(\delta):=\mathcal{E}_R\bold{e}_{\delta}$ 
 on which $(\varphi,\Gamma)$ acts by $\varphi(\bold{e}_{\delta})=\delta(p)\cdot\bold{e}_{\delta}, 
 \gamma'(\bold{e}_{\delta})=\delta(\gamma')\cdot \bold{e}_{\delta}$ ($\gamma'\in \Gamma$). 
 For $\mathcal{E}_R(\delta)$, one has an $\mathcal{E}_R(\Gamma)$-linear isomorphism 
 $$\mathcal{E}_R(\Gamma)\isom \mathcal{E}_R(\delta)^{\psi=0}:\lambda\mapsto \lambda\cdot((1+X)^{-1}\bold{e}_{\delta}),$$ and, under the isomorphism (\ref{9.11}) for $D=\mathcal{E}_R(\delta)$, the base change to $\mathcal{E}_R(\Gamma)$ of the local $\varepsilon$-isomorphism $\varepsilon_{R,\zeta}^{\mathrm{Iw}}(\mathcal{E}_R(\delta)):\bold{1}_{\Lambda_R(\Gamma)}\isom\Delta_R^{\mathrm{Iw}}(\mathcal{E}_R(\delta))$ which is defined in \cite{Ka93b} is the natural one induced by 
 the isomorphism 
 $$\mathcal{E}_R(\Gamma)\isom \mathcal{E}_R(\delta)^{\psi=0}\otimes_R (R\bold{e}_{\delta})^{\vee}
 :\lambda\mapsto (\lambda\cdot((1+X)^{-1}\bold{e}_{\delta}))\otimes \bold{e}_{\delta}^{\vee}.$$
This fact easily follows from the another definition of $\varepsilon^{\mathrm{Iw}}_{R,\zeta}(\mathcal{E}_R(\delta))$ given in \S 4.1 (and Remark 4.9 and Lemma 4.10) of \cite{Na14b}.
 
 \end{rem}

\subsection{A conjectural definition of the local $\varepsilon$-isomorphism}
In this subsection, we first recall the definition of (a multivariable version of) the Colmez's convolution 
pairing. After that, we propose a conjectural definition of the local 
$\varepsilon$-isomorphism using the convolution pairing.

Let $D_1,\cdots, D_{n+1}$ be 
\'etale $(\varphi,\Gamma)$-modules over $\mathcal{E}_R$, and let 
$$M:D_1\times D_2\times \cdots \times D_n\rightarrow D_{n+1}$$ be 
an $\mathcal{E}_R$-multilinear pairing compatible with $\varphi$ and $\Gamma$, i.e. 
we have $$M(\varphi(x_1), \cdots, \varphi(x_n))=\varphi(M(x_1,\cdots, x_n))$$ and 
$$M(\gamma'(x_1),\cdots, \gamma'(x_n))=\gamma'(M(x_1,\cdots x_n))$$ for any $x_i\in D_i$ and 
$\gamma'\in \Gamma$. 
For such a data, we define a map 
$$M_{\mathbb{Z}_p^{\times}}^{(\zeta)}:D_1^{\psi=0}\times \cdots \times D_n^{\psi=0}\rightarrow D_{n+1}^{\psi=0}:(x_1,\cdots, x_n)\mapsto M^{(\zeta)}_{\mathbb{Z}_p^{\times}}(x_1,\cdots x_n)=:(*)
$$ by the formula
$$(*):=\lim_{n\rightarrow\infty} \sum_{i_1, \cdots, i_n
\in \mathbb{Z}_p^{\times} \text{ mod } p^n}(1+X)^{i_1\cdots i_n}\varphi^n(
M(\sigma_{j_1}\cdot \psi^n((1+X)^{-i_1} x_1), \cdots, \sigma_{j_n}\cdot \psi^n((1+X)^{-i_n}x_n))),$$
where we set $j_k:=\prod_{k'\not=k}i_{k'}$. This is a multivariable version of the Colmez's convolution pairing defined V.4 of \cite{Co10a}, whose well-definedness can be proved in the same way as in Proposition V.4.1 of \cite{Co10a}. We remark that this pairing $M_{\mathbb{Z}_p^{\times}}^{(\zeta)}$ depends on the choice of the parameter 
$X=X_{\zeta}$, i.e. the choice of $\zeta\in \mathbb{Z}_p(1)$. We can easily check that this dependence can be written by the formula
$$M^{(\zeta^a)}_{\mathbb{Z}_p^{\times}}=[\sigma_a]^{-(n-1)}\cdot M_{\mathbb{Z}_p^{\times}}^{(\zeta)}$$ for any $a\in \mathbb{Z}_p^{\times}$. Moreover, we have
$$M^{(\zeta)}_{\mathbb{Z}_p^{\times}}(x_1,\cdots, \gamma'(x_i), \cdots, x_n)=\gamma'(M_{\mathbb{Z}_p^{\times}}(x_1,\cdots, x_i,\cdots,x_n))$$ for any $i$ and 
$\gamma'\in \Gamma$, in particular, $M^{(\zeta)}_{\mathbb{Z}_p^{\times}}$ is $\mathcal{E}_R(\Gamma)$-multilinear.

We next formulate a conjecture on the conjectural definition of the local $\varepsilon$-isomorphism. 
Let $D$ be an \'etale $(\varphi,\Gamma)$-module over $\mathcal{E}_R$. Applying the convolution pairing to the highest wedge product
$$\wedge :D^{\times r_D}\rightarrow \mathrm{det}_{\mathcal{E}_R}D:(x_1,\cdots, x_{r_D})\mapsto 
x_1\wedge\cdots \wedge x_{r_D},$$
we obtain an $\mathcal{E}_R(\Gamma)$-multilinear pairing 
$$\wedge^{(\zeta)}_{\mathbb{Z}_p^{\times}}:(D^{\psi=0})^{\times r_D}\rightarrow (\mathrm{det}_{\mathcal{E}_R}D)^{\psi=0}.$$
It is easy to see that this map is alternating. Hence this induces an $\mathcal{E}_R(\Gamma)$-linear morphism 
$$\wedge^{(\zeta)}_{\mathbb{Z}_p^{\times}}:\mathrm{det}_{\mathcal{E}_R(\Gamma)}D^{\psi=0}\rightarrow 
(\mathrm{det}_{\mathcal{E}_R}D)^{\psi=0}.$$

Concerning the relationship between this map with the local $\varepsilon$-isomorphism, we 
propose the following conjecture which grew out from discussions with S.Yasuda. Recall that we have 
canonical isomorphisms $\Delta_{R}^{\mathrm{Iw}}(D)\otimes_{\Lambda_R(\Gamma)}\mathcal{E}_R(\Gamma)\isom (\mathrm{det}_{\mathcal{E}_R(\Gamma)}D^{\psi=0}\otimes_R\mathcal{L}_R(D)^{\vee}, r_D)^{-1}$ and $\mathcal{L}_R(D)=\mathcal{L}_R(\mathrm{det}_{\mathcal{E}_R}D)$.
\begin{conjecture}\label{2.11}
\begin{itemize}
\item[(1)]For any $D$, the map $\wedge^{(\zeta)}_{\mathbb{Z}_p^{\times}}:\mathrm{det}_{\mathcal{E}_R(\Gamma)}D^{\psi=0}\rightarrow 
(\mathrm{det}_{\mathcal{E}_R}D)^{\psi=0}$ is isomorphism.
\item[(2)]If $(1)$ holds for $D$, then 
the isomorphism 
$$\wedge_{\mathbb{Z}_p^{\times}}^{(\zeta)}:\Delta_R^{\mathrm{Iw}}(D)\otimes_{\Lambda_R(\Gamma)}\mathcal{E}_R(\Gamma)
\isom \Delta_{R}^{\mathrm{Iw}}(\mathrm{det}_{\mathcal{E}_R}D)\otimes_{\Lambda_R(\Gamma)}\mathcal{E}_R(\Gamma)$$ induced by the isomorphism $$\mathrm{det}_{\mathcal{E}_R(\Gamma)}D^{\psi=0}\otimes_R\mathcal{L}_R(D)^{\vee}
\isom (\mathrm{det}_{\mathcal{E}_R}D)^{\psi=0}\otimes_R\mathcal{L}_R(\mathrm{det}_{\mathcal{E}_R}D)^{\vee}$$ defined by $(x_1\wedge \cdots \wedge x_{r_D})\otimes y\mapsto \wedge_{\mathbb{Z}_p^{\times}}^{(\zeta)}(x_1\wedge \cdots \wedge x_{r_D})\otimes y$ uniquely descends to a $\Lambda_R(\Gamma)$-linear isomorphism 
$$\wedge_{\mathbb{Z}_p^{\times}}^{(\zeta)}:\Delta_R^{\mathrm{Iw}}(D)\isom \Delta_{R}^{\mathrm{Iw}}(\mathrm{det}_{\mathcal{E}_R}D).$$

\item[(3)]If $(2)$ holds for $D$, then the conjectural $\varepsilon$-isomorphism $\varepsilon_{R,\zeta}^{\mathrm{Iw}}(D):\bold{1}_{\Lambda_R(\Gamma)}\isom\Delta_R^{\mathrm{Iw}}(D)$ satisfies the commutative diagram 
\begin{equation*}
\begin{CD}
\Delta_R^{\mathrm{Iw}}(D)@> \wedge_{\mathbb{Z}_p^{\times}}^{(\zeta)} >> \Delta_R^{\mathrm{Iw}}(\mathrm{det}_{\mathcal{E}_R}D)\\
@AA \varepsilon_{R,\zeta}^{\mathrm{Iw}}(D)A  @ AA \varepsilon_{R,\zeta}^{\mathrm{Iw}}(\mathrm{det}_{\mathcal{E}_R}D) A \\
\bold{1}_{\Lambda_R(\Gamma)}@>> \mathrm{id} > \bold{1}_{\Lambda_R(\Gamma)}, 
\end{CD}
\end{equation*}
where the isomorphism $\varepsilon_{R,\zeta}^{\mathrm{Iw}}(\mathrm{det}_{\mathcal{E}_R}D)$ is the $\varepsilon$-isomorphism defined by Kato \cite{Ka93b} $($or Remark $\ref{rankone})$.

\end{itemize}
\end{conjecture}
\begin{rem}\label{3333}
The condition (3) in the conjecture above
says that, if (2) is true for $D$, then the composite
$(\wedge^{(\zeta)}_{\mathbb{Z}_p^{\times}})^{-1}\circ\varepsilon_{R,\zeta}^{\mathrm{Iw}}(\mathrm{det}_{\mathcal{E}_R}D):\bold{1}_{\Lambda_R(\Gamma)}\isom\Delta_R^{\mathrm{Iw}}(D)$ satisfies 
all the conditions $(1), \cdots, (5)$ in Conjecture \ref{2.2}. For example, since one has $$\wedge_{\mathbb{Z}_p^{\times}}^{(\zeta^a)}=
[\sigma_{a}]^{r_D-1}\cdot \wedge_{\mathbb{Z}_p^{\times}}^{(\zeta)}:\Delta_R^{\mathrm{Iw}}(D)
\isom \Delta_R^{\mathrm{Iw}}(\mathrm{det}_{\mathcal{E}_R}D)$$ 
(which follows from $\wedge_{\mathbb{Z}_p^{\times}}^{(\zeta^a)}=
[\sigma_{a}]^{-(r_D-1)}\cdot \wedge_{\mathbb{Z}_p^{\times}}^{(\zeta)}:\mathrm{det}_{\mathcal{E}_R(\Gamma)}(D^{\psi=0})\rightarrow (\mathrm{det}_{\mathcal{E}_R}D)^{\psi=0}$) and 
\begin{multline*}
\varepsilon_{R,\zeta^a}^{\mathrm{Iw}}(\mathrm{det}_{\mathcal{E}_R}D)=(\mathrm{det}_{\mathcal{E}_{\Lambda_R(\Gamma)}}\bold{Dfm}(\mathrm{det}_{\mathcal{E}_R}D)(a))\cdot \varepsilon_{R,\zeta}^{\mathrm{Iw}}(\mathrm{det}_{\mathcal{E}_R}D)\\
=\mathrm{det}_{\mathcal{E}_R}D(a)\cdot [\sigma_a]^{-1}\cdot \varepsilon_{R,\zeta}^{\mathrm{Iw}}(\mathrm{det}_{\mathcal{E}_R}D)
\end{multline*} 
for any $a\in \mathbb{Z}_p^{\times}$,  we obtain 
\begin{multline*}
(\wedge^{(\zeta^a)}_{\mathbb{Z}_p^{\times}})^{-1}\circ\varepsilon_{R,\zeta^a}^{\mathrm{Iw}}(\mathrm{det}_{\mathcal{E}_R}D)=[\sigma_{a}]^{-r_D+1}\cdot(\mathrm{det}_{\mathcal{E}_R}D(a)\cdot [\sigma_a]^{-1})\cdot (\wedge_{\mathbb{Z}_p^{\times}}^{(\zeta)})^{-1}\circ\varepsilon_{R,\zeta}^{\mathrm{Iw}}(\mathrm{det}_{\mathcal{E}_R}D)
\\
=\mathrm{det}_{\mathcal{E}_{\Lambda_R(\Gamma)}}
\bold{Dfm}(D)(a)\cdot (\wedge_{\mathbb{Z}_p^{\times}}^{(\zeta)})^{-1}\circ \varepsilon_{R,\zeta}^{\mathrm{Iw}}(\mathrm{det}_{\mathcal{E}_R}D), 
\end{multline*}
i.e. the isomorphism $(\wedge_{\mathbb{Z}_p^{\times}}^{(\zeta)})^{-1}\circ\varepsilon_{R,\zeta}^{\mathrm{Iw}}(\mathrm{det}_{\mathcal{E}_R}D)$ satisfies the condition (3) in Conjecture \ref{2.2}.

\end{rem}
\begin{rem}\label{2.13}
In the next section, we prove almost all the parts of the  conjecture above for the rank two case. 
In fact, we can prove many parts of the conjecture even for the higher rank case. However, 
we do not pursue this problem in the present article since 
the main theme of this article is to pursue the connection between the local $\varepsilon$-conjecture with the $p$-adic local Langlands correspondence for $\mathrm{GL}_2(\mathbb{Q}_p)$.
In the next article \cite{Na}, we will prove (1), (almost all the parts of) (2) for the higher rank case, and prove that the isomorphism $(\wedge_{\mathbb{Z}_p^{\times}}^{(\zeta)})^{-1}\circ\varepsilon_{R,\zeta}^{\mathrm{Iw}}(\mathrm{det}_{\mathcal{E}_R}D):\bold{1}_{\Lambda_R(\Gamma)}\isom\Delta_R^{\mathrm{Iw}}(D)$ (obtained by (2)) satisfies the conditions $(1),\cdots, (4)$ in Conjecture \ref{2.2}. 
Moreover, we will prove that this isomorphism satisfies the condition (5) for the crystabelline case. 
\end{rem}

\section{Local $\varepsilon$-isomorphisms for rank two $p$-adic representations of $\mathrm{Gal}(\overline{\mathbb{Q}}_p/\mathbb{Q}_p)$}
In this section, using the $p$-adic local Langlands correspondence for $\mathrm{GL}_2(\mathbb{Q}_p)$, we prove almost all parts of 
Conjecture \ref{2.2} and Conjecture \ref{2.11} for the rank two case. 

\subsection{Statement of the main theorem on the local $\varepsilon$-conjecture}
We start this section by stating our main result concerning the local $\varepsilon$-conjecture for the rank two case. We say that an \'etale $(\varphi,\Gamma)$-module $D$ over $\mathcal{E}_L$ is de Rham,  trianguline, etc. if the corresponding $V(D):=T(D)$ is so. If $D$ is de Rham, we set $\varepsilon_{L,\zeta}^{\mathrm{dR}}(D):=\varepsilon_{L,\zeta}^{\mathrm{dR}}(V(D))$, which we regard as an isomorphism $\bold{1}_L\isom 
\Delta_L(D)$ by the canonical isomorphism $\Delta_L(D)\isom \Delta_L(V(D))$.

\begin{thm}\label{3.0}

\begin{itemize}
\item[(1)]Conjecture $\ref{2.11}\, (1)$ is true for all the $(\varphi,\Gamma)$-modules of rank two. 
\item[(2)]Conjecture $\ref{2.11}\, (2)$ is true for ``almost all" the $(\varphi,\Gamma)$-modules of rank two.
\item[(3)]For $D$ as in $(2)$ $($then we can define an isomorphism
$$\varepsilon_{R,\zeta}^{\mathrm{Iw}}(D):=(\wedge_{\mathbb{Z}_p^{\times}}^{(\zeta)})^{-1}\circ\varepsilon_{R,\zeta}^{\mathrm{Iw}}(\mathrm{det}_{\mathcal{E}_R}D):\bold{1}_{\Lambda_R(\Gamma)}\isom\Delta_R^{\mathrm{Iw}}(D)),$$ we set $\varepsilon_{R,\zeta}(D):\bold{1}_R\isom\Delta_R(D)$ for the 
base change of $\varepsilon_{R,\zeta}^{\mathrm{Iw}}(D)$ by $f_{\bold{1}}:\Lambda_R(\Gamma)\rightarrow R:[\gamma']\mapsto 1\,\, (\gamma'\in \Gamma)$. Then the set of isomorphisms $\{\varepsilon_{R,\zeta}(D)\}_{(R,D)}$, where $D$ run through all the $D$ of rank one or rank two as in $(2)$, satisfies the conditions $(1),\cdots, (4)$ in Conjecture 
$\ref{2.2}$ and satisfies the following: 

For any pair $(L, D)$ such that $D$ is de Rham of rank one or two, 
\begin{itemize}
\item[(i)]if $D$ is trianguline, then we have 
$$\varepsilon_{L,\zeta}(D)=\varepsilon_{L,\zeta}^{\mathrm{dR}}(D),$$
\item[(ii)]if $D$ is non-trianguline with the distinct Hodge-Tate weights $\{0, k\}$ for $k\geqq 1$, then we have 
$$\varepsilon_{\overline{L},\zeta}(D(-r)(\delta))=\varepsilon_{\overline{L},\zeta}^{\mathrm{dR}}(D(-r)(\delta))$$
for any pair $(r,\delta)$ such that $0\leqq r\leqq k-1$ and $\delta:\Gamma\rightarrow \overline{L}^{\times}$ is a homomorphism with finite image.

\end{itemize}

\end{itemize}

\end{thm}

We will prove this theorem in the next subsections:(1) is proved in Proposition \ref{3.2.1}, 
(2) is proved in Proposition \ref{3.3.1} (see this proposition and Remark \ref{3.4} for the precise meaning of ``almost all" in the  theorem above), (3) (i) is proved in \S3.3, (3) (ii) is proved in \S3.4.

\subsection{Definition of the $\varepsilon$-isomorphisms}
In \cite{Co10b}, Colmez constructed a correspondence $D\mapsto \Pi(D)$ from 
(almost all) \'etale $(\varphi,\Gamma)$-modules of rank two to representations of $\mathrm{GL}_2(\mathbb{Q}_p)$. In the construction of $\Pi(D)$, he introduced a mysterious involution 
$w_{\delta_D}: D^{\psi=0}\isom D^{\psi=0}$ (whose definition we recall below) which is intimately related with the action of 
$\begin{pmatrix}0 & 1\\ 1& 0\end{pmatrix}\in\mathrm{GL}_2(\mathbb{Q}_p)$ on $\Pi(D)$. Moreover, he
 proved a formula describing the convolution 
 pairing  $\wedge_{\mathbb{Z}_p^{\times}}^{(\zeta)}$ using the involution $w_{\delta_D}$ and the Iwasawa pairing 
 $\{-,-\}_{0,\mathrm{Iw}}:(D^*)^{\psi=0,\iota}\times D^{\psi=0}\rightarrow \mathcal{E}_R(\Gamma)$, which we also recall below. Since the $\varepsilon$-constant of an irreducible smooth admissible 
 representation of $\mathrm{GL}_2(\mathbb{Q}_p)$ can be described using the action of 
 $\begin{pmatrix}0 & 1\\ 1& 0\end{pmatrix}$ by the classical theory of Kirillov model, this formula 
 is crucial for our application to the local $\varepsilon$-conjecture.
 
 We start this subsection by recalling the definitions of some of analytic operations on $D^{\psi=0}$ 
 defined in \cite{Co10a}, \cite{Co10b}. These operations also depend on the choice of the parameter $X=X_{\zeta}\in \mathcal{E}_R$, i.e., the choice of $\bold{e}_1:=\zeta\in \mathbb{Z}_p(1)$, which we have fixed. 
 
For a continuous homomorphism $\delta:\Gamma\rightarrow R^{\times}$, Colmez defined in V of \cite{Co10a}
the following map 
$$m_{\delta}^{(\zeta)}:D^{\psi=0}\rightarrow D^{\psi=0}:x\mapsto \lim_{n\rightarrow \infty}\sum_{i\in \mathbb{Z}_p^{\times} \text{mod} p^n}\delta(i)(1+X)^i\varphi^n\psi^n((1+X)^{-i}x).$$
We remark that this map satisfies 
$m^{(\zeta^a)}_{\delta}=\delta(a)^{-1}\cdot m_{\delta}^{(\zeta)}$ for $a\in \mathbb{Z}_p^{\times}$, 
$m^{(\zeta)}_{\bold{1}}=\mathrm{id}_{D^{\psi=0}}$ for the trivial homomorphism $\bold{1}:\Gamma\rightarrow R^{\times}$, $m^{(\zeta)}_{\delta_1}\circ m^{(\zeta)}_{\delta_2}=m^{(\zeta)}_{\delta_1\cdot \delta_2}$ for any $\delta_1,\delta_2$, 
and $\sigma_a\circ m^{(\zeta)}_{\delta}=\delta(a)^{-1}m^{(\zeta)}_{\delta}\circ \sigma_a$ for $a\in \mathbb{Z}_p^{\times}$.
In particular, the map 
$$m^{(\zeta)}_{\delta}\otimes \bold{e}_{\delta}:D^{\psi=0}\isom D(\delta)^{\psi=0}:x\mapsto m^{(\zeta)}_{\delta}(x)\otimes \bold{e}_{\delta}$$
is an isomorphism of $\mathcal{E}_R(\Gamma)$-modules. 
In V of \cite{Co10a}, he also defined an involution 
$$w_*^{(\zeta)}:D^{\psi=0}\isom D^{\psi=0}$$ by the formula
$$w^{(\zeta)}_*(x):=\lim_{n\rightarrow +\infty} 
\sum_{i\in \mathbb{Z}^{\times}_p \mathrm{ mod }\, p^n} (1+X)^{1/i}\sigma_{-1/i^2}\cdot 
\varphi^n\psi^n((1+X)^{-i}x)$$ 
and also defined in II of \cite{Co10b} an involution
$$w^{(\zeta)}_{\delta}:=m^{(\zeta)}_{\delta^{-1}}\circ w^{(\zeta)}_*:D^{\psi=0}\isom D^{\psi=0}$$
for any $\delta:\Gamma\rightarrow R^{\times}$.
By definition, the latter satisfies the equalities $w_{\delta}^{(\zeta^a)}(\sigma_a(x))=\sigma_a(w_{\delta}^{(\zeta)}(x))$ and $w^{(\zeta)}_{\delta}(\sigma_a(x))=\delta(a)\sigma_{a^{-1}}(w^{(\zeta)}_{\delta}(x))$ for any $a\in \mathbb{Z}_p^{\times}$. In particular, this induces an $\mathcal{E}_R(\Gamma)$-linear isomorphism 
$$w^{(\zeta)}_{\delta}\otimes \bold{e}_{\delta^{-1}}:D^{\psi=0}\isom D(\delta^{-1})^{\psi=0,\iota}:
x\mapsto w^{(\zeta)}_{\delta}(x)\otimes \bold{e}_{\delta^{-1}}.$$

Now we assume that $D$ is of rank two. Set 
$$\delta_D:=\chi^{-1}\cdot \mathrm{det}_{\mathcal{E}_R}D:\mathbb{Q}_p^{\times}\rightarrow R^{\times}.$$ Using the canonical isomorphism 
$\mathcal{E}_R\otimes_R\mathcal{L}_{R}(D)\isom \mathrm{det}_{\mathcal{E}_R}D:f\otimes x\mapsto f\cdot x$, we obtain a canonical isomorphism 
$$\mathrm{det}_{\mathcal{E}_R}D\otimes_R\mathcal{L}_{R}(D)^{\vee}\isom \mathcal{E}_R\otimes_R\mathcal{L}_{R}(D)\otimes_R\mathcal{L}_{R}(D)^{\vee}\isom \mathcal{E}_R.$$ Using this isomorphism, we define the following canonical isomorphism of $(\varphi,\Gamma)$-modules 
$$D\otimes_R\mathcal{L}_{R}(D)^{\vee}\isom D^{\vee}: x\otimes z^{\vee} \mapsto [y\mapsto (y\wedge x)\otimes z^{\vee}]$$
for $x, y\in D, z\in \mathcal{L}_R(D)^{\times}$, 
by which we identify both sides. By these isomorphisms, 
we also obtain the following canonical  isomorphism 
$$D^{\psi=0}\otimes_R\mathcal{L}_{R}(D)^{\vee}\isom (D^*)^{\psi=0, \iota}:x\otimes z^{\vee}\mapsto w^{(\zeta)}_{\delta_D}(x)\otimes z^{\vee}\otimes \bold{e}_{1}$$ of 
$\mathcal{E}_R(\Gamma)$-modules. 

Using these preliminaries, we define the following $\mathcal{E}_R(\Gamma)$-bilinear perfect pairing 
$$[-,-]^{(\zeta)}_{\mathrm{Iw}}:D^{\psi=0}\otimes \mathcal{L}_R(D)^{\vee}\times D^{\psi=0}\rightarrow \mathcal{E}_R(\Gamma): (x\otimes z^{\vee}, y)\mapsto \{w^{(\zeta)}_{\delta_D}(x)\otimes z^{\vee}\otimes \bold{e}_1, y\}_{0,\mathrm{Iw}}$$
which is a modified version of the Colmez's pairing defined in Corollaire VI.6.2 of \cite{Co10b}. This pairing  is related with the convolution pairing $\wedge^{(\zeta)}_{\mathbb{Z}_p^{\times}}:D^{\psi=0}\times D^{\psi=0}\rightarrow (\mathrm{det}_{\mathcal{E}_R}D)^{\psi=0}$ as follows. Let us consider the $R$-linear map $d:\mathcal{E}_R\rightarrow 
\mathcal{E}_R(1):f(X)\mapsto (1+X)\frac{df(X)}{dX}\otimes \bold{e}_1$. It is easy to see that this does not depend on the choice of $\zeta\in \mathbb{Z}_p(1)$, and satisfies $\sigma_a\circ d=d\circ \sigma_a$ ($a\in \mathbb{Z}_p^{\times}$) and $\varphi\circ d=p\cdot d\circ \varphi$, and induces an $\mathcal{E}_R(\Gamma)$-linear isomorphism 
$d:\mathcal{E}_R^{\psi=0}\isom \mathcal{E}_R(1)^{\psi=0}$. We note that one has $d|_{\mathcal{E}_R^{\psi=0}}=m^{(\zeta)}_{\chi}\otimes \bold{e}_1$ since both are $\mathcal{E}_R(\Gamma)$-linear and one has $d((1+X))=(1+X)\otimes\bold{e}_1=m_{\chi}^{(\zeta)}((1+X))\otimes\bold{e}_1$.

As a consequence of Colmez's generalized reciprocity law (see Th\'eor\`eme VI.2.1 of \cite{Co10a}), he proved, in the proof of Corollaire VI.6.2 \cite{Co10b}, that $[-,-]^{(\zeta)}_{\mathrm{Iw}}$ satisfies the following equality
\begin{equation}\label{reciprocity}
d([x\otimes z^{\vee}, y]^{(\zeta)}_{\mathrm{Iw}}\cdot (1+X))=-\delta_D(-1)\cdot m^{(\zeta)}_{\delta_D^{-1}}(\wedge^{(\zeta)}_{\mathbb{Z}_p^{\times}}(x, y))\otimes z^{\vee}\otimes \bold{e}_1
\end{equation}
in $\mathcal{E}_R(1)^{\psi=0}\isom (\mathrm{det}_{\mathcal{E}_R}D)^{\psi=0}\otimes_R\mathcal{L}_R(D)^{\vee}(1)$. 

Since $\wedge^{(\zeta)}_{\mathbb{Z}_p^{\times}}$ is anti-symmetric, this formula implies that the perfect pairing 
$[-,-]^{(\zeta)}_{\mathrm{Iw}}$ is also anti-symmetric, i.e. we have $[x\otimes z^{\vee} , y]_{\mathrm{Iw}}=-[y\otimes z^{\vee}, x]_{\mathrm{Iw}}$ for any $x, y\in D^{\psi=0}$ and $z\in \mathcal{L}_R(D)$. Therefore, this induces an $\mathcal{E}_R(\Gamma)$-linear isomorphism 
$$\mathrm{det}_{\mathcal{E}_R(\Gamma)}D^{\psi=0}\otimes_R\mathcal{L}_R(D)^{\vee}
\isom \mathcal{E}_R(\Gamma):(x\wedge y)\otimes z^{\vee}\mapsto [\sigma_{-1}]\cdot[x\otimes z^{\vee}, y]^{(\zeta)}_{\mathrm{Iw}}.$$
The last isomorphism, together with (\ref{9.11}), naturally induces 
an $\mathcal{E}_R(\Gamma)$-linear isomorphism (which we denote by)
$$\eta_{R,\zeta}(D):\bold{1}_{\mathcal{E}_R(\Gamma)}\isom\Delta_R^{\mathrm{Iw}}(D)\otimes_{\Lambda_R(\Gamma)}\mathcal{E}_R(\Gamma).$$ 

We first prove the following proposition concerning the alternative description of our conjectural $\varepsilon$-isomorphism, in particular, which proves Conjecture \ref{2.11} (1) for the rank two case.

\begin{prop}\label{3.2.1}
The map $\wedge_{\mathbb{Z}_p^{\times}}^{(\zeta)}: \mathrm{det}_{\mathcal{E}_R(\Gamma)}D^{\psi=0}\rightarrow (\mathrm{det}_{\mathcal{E}_R}D)^{\psi=0}$ is isomorphism, and the isomorphism $\eta_{R,\zeta}(D)$ fits into  the following commutative diagram:
\begin{equation*}
\begin{CD}
\Delta_R^{\mathrm{Iw}}(D)\otimes_{\Lambda_R(\Gamma)}\mathcal{E}_R(\Gamma)@> \wedge_{\mathbb{Z}_p^{\times}}^{(\zeta)} >> \Delta_R^{\mathrm{Iw}}(\mathrm{det}_{\mathcal{E}_R}D)\otimes_{\Lambda_R(\Gamma)}\mathcal{E}_R(\Gamma)\\
@AA \eta_{R,\zeta}(D)A  @ AA \varepsilon_{R,\zeta}^{\mathrm{Iw}}(\mathrm{det}_{\mathcal{E}_R}D)\otimes \mathrm{id}_{\mathcal{E}_R(\Gamma)}A \\
\bold{1}_{\mathcal{E}_R(\Gamma)}@>> \mathrm{id} > \bold{1}_{\mathcal{E}_R(\Gamma)}.
\end{CD}
\end{equation*}

\end{prop}
\begin{proof}
By Remark \ref{rankone}, it suffices to show the equality 
\begin{equation}\label{12.12}
[\sigma_{-1}]\cdot[x\otimes z^{\vee}, y]^{(\zeta)}_{\mathrm{Iw}}\cdot((1+X)^{-1}\cdot z)
=\wedge_{\mathbb{Z}_p^{\times}}^{(\zeta)}(x, y)
\end{equation}
for any $x, y\in D^{\psi=0}$, $z\in \mathcal{L}_R(D)^{\times}=\mathcal{L}_R(\mathrm{det}_{\mathcal{E}_R}D)^{\times}$. 

We prove this equality as follows. We first remark that, since one has $d=m^{(\zeta)}_{\chi}\otimes \bold{e}_1$, the equality (\ref{reciprocity}) is equivalent to the equality 
\begin{equation}
m^{(\zeta)}_{\chi}([x\otimes z^{\vee}, y]^{(\zeta)}_{\mathrm{Iw}}\cdot (1+X))\otimes\bold{e}_1=-\delta_D(-1)\cdot m^{(\zeta)}_{\delta_D^{-1}}(\wedge^{(\zeta)}_{\mathbb{Z}_p^{\times}}(x, y))\otimes z^{\vee}\otimes\bold{e}_1.
\end{equation}
Applying the $\mathcal{E}_R(\Gamma)$-linear 
isomorphism $m_{\delta_D}\otimes z\otimes \bold{e}_{-1}:\mathcal{E}_R(1)^{\psi=0}\isom (\mathrm{det}_{\mathcal{E}_R}D)^{\psi=0}$ to this equality, the right hand side is equal to  \begin{multline*}
-\delta_D(-1)\cdot m_{\delta_D}^{(\zeta)}(m^{(\zeta)}_{\delta_D^{-1}}(\wedge^{(\zeta)}_{\mathbb{Z}_p^{\times}}(x, y))\otimes z^{\vee}\otimes\bold{e}_1)\otimes z\otimes \bold{e}_{-1}\\
=-\delta_D(-1)\cdot m_{\delta_D}^{(\zeta)}(m^{(\zeta)}_{\delta_D^{-1}}(\wedge^{(\zeta)}_{\mathbb{Z}_p^{\times}}(x, y)))\otimes z^{\vee}\otimes\bold{e}_1\otimes z\otimes \bold{e}_{-1}=-\delta_D(-1)\cdot\wedge_{\mathbb{Z}_p^{\times}}^{(\zeta)}(x,y)
\end{multline*}
since one has $m_{\delta}(x\otimes \bold{e}_{\delta'})=m_{\delta}(x)\otimes 
\bold{e}_{\delta'}$ and $m_{\delta}\circ m_{\delta'}=m_{\delta\cdot\delta'}$ 
for any $D$ and $\delta, \delta'$, and the left hand side is equal to (set $\delta_0:=\mathrm{det}_{\mathcal{E}_R}D|_{\mathbb{Z}_p^{\times}}$)

\begin{multline*}
m^{(\zeta)}_{\delta_D}(m^{(\zeta)}_{\chi}([x\otimes z^{\vee}, y]^{(\zeta)}_{\mathrm{Iw}}\cdot (1+X))\otimes\bold{e}_1)\otimes z\otimes \bold{e}_{-1}\\
=m^{(\zeta)}_{\delta_D}(m^{(\zeta)}_{\chi}( [x\otimes z^{\vee}, y]^{(\zeta)}_{\mathrm{Iw}}\cdot (1+X)))\otimes \bold{e}_1\otimes z\otimes \bold{e}_{-1}\\
=m^{(\zeta)}_{\delta_0}([x\otimes z^{\vee}, y]^{(\zeta)}_{\mathrm{Iw}}\cdot (1+X))\otimes 
z=[x\otimes z^{\vee},y]^{(\zeta)}_{\mathrm{Iw}}\cdot (m^{(\zeta)}_{\delta_0}(1+X)\otimes z)\\
=[x\otimes z^{\vee},y]^{(\zeta)}_{\mathrm{Iw}}\cdot ((1+X)\cdot z)=-\delta_D(-1)\cdot[\sigma_{-1}]\cdot[x\otimes z^{\vee},y]_{\mathrm{Iw}}^{(\zeta)}\cdot((1+X)^{-1}\cdot z),
\end{multline*}
where the third equality follows from the $\mathcal{E}_R(\Gamma)$-linearity of $m^{(\zeta)}_{\delta_0}\otimes z$, and the fourth follows from $m^{(\zeta)}_{\delta_0}(1+X)=(1+X)$, from which the equality (\ref{12.12}) follows.

\end{proof}

 Before proving (2) of Conjecture \ref{2.11}, we show the isomorphism $\eta_{R,\zeta}(D)$ 
 satisfies the condition similar to (4) in Conjecture \ref{2.2} (over the ring $\mathcal{E}_R(\Gamma)$). 

 \begin{lemma}\label{llll}
 Let $D$ be an \'etale $(\varphi,\Gamma)$-module over $\mathcal{E}_R$ of rank two. Then the isomorphisms $\eta_{R,\zeta}(D)$ and $\eta_{R,\zeta}(D^*)$ fit into the following commutative diagram: 
 \begin{equation*}
\begin{CD}
\Delta^{\mathrm{Iw}}_{R}(D)\otimes_{\Lambda_R(\Gamma)}\mathcal{E}_R(\Gamma)
@>>> (\Delta^{\mathrm{Iw}}_{R}(D^*)^{\iota}\otimes_{\Lambda_R(\Gamma)}\mathcal{E}_R(\Gamma))^{\vee}\boxtimes(\mathcal{E}_R(\Gamma)(r_T),0) \\
@A \mathrm{det}_{\mathcal{E}_R}D(\sigma_{-1})\cdot \eta_{R,\zeta}(D) AA  @ VV(\eta_{R,\zeta}(D^*)^{\iota})^{\vee}\boxtimes [\bold{e}_{r_T}\mapsto 1]V \\
\bold{1}_{\mathcal{E}_R(\Gamma)}@> \mathrm{can}>> \bold{1}_{\mathcal{E}_R(\Gamma)}\boxtimes \bold{1}_{\mathcal{E}_R(\Gamma)}.
\end{CD}
\end{equation*}
Here the upper horizontal arrow is the base change to $\mathcal{E}_R(\Gamma)$ of the isomorphism 
$\Delta^{\mathrm{Iw}}_{R}(D)\isom (\Delta^{\mathrm{Iw}}_{R}(D^*)^{\iota})^{\vee}\boxtimes(\Lambda_R(\Gamma)(r_T),0)$
defined by the Tate duality.

 \end{lemma}
 
 Before starting the proof, let us introduce the following notation.
In the proof we will use the pairings $[-,-]^{(\zeta)}_{\mathrm{Iw}}$ 
and $\{-,-\}_{\mathrm{Iw},0}$ for $D$ and those for $D^*$ 
simultaneously. In order to distinguish the pairings for $D$ with those for $D^*$,
we will denote, for any \'etale $(\varphi,\Gamma)$-module $D_1$ of rank two,
the pairings $[-,-]^{(\zeta)}_{\mathrm{Iw}}$ 
and $\{-,-\}_{\mathrm{Iw},0}$ for $D_1$
by $[-,-]^{(\zeta)}_{\mathrm{Iw},D_1}$ 
and $\{-,-\}^{D_1}_{\mathrm{Iw},0}$, respectively.
 
 \begin{proof}
 Fix $z\in \mathcal{L}_R(D)^{\times}$. Then we have $z^{\vee}\otimes \bold{e}_2\in \mathcal{L}_R(D^*)^{\times}$.
 By definition, it suffices to show that the following diagram is commutative:
  \begin{equation*}
\begin{CD}
\mathrm{det}_{\mathcal{E}_R(\Gamma)}D^{\psi=0}
@> (a)>> (\mathrm{det}_{\mathcal{E}_R(\Gamma)}(D^*)^{\psi=0,\iota})^{\vee} \\
@V (b) VV  @ AA (c)\boxtimes[\bold{e}_2\mapsto 1] A \\
\mathcal{E}_R(\Gamma)\otimes_R\mathcal{L}_R(D)@> (d) >> 
(\mathcal{E}_R(\Gamma)\otimes_R\mathcal{L}_R(D^*))^{\vee}\otimes_RR(2),
\end{CD}
\end{equation*}
where the horizontal arrows are the natural one defined by the Tate duality, and (b) is defined by $x\wedge y\mapsto \mathrm{det}_{\mathcal{E}_R}D(\sigma_{-1})\cdot [\sigma_{-1}]\cdot[x\otimes z^{\vee} , y]^{(\zeta)}_{\mathrm{Iw}, D}\otimes z$ for $x, y\in D^{\psi=0}$ and 
(c) is the dual of the map $x'\wedge y'\mapsto \iota([\sigma_{-1}]\cdot [x'\otimes 
(z\otimes \bold{e}_{-2}), y']^{(\zeta)}_{\mathrm{Iw}, D^*})\otimes (z^{\vee}\otimes \bold{e}_2)$ for $x', y'\in (D^*)^{\psi=0}$. We prove this commutativity as follows.

Take a basis $\{x,y\}$ of $D^{\psi=0}$. Since we have an isomorphism 
$w^{(\zeta)}_{\delta_D}\otimes z^{\vee}\otimes \bold{e}_1:D^{\psi=0}\isom (D^*)^{\psi=0,\iota}$, 
then $\{w^{(\zeta)}_{\delta_D}(x)\otimes z^{\vee}\otimes \bold{e}_1, w^{(\zeta)}_{\delta_D}(y)\otimes z^{\vee}\otimes \bold{e}_1\}$ is a basis of $(D^*)^{\psi=0, \iota}$. 
Then, (a) sends $x\wedge y$ to $f\in (\mathrm{det}_{\mathcal{E}_R(\Gamma)}(D^*)^{\psi=0,\iota})^{\vee}$ defined by 
\begin{multline*}
f((w^{(\zeta)}_{\delta_D}(x)\otimes z^{\vee}\otimes \bold{e}_1)\wedge(w^{(\zeta)}_{\delta_D}(y)\otimes z^{\vee}\otimes \bold{e}_1))\\
=[x\otimes z^{\vee}, x]^{(\zeta)}_{\mathrm{Iw},D}\cdot [y\otimes z^{\vee}, y]^{(\zeta)}_{\mathrm{Iw},D}
-[x\otimes z^{\vee}, y]^{(\zeta)}_{\mathrm{Iw},D}\cdot [y\otimes z^{\vee}, x]^{(\zeta)}_{\mathrm{Iw},D}
=([x\otimes z^{\vee}, y]^{(\zeta)}_{\mathrm{Iw},D})^2,
\end{multline*}
where the first equality is by definition and the second follows since $[, ]_{\mathrm{Iw}}^{(\zeta)}$ is anti-symmetric. By definition, the composite $((c)\boxtimes[\bold{e}_1\mapsto 1])\circ(d)\circ (b)$ 
sends $x\wedge y$ to $f' \in (\mathrm{det}_{\mathcal{E}_R(\Gamma)}(D^*)^{\psi=0,\iota})^{\vee}$ defined by 
\begin{multline*}
f'((w^{(\zeta)}_{\delta_D}(x)\otimes z^{\vee}\otimes \bold{e}_1)\wedge(w^{(\zeta)}_{\delta_D}(y)\otimes z^{\vee}\otimes \bold{e}_1))\\
=\mathrm{det}_{\mathcal{E}_R}D(\sigma_{-1})\cdot 
[x\otimes z^{\vee}, y]^{(\zeta)}_{\mathrm{Iw},D}\cdot 
\iota([(w^{(\zeta)}_{\delta_D}(x)\otimes z^{\vee}\otimes \bold{e}_1)\otimes (z\otimes \bold{e}_{-2}), 
w^{(\zeta)}_{\delta_D}(y)\otimes z^{\vee}\otimes \bold{e}_1)]^{(\zeta)}_{\mathrm{Iw},D^*}).
\end{multline*}
Therefore, it suffices to show the equality 
\begin{equation}\label{kkkk}
[x\otimes z^{\vee}, y]^{(\zeta)}_{\mathrm{Iw},D}
=\mathrm{det}_{\mathcal{E}_R}D(\sigma_{-1})\cdot \iota([(w^{(\zeta)}_{\delta_D}(x)\otimes z^{\vee}\otimes \bold{e}_1)\otimes (z\otimes \bold{e}_{-2}), 
w^{(\zeta)}_{\delta_D}(y)\otimes z^{\vee}\otimes \bold{e}_1)]^{(\zeta)}_{\mathrm{Iw},D^*}).
\end{equation}

To show this equality, we first remark that one has
\begin{multline*}
w^{(\zeta)}_{\delta_{D^*}}(w^{(\zeta)}_{\delta_D}(x)\otimes z^{\vee}\otimes \bold{e}_1) 
=w^{(\zeta)}_{\delta^{-1}_{D}}(w^{(\zeta)}_{\delta_D}(x)\otimes z^{\vee}\otimes \bold{e}_1)\\
=m^{(\zeta)}_{\delta_D}\circ w^{(\zeta)}_*(m_{\delta_D^{-1}}\circ w^{(\zeta)}_*(x)\otimes z^{\vee}\otimes \bold{e}_1)
=\delta_D(-1)\cdot(m^{(\zeta)}_{\delta_D}\circ m^{(\zeta)}_{\delta_D^{-2}}\circ w^{(\zeta)}_*\circ m^{(\zeta)}_{\delta_D^{-1}}\circ w^{(\zeta)}_*(x))
\otimes z^{\vee}\otimes \bold{e}_1\\
=\delta_D(-1)\cdot(m^{(\zeta)}_{\delta_D}\circ m^{(\zeta)}_{\delta_D^{-2}}\circ m^{(\zeta)}_{\delta_D}\circ w^{(\zeta)}_*\circ w^{(\zeta)}_*(x))
\otimes z^{\vee}\otimes \bold{e}_1=\delta_D(-1)\cdot x\otimes z^{\vee}\otimes \bold{e}_1,
\end{multline*}
where the third equality follows from $w^{(\zeta)}_*(x\otimes \bold{e}_{\delta})=\delta(-1)\cdot m^{(\zeta)}_{\delta}\circ w^{(\zeta)}_*(x)\otimes \bold{e}_{\delta}$ (Corollaire V.5.2. of \cite{Co10a}) and the fourth follows from $m^{(\zeta)}_{\delta^{-1}}\circ w^{(\zeta)}_*=w^{(\zeta)}_*\circ m^{(\zeta)}_{\delta}$ (Proposition V.2.4 of \cite{Co10a}) for 
any $\delta$. Hence, the right hand side of (\ref{kkkk}) is equal to 
\begin{multline*}
\mathrm{det}_{\mathcal{E}_R}D(\sigma_{-1})\cdot\delta_D(-1)\cdot \iota(\{(x\otimes z^{\vee}\otimes \bold{e}_1)\otimes (z\otimes \bold{e}_{-2})\otimes \bold{e}_1, w^{(\zeta)}_{\delta_D}(y)\otimes z^{\vee}\otimes \bold{e}_1\}^{D^*}_{\mathrm{Iw}, 0})\\
=-\mathrm{det}_{\mathcal{E}_R}D(\sigma_{-1})\cdot\delta_D(-1)\cdot \iota(\{x, w^{(\zeta)}_{\delta_D}(y)\otimes z^{\vee}\otimes \bold{e}_1\}^{D^*}_{\mathrm{Iw}, 0})
=\iota(\{x, w^{(\zeta)}_{\delta_D}(y)\otimes z^{\vee}\otimes \bold{e}_1\}^{D^*}_{\mathrm{Iw}, 0})\\
=-\{w^{(\zeta)}_{\delta_D}(y)\otimes z^{\vee}\otimes \bold{e}_1,x\}_{\mathrm{Iw},0}^D
=-[y\otimes z^{\vee}, x]^{(\zeta)}_{\mathrm{Iw},D}=[x\otimes z^{\vee}, y]_{\mathrm{Iw},D}^{(\zeta)},
\end{multline*}
where the first equality follows from the fact that the composite of the canonical isomorphisms 
$(D\otimes_R\mathcal{L}_R(D)^{\vee})\otimes_R\mathcal{L}_R(D^{\vee})^{\vee}\isom D^{\vee}\otimes_R\mathcal{L}_D(D^{\vee})^{\vee}\isom (D^{\vee})^{\vee}$ is given by $x\otimes z^{\vee}\otimes z\mapsto [f\mapsto -f(x)]$ for any $z\in \mathcal{L}_R(D)^{\times}$, which shows the equality (\ref{kkkk}), hence finishes to prove the lemma.

 \end{proof}
 

We next prove (2) of Conjecture \ref{2.11} under the following assumption. 
Let $\overline{V}$ be an $\mathbb{F}$-representation of $G_{\mathbb{Q}_p}$ of over a finite 
field $\mathbb{F}$ of characteristic $p$. We denote by $R_{\overline{V}}$ the universal deformation 
ring of $\overline{V}$ (resp. a versal deformation ring or the universal framed deformation ring)
if it exists (resp. the universal deformation ring does not exist), and denote by $V^{\mathrm{univ}}$ the universal deformation (resp. 
a versal deformation or the underlying representation of the universal framed deformation) of $\overline{V}$ over $R_{\overline{V}}$.  Let $R_0$ be a topological $\mathbb{Z}_p$-algebra satisfying 
the condition (i) in \S2.1. Let $R$ be either $R_0$ or $R_0[1/p]$ (resp. $R=R_0[1/p]$) when $p\geqq 3$ (resp. $p=2$). Let $V$ be an $R$-representation of $G_{\mathbb{Q}_p}$. Set $V_0:=V$ (resp. $V_0$ a $G_{\mathbb{Q}_p}$-stable  $R_0$-lattice of $V$) if $R=R_0$ (resp. $R=R_0[1/p]$). Since $R_0$ is a finite product of local rings, we may assume that $R_0$ is local and denote by $\mathfrak{m}_{R_0}$ the maximal ideal of $R_0$. If we set $\overline{V}:=V_0\otimes_{R_0}R_0/\mathfrak{m}_{R_0}$, then there exists a homomorphism $R_{\overline{V}}\rightarrow R$ such that $V^{\mathrm{univ}}\otimes_{R_{\overline{V}}}R\isom V$. 
Set $\mathcal{X}:=\mathrm{Spec}(R_{\overline{V}}[1/p])$, and denote by 
$\mathcal{X}_0$ the subset of all the closed points in $\mathcal{X}$.

\begin{prop}\label{3.3.1}
Let $D$ be an \'etale $(\varphi,\Gamma)$-module over $\mathcal{E}_R$ or rank two. 
Set $V:=V(D)$ and $\overline{V}:=V_0\otimes_{R_0}R_0/\mathfrak{m}_{R_0}$ for 
an $R_0$-lattice $V_0$ of $V$. Assume one of the following conditions $(1)$ and $(2):$
\begin{itemize}
\item[(1)]$p\geqq 3$.
\item[(2)]$p=2$ and, for an $R_{\overline{V}}$ as above, 
$$\mathcal{X}_{\mathrm{cris}}:=\{x\in \mathcal{X}_0|V_x:=x^*(V^{\mathrm{univ}}) \text{ is absolutely irreducible and crystalline} \}$$  is Zariski dense in $\mathcal{X}$, 

\end{itemize}
then the isomorphism $\eta_{R,\zeta}(D)$ descends to $\Lambda_R(\Gamma)$, which we denote by 
$$\varepsilon_{R,\zeta}^{\mathrm{Iw}}(D):\bold{1}_{\Lambda_R(\Gamma)}\isom \Delta_R^{\mathrm{Iw}}(D).$$

\end{prop}
\begin{proof}
We first remark that, when $p=3$, the second condition in (2) (i.e. density of $\mathcal{X}_{\mathrm{cris}}$) always holds for the universal (or a versal) deformation ring $R_{\overline{V}}$ and it is known to be an integral domain 
(in particular $p$-torsion free) by the results of \cite{Co08}, \cite{Ki10}, \cite{Bo10} and \cite{BJ14}. 

By the compatibility with the base change, it suffices to show the proposition for $V^{\mathrm{univ}}$ 
(for $V^{\mathrm{univ}}[1/p]$ if $p=2$). 
Set $R:=R_{\overline{V}}$ and $V:=V^{\mathrm{univ}}$ for simplicity. 
The $p$-torsion freeness of $R$ when $p\geqq 3$ implies that we have $\Lambda_R(\Gamma)=\mathcal{E}_R(\Gamma)\cap \Lambda_{R[1/p]}(\Gamma)$.
Therefore, it suffices to show the theorem for $R[1/p]$ (for any $p$). Moreover, since 
we have $\Lambda_{R[1/p]}(\Gamma)=\mathrm{Ker}(\mathcal{E}_{R[1/p]}(\Gamma)\rightarrow \prod_{x\in \mathcal{X}_{\mathrm{cris}}} \mathcal{E}_{L_x}(\Gamma)/\Lambda_{L_x}(\Gamma))$ (here, $L_x$ is the residue field at $x$) by the assumption on the density, it suffices to show the proposition for $V_x:=x^*(V)$ for any $x\in \mathcal{X}_{\mathrm{cris}}$. 

Let $V$ be an absolutely irreducible $L$-representation for a finite extension $L$ of $\mathbb{Q}_p$ 
corresponding to a point in $\mathcal{X}$. 
Set $D:=D(V)$. Then, one has $D^{\varphi=1}=D/(\psi-1)D=0$ and $D^{\psi=1}\isom (1-\varphi)D^{\psi=1}=:\mathcal{C}(D)$ is a free $\Lambda_L(\Gamma)$-module of rank two by 
\S II, \S VI of \cite{Co10a}, and the same results hold for $D^*$. Hence, as in the case of $\Delta_{R}^{\mathrm{Iw}}(D)\otimes_{\Lambda_R(\Gamma)}\mathcal{E}_R(\Gamma)$, we obtain a canonical isomorphism 
$$\Delta_{L}^{\mathrm{Iw}}(D)\isom (\mathrm{det}_{\Lambda_L(\Gamma)}\mathcal{C}(D)\otimes_L\mathcal{L}_L(D)^{\vee}, 0)^{-1}.$$
 Moreover, the Iwasawa pairing 
$\{-,-\}_{0,\mathrm{Iw}}:\mathcal{C}(D^*)^{\iota}\times \mathcal{C}(D)\rightarrow \Lambda_L(\Gamma)$ is perfect by Proposition VI.1.2 of \cite{Co10a}, and, if we fix $z\in \mathcal{L}_L(D)^{\times}$, one has an isomorphism 
$$\mathcal{C}(D)\isom \mathcal{C}(D^*): x\mapsto w_{\delta_D}(x)\otimes z^{\vee}\otimes \bold{e}_1$$
by Proposition V.2.1 of \cite{Co10b}. Therefore, we obtain an isomorphism 
$$\mathrm{det}_{\Lambda_L(\Gamma)}\mathcal{C}(D)\otimes_L\mathcal{L}_L(D)^{\vee}\isom 
\Lambda_L(\Gamma): (x\wedge y)\otimes z^{\vee}\mapsto \{w_{\delta_D}(x)\otimes z^{\vee}\otimes \bold{e}_1, y\}_{0,\mathrm{Iw}},$$which proves the proposition for $D$ by definition of $\eta_{R,\zeta}(D)$.
\end{proof}
\begin{rem}\label{3.4}
Even when $p=2$, the assumption in the proposition holds for almost all the cases (the author does not know any example which does not satisfies the assumption). For example, for any $L$-representation $V$ for a finite extension $L$ of $\mathbb{Q}_p$, there exists an $\mathcal{O}_L$-lattice $V_0$ of $V$ such that its $R_{\overline{V}}$ satisfies the assumption (see \cite{CDP14a}).

\end{rem}

\begin{rem}\label{3.5.1}
In the proof of the  proposition above, the most subtle point is to show that 
the isomorphism $D^{\psi=0}\isom (D^*)^{\psi=0}:x\mapsto w_{\delta_D}(x)\otimes z^{\vee}\otimes \bold{e}_1$ induces an isomorphism 
$\mathcal{C}(D)\isom \mathcal{C}(D^*)$ when $D$ is absolutely irreducible. This fact is a consequence of the $\mathrm{GL}_2(\mathbb{Q}_p)$-compatibility of the pair $(D,\delta_D)$ for such $D$ (see \S III of \cite{CD14}), which is a very deep result in the theory of $p$-adic local Langlands correspondence for $\mathrm{GL}_2(\mathbb{Q}_p)$. 
In the next article \cite{Na}, we will give another proof of this proposition, and prove a similar proposition even for 
higher rank case under a similar assumption on the density of $\mathcal{X}_{\mathrm{cris}}$ by directly comparing the convolution $\wedge^{(\zeta)}_{\mathbb{Z}_p^{\times}}$ with the local $\varepsilon$-isomorphisms defined in \cite{BB08} and \cite{Na14b}.

\end{rem}

From now on, we only treat the $(\varphi,\Gamma)$-modules of rank two which satisfy the assumption in Proposition \ref{3.3.1} without any comment, which gives no restriction to the results proved in the next sections since any $L$-representations of $G_{\mathbb{Q}_p}$ of rank two satisfies the assumption by Remark \ref{3.4}.

Specializing the  $\varepsilon$-isomorphism above, we define the $\varepsilon$-isomorphism 
$\varepsilon_{R,\zeta}(D)$ as follows.

\begin{defn}\label{3.6.1}
Let $D$ be an \'etale $(\varphi,\Gamma)$-module over $\mathcal{E}_R$ of rank two. We define the  isomorphism $\varepsilon_{R,\zeta}(D)$ to be the base change
$$\varepsilon_{R,\zeta}(D):=\varepsilon_{R,\zeta}^{\mathrm{Iw}}(D)\otimes_{\Lambda_R(\Gamma), f_{\bold{1}}}\mathrm{id}_R:\bold{1}_R\isom \Delta_{R}^{\mathrm{Iw}}(D)\otimes_{\Lambda_R(\Gamma), f_{\bold{1}}}R\isom
\Delta_R(D),$$ where 
$f_{\bold{1}}:\Lambda_R(\Gamma)\rightarrow R$ is defined by $f_{\bold{1}}([\gamma']):=1$ for $\gamma'\in \Gamma$.

\end{defn}

\begin{corollary}
Our local $\varepsilon$-isomorphism $\varepsilon_{R,\zeta}(D)$ defined in Definition $\ref{3.6.1}$ satisfies the conditions $(1), (3), (4)$ of Conjecture $\ref{2.2}$.
\end{corollary}
\begin{proof}
That $\varepsilon_{R,\zeta}(D)$ satisfies the condition (1) is trivial by definition. 
To show that $\varepsilon_{R,\zeta}(D)$ satisfies the conditions (3) and (4), it suffices to show that 
$\varepsilon^{\mathrm{Iw}}_{R,\zeta}(D)$ satisfies (3) and (4). Since the canonical map $\Lambda_R(\Gamma)\rightarrow \mathcal{E}_R(\Gamma)$ is injective, this claim follows from 
Remark \ref{3333} and Lemma \ref{llll}.
\end{proof}
\begin{rem}\label{3.7.1}
By definition and Lemma \ref{2.10.1}, we also have  
$$\varepsilon_{R,\zeta}(D(\delta))=\varepsilon_{R,\zeta}^{\mathrm{Iw}}(D)\otimes_{\Lambda_R(\Gamma),f_{\delta}}R$$ for any $\delta:\Gamma \rightarrow R^{\times}$ under the canonical isomorphism $$\Delta_R(D(\delta))\isom \Delta_R^{\mathrm{Iw}}(D)\otimes_{\Lambda_R(\Gamma),f_{\delta}}R.$$
\end{rem}

 

\subsection{The verification of the de Rham condition: the trianguline case}
This and the next subsections are the technical hearts of this article, where we prove that our $\varepsilon$-isomorphism defined in Definition \ref{3.6.1}
satisfies the condition (5) in Conjecture \ref{2.2} which we call the de Rham condition. In this subsection, we prove this condition in the trianguline case by comparing the local $\varepsilon$-isomorphism defined in Definition \ref{3.6.1} with that defined in the previous article \cite{Na14b}. 

In \cite{Na14b}, we generalized the ($p$-adic) local $\varepsilon$-conjecture for rigid analytic families 
of $(\varphi,\Gamma)$-modules over the Robba ring, and proved this generalized version of conjecture for families of trianguline $(\varphi,\Gamma)$-modules, (a special case of) which we briefly recall now. For details, see \cite{KPX14} for the general results on the cohomology theory 
of $(\varphi,\Gamma)$-modules over the Robba ring, and \cite{Na14b} for the generalized version of the 
local $\varepsilon$-conjecture.

We denote by  $|-|:\overline{\mathbb{Q}}_p^{\times}\rightarrow \mathbb{Q}_{> 0}$ the absolute value normalized by 
$|p|:=1/p$. Define topological $L$-algebras $\mathcal{R}_L^{(n)}$ ($n\geqq 1$) and $\mathcal{R}_L$ by 
$$\mathcal{R}^{(n)}_L:=\left\{\left.\sum_{m\in \mathbb{Z}}a_mX^m\right|a_m\in L, \sum_{m\in \mathbb{Z}}a_mX^m
\text{ is convergent on } 
|\zeta_{p^n}-1|\leqq |X| <1\right\}$$
and $\mathcal{R}_L:=\bigcup_{n\geqq 1}\mathcal{R}_L^{(n)}$
on which $\varphi$ and $\Gamma$ act by $\varphi(X)=(1+X)^p-1$ and $\gamma'(X)=(1+X)^{\chi(\gamma')}-1$ ($\gamma'\in \Gamma$). For $n\geqq 1$, we say that $M^{(n)}$ is a $(\varphi,\Gamma)$-module over $\mathcal{R}^{(n)}_L$ if $M^{(n)}$ is a finite free $\mathcal{R}^{(n)}_L$-module with a Frobenius structure 
$$\varphi^*M^{(n)}:=M^{(n)}\otimes_{\mathcal{R}_L^{(n)},\varphi}\mathcal{R}_L^{(n+1)}\isom M^{(n+1)}:=M^{(n)}\otimes_{\mathcal{R}_L^{(n)}}\mathcal{R}_L^{(n+1)}$$ and a continuous semi-linear action of $\Gamma$ which commutes with the Frobenius structure. We say that an $\mathcal{R}_L$-module $M$ is a $(\varphi,\Gamma)$-module over $\mathcal{R}_L$ if it is the base change of a $(\varphi,\Gamma)$-module 
$M^{(n)}$ over $\mathcal{R}_L^{(n)}$ for some $n\geqq 1$. We denote by $n(M)\geqq 1$ the smallest such $n$, and define $M^{(n)}:=M^{(n(M))}\otimes_{\mathcal{R}_L^{(n(M))}}\mathcal{R}_L^{(n)}$ (remark that this is well-defined). 
By the theorems of Cherbonnier-Colmez \cite{CC98} and Kedlaya \cite{Ke04}, one has an exact fully faithful functor 
$$D\mapsto D_{\mathrm{rig}}:=D^{\dagger}\otimes_{\mathcal{E}_L^{\dagger}}\mathcal{R}_L$$ from the category of 
\'etale $(\varphi,\Gamma)$-modules over $\mathcal{E}_L$ to that of $(\varphi,\Gamma)$-modules over $\mathcal{R}_L$, where $D^{\dagger}$ is the largest \'etale $(\varphi,\Gamma)$-submodule 
of $D$ defined over 
$$\mathcal{E}_L^{\dagger}:=\{f(X)\in \mathcal{E}_L|f(X) \text{ is convergent on } 
r\leqq |X| <1 \text{ for some }r<1\}.$$ For any $(\varphi,\Gamma)$-module $M$ over $\mathcal{R}_L$, we can similarly define $\mathrm{C}^{\bullet}_{\varphi,\gamma}(M)$, $\Delta_L(M)$, $\bold{Dfm}(M)$ and $\Delta_L^{\mathrm{Iw}}(M)$, etc. as follows. First, we define 
$$\mathrm{C}^{\bullet}_{\varphi,\gamma}(M),\,\, \mathrm{C}^{\bullet}_{\psi,\gamma}(M) \text{ and }\Delta_{L,1}(M)$$ in the same way as in the \'etale $(\varphi,\Gamma)$-case. 
To define $\Delta_{L,2}(D)$, we first recall that the rank one $(\varphi,\Gamma)$-modules over $\mathcal{R}_L$ are classified by the continuous homomorphisms $\delta:\mathbb{Q}_p^{\times}\rightarrow L^{\times}$, i.e. the rank one $(\varphi,\Gamma)$-module corresponding to $\delta$ is defined 
by 
$$\mathcal{R}_L(\delta):=\mathcal{R}_L\bold{e}_{\delta}$$ on which $\varphi$ and $\Gamma$ act by 
$$\varphi(\bold{e}_{\delta})=\delta(p)\cdot\bold{e}_{\delta} \text{ and }\gamma'(\bold{e}_{\delta})=\delta(\chi(\gamma'))\cdot\bold{e}_{\delta} \text{ for } \gamma'\in \Gamma.$$ For any $(\varphi,\Gamma)$-module $M$ over $\mathcal{R}_L$, we also regard $\mathrm{det}_{\mathcal{R}_L}M$ as continuous homomorphisms $\mathrm{det}_{\mathcal{R}_L}M:\mathbb{Q}_p^{\times}\rightarrow L^{\times}$ or $\mathrm{det}_{\mathcal{R}_L}M:W_{\mathbb{Q}_p}^{\mathrm{ab}}\rightarrow L^{\times}$ by this correspondence and the local class field theory. Using the homomorphism $\mathrm{det}_{\mathcal{R}_L}M$, we define 
$$\mathcal{L}_L(M),\,\, \Delta_{L,2}(M) \text{ and }\Delta_L(M)$$ in the same way as in the \'etale case. To define $\bold{Dfm}(M)$ and $\Delta_L^{\mathrm{Iw}}(M)$, we first define $\Lambda_L(\Gamma)$-algebras $\mathcal{R}_L^{\infty}(\Gamma)$ and $\mathcal{R}_L(\Gamma)$ as follows. Fix a decomposition $\Gamma\isom \Gamma_{\mathrm{tor}}\times \mathbb{Z}_p$ and set $\gamma_0\in \Gamma$ corresponding to $(e,1)$ on the right hand side. Then, we define $$\mathcal{R}_L^{\infty}(\Gamma):=\mathbb{Z}_p[\Gamma_{\mathrm{tor}}]\otimes_{\mathbb{Z}_p}\mathcal{R}_L^{\infty}([\gamma_0]-1) \text{ and }\mathcal{R}_L(\Gamma):= \mathbb{Z}_p[\Gamma_{\mathrm{tor}}]\otimes_{\mathbb{Z}_p}\mathcal{R}_L([\gamma_0]-1),$$ where we set 
$$\mathcal{R}_L([\gamma_0]-1):=\left\{\left.\sum_{m\in \mathbb{Z}}a_m([\gamma_0]-1)^m\right|\sum_{m\in \mathbb{Z}}a_mX^m\in \mathcal{R}_L\right\}$$ and 
$$\mathcal{R}_L^{\infty}([\gamma_0]-1):=\left\{\sum_{m\geqq 0}a_m([\gamma_0]-1)^m\in \mathcal{R}_L([\gamma_0]-1)\right\}.$$ 
We define a $(\varphi,\Gamma)$-module $\bold{Dfm}(M)$
 over 
$\mathcal{R}_L\hat{\otimes}_L\mathcal{R}_L^{\infty}(\Gamma)$ (which is the relative Robba ring 
with coefficients in $\mathcal{R}_L^{\infty}(\Gamma)$) to be
$$\bold{Dfm}(M):=M\hat{\otimes}_L\mathcal{R}_L^{\infty}(\Gamma)$$ as an 
$\mathcal{R}_L\hat{\otimes}_L\mathcal{R}_L^{\infty}(\Gamma)$-module on which $\varphi$ and $\Gamma$ act by 
$$\varphi(x\hat{\otimes} y):=\varphi(x)\hat{\otimes}y\text{ and }\gamma'(x\hat{\otimes}y):=\gamma'(x)\hat{\otimes}[\gamma']^{-1}\cdot y$$ for $x\in M, y\in \mathcal{R}_L^{\infty}(\Gamma)$ and $\gamma'\in \Gamma$. By \cite{KPX14}, 
one similarly has the following canonical quasi-isomorphisms 
$$C^{\bullet}_{\varphi,\gamma}(\bold{Dfm}(M))\isom 
C^{\bullet}_{\psi,\gamma}(\bold{Dfm}(M))\isom C^{\bullet}_{\psi}(M)$$ of complexes of $\mathcal{R}_L^{\infty}(\Gamma)$-modules, and it is known that these are perfect complexes of $\mathcal{R}_L^{\infty}(\Gamma)$-modules. One also has the (extended) Iwasawa pairing 
$$\{-,-\}_{0,\mathrm{Iw}}:(M^*)^{\psi=0,\iota}\times M^{\psi=0}\rightarrow \mathcal{R}_L(\Gamma)$$
by \S4.2 of \cite{KPX14}.
Therefore, we can similarly define the following graded invertible 
$\mathcal{R}_L^{\infty}(\Gamma)$-modules
$$\Delta_{L,1}^{\mathrm{Iw}}(M):=\mathrm{Det}_{\mathcal{R}_L^{\infty}(\Gamma)}(C^{\bullet}_{\varphi,\gamma}(\bold{Dfm}(M)))$$ and 
$$\Delta_{L}^{\mathrm{Iw}}(M):=\Delta_{L,1}^{\mathrm{Iw}}(M)\boxtimes_{\mathcal{R}_L^{\infty}(\Gamma)}(\Delta_{L,2}(M)\otimes_L\mathcal{R}_L^{\infty}(\Gamma))$$
(remark that we have $\Delta_{\mathcal{R}^{\infty}_L(\Gamma),2}(\bold{Dfm}(M))\isom \Delta_{L,2}(M)\otimes_L\mathcal{R}_L^{\infty}(\Gamma)$),  and we can similarly obtain a canonical isomorphism 
$$\Delta_{L}^{\mathrm{Iw}}(M)\otimes_{\mathcal{R}_L^{\infty}(\Gamma)}\mathcal{R}_L(\Gamma)\isom (\mathrm{det}_{\mathcal{R}_L(\Gamma)}M^{\psi=0}\otimes_L\mathcal{L}_L(M)^{\vee}, 0)^{-1}$$ 
using Proposition 4.3.8 (3) \cite{KPX14} (precisely, this proposition is proved under the assumption that 
$M/(\psi-1)=M^*/(\psi-1)=0$, but we can easily prove  the statement (3) of this proposition for general $M$ in a similar way). 

One can also generalize the $p$-adic Hodge theory for $(\varphi,\Gamma)$-modules over $\mathcal{R}_L$. For a field $F$ of characteristic zero and $n\in \mathbb{Z}_{\geqq 1}$, we set $F_n:=F\otimes_{\mathbb{Q}}\mathbb{Q}(\mu_{p^n})$, and set $F_{\infty}:=\bigcup_{n\geqq 1} F_n$.
Set $t:=\mathrm{log}(1+X)\in \mathcal{R}_L$ (remark that this corresponds to $t_{\zeta}\in \bold{B}^+_{\mathrm{dR}}$).
Set $\bold{D}_{\mathrm{cris}}(M):=M[1/t]^{\Gamma}$. For $n\geqq 1$, one has the following $\Gamma$-equivariant injection 
$$\iota_n:\mathcal{R}_L^{(n)}\hookrightarrow L_n[[t]]:f(X)\mapsto f(\zeta_{p^n}\mathrm{exp}(t/p^n)-1).$$
Using this map, we define for $n\geqq n(M)$
$$\bold{D}_{\mathrm{dif},n}^+(M):=M^{(n)}\otimes_{\mathcal{R}_L^{(n)},  \iota_n}L_n[[t]],\,\,\bold{D}_{\mathrm{dif}, n}(M):=\bold{D}_{\mathrm{dif},n}^+(M)[1/t]$$
and $\bold{D}_{\mathrm{dif},\infty}^{*}(M):=\varinjlim_{n\geqq n(M)}\bold{D}_{\mathrm{dif},n}^{*}(M)$ for $*=+$ or $*=\phi$ (the empty set), where the transition maps are the maps $\bold{D}^+_{\mathrm{dif}, n}(M)\rightarrow \bold{D}^+_{\mathrm{dif},n+1}(M):x\otimes y\mapsto \varphi(x)\otimes y$. Using these, we define 
$$\bold{D}_{\mathrm{dR}}(M):=\bold{D}_{\mathrm{dif},\infty}(M)^{\Gamma},\,\,\bold{D}_{\mathrm{dR}}^i(M):=(t^i\bold{D}^+_{\mathrm{dif},\infty}(M))^{\Gamma}, \text{ and } t_M:=\bold{D}_{\mathrm{dR}}(M)/\bold{D}^0_{\mathrm{dR}}(M)$$ 
for $i\in \mathbb{Z}$. Using these (and those for $M\otimes_{\mathcal{R}_L}\mathcal{R}_L(K)$, where $\mathcal{R}_L(K)$ is the Robba ring of a finite extension $K$ of $\mathbb{Q}_p$), one can define the notions of a crystalline  $(\varphi,\Gamma)$-module, a de Rham $(\varphi,\Gamma)$-module, etc.\  over the Robba ring. In particular, for any de Rham $M$ (which is also known to be potentially semi-stable), one can define $\bold{D}_{\mathrm{pst}}(M)$ and its associated $L$-representation $W(M)$ of $'W_{\mathbb{Q}_p}$, as we usually do for a de Rham representation of $G_{\mathbb{Q}_p}$. By \S2 of \cite{Na14b}, one can also generalize the Bloch-Kato's fundamental exact sequence 
\begin{multline}
0\rightarrow \mathrm{H}^0_{\varphi,\gamma}(M)\rightarrow \bold{D}_{\mathrm{cris}}(M)\xrightarrow{(a)}
\bold{D}_{\mathrm{cris}}(M)\oplus t_M\xrightarrow{(b)}\mathrm{H}^1_{\varphi,\gamma}(M)\\
\xrightarrow{(c)}\bold{D}_{\mathrm{cris}}(M^*)^{\vee}\oplus \bold{D}_{\mathrm{dR}}^0(M)\xrightarrow{(d)}\bold{D}_{\mathrm{cris}}(M^{*})^{\vee}\rightarrow \mathrm{H}^2_{\varphi,\gamma}(M)\rightarrow 0,
\end{multline}
as in the exact sequence (\ref{exp}) in \S2.1 for any de Rham $M$. Using this,  one can define the de Rham $\varepsilon$-isomorphism 
$$\varepsilon_{L,\zeta}^{\mathrm{dR}}(M):\bold{1}_L\isom \Delta_L(M)$$
for any de Rham $M$ (see \S3.3 of \cite{Na14b} for the precise definition) in the same way as that for de Rham $V$.

Let $D$ be an \'etale $(\varphi,\Gamma)$-module over $\mathcal{E}_L$. Then, one has the following canonical comparison isomorphisms 
$$\Delta_{L}(D)\isom \Delta_{L}(D_{\mathrm{rig}})$$ 
and 
$$\Delta_{L}^{\mathrm{Iw}}(D)\otimes_{\Lambda_L(\Gamma)}\mathcal{R}_L^{\infty}(\Gamma)
\isom \Delta_{L}^{\mathrm{Iw}}(D_{\mathrm{rig}})$$
by Proposition 2.7 \cite{Li08} and Theorem 1.9 of \cite{Po13} respectively. For $V$ a de Rham $L$-representation of $G_{\mathbb{Q}_p}$, one has canonical isomorphisms 
$\bold{D}_{\mathrm{dR}}(V)\isom \bold{D}_{\mathrm{dR}}(D(V)_{\mathrm{rig}})$, $\bold{D}_{\mathrm{cris}}(V)\isom \bold{D}_{\mathrm{cris}}(D(V)_{\mathrm{rig}})$ and $W(V)\isom W(D(V)_{\mathrm{rig}})$, etc. and one has
\begin{equation}\label{aaaa}\varepsilon_{L,\zeta}^{\mathrm{dR}}(V)=\varepsilon_{L,\zeta}^{\mathrm{dR}}(D(V)_{\mathrm{rig}})\end{equation} under the isomorphism $\Delta_L(V)\isom \Delta_L(D(V))\isom \Delta_L(D(V)_{\mathrm{rig}})$ (see \cite{Na14b}), by which we freely identify the 
both sides of \eqref{aaaa} with each other.

Let us go back to our situation. Let $D$ be an \'etale $(\varphi,\Gamma)$-module over $\mathcal{E}_L$ of rank two. We assume that $D$ is trianguline, which means that there exist a finite extension $L'$ of $L$ and continuous homomorphisms $\delta_1,\delta_2:\mathbb{Q}_p^{\times}\rightarrow (L')^{\times}$ such that $D_{\mathrm{rig}}\otimes_LL'$ sits in an exact sequence of the following form
$$0\rightarrow \mathcal{R}_{L'}(\delta_1)\rightarrow D_{\mathrm{rig}}\otimes_LL'\rightarrow 
\mathcal{R}_{L'}(\delta_2)\rightarrow 0.$$ Since scalar extensions do not affect  our results, we assume that $L'=L$ from now on. In our previous article \cite{Na14b}, we defined an $\varepsilon$-isomorphism 
$$\varepsilon_{L,\zeta}^{\mathrm{Iw}}(\mathcal{R}_L(\delta)):\bold{1}_{\mathcal{R}_L^{\infty}(\Gamma)}\isom \Delta_{L}^{\mathrm{Iw}}(\mathcal{R}_L(\delta))$$ for any continuous homomorphism 
$\delta:\mathbb{Q}_p^{\times}\rightarrow L^{\times}$, and showed that this satisfies the same conditions (1), (3), (4) and (5) in Conjecture \ref{2.2}. The main result of this subsection is the following theorem.

\begin{thm}\label{3.1}
Under the situation above, one has an equality 
$$\varepsilon^{\mathrm{Iw}}_{L,\zeta}(D)\otimes \mathrm{id}_{\mathcal{R}^{\infty}_L(\Gamma)}
=\varepsilon^{\mathrm{Iw}}_{L,\zeta}(\mathcal{R}_L(\delta_1))\boxtimes\varepsilon^{\mathrm{Iw}}_{L,\zeta}(\mathcal{R}_L(\delta_2))$$
under the canonical isomorphism 
$$\Delta_{L}^{\mathrm{Iw}}(D)\otimes_{\Lambda_L(\Gamma)}\mathcal{R}_L^{\infty}(\Gamma)
\isom\Delta^{\mathrm{Iw}}_{L}(D_{\mathrm{rig}})\isom \Delta^{\mathrm{Iw}}_{L}(\mathcal{R}_L(\delta_1))\boxtimes\Delta^{\mathrm{Iw}}_{L}(\mathcal{R}_L(\delta_2))$$ which is induced by the  exact sequence above
$0\rightarrow \mathcal{R}_L(\delta_1)\rightarrow D_{\mathrm{rig}}\rightarrow \mathcal{R}_L(\delta_2)\rightarrow 0$.

\end{thm}

Before proving this theorem, we first show the equality $\varepsilon_{L,\zeta}(D)=\varepsilon_{L,\zeta}^{\mathrm{dR}}(D)$ for trianguline and de Rham $D$ as a corollary of this theorem.
\begin{corollary}\label{3.9.1}
Let $D$ be an \'etale $(\varphi,\Gamma)$-module over $\mathcal{E}_L$ of rank two which is de Rham and trianguline, 
then we have
$$\varepsilon_{L,\zeta}(D)=\varepsilon_{L,\zeta}^{\mathrm{dR}}(D).$$

\end{corollary}
\begin{proof}
Specializing the equality $\varepsilon^{\mathrm{Iw}}_{L,\zeta}(D)\otimes \mathrm{id}_{\mathcal{R}^{\infty}_L(\Gamma)}
=\varepsilon^{\mathrm{Iw}}_{L,\zeta}(\mathcal{R}_L(\delta_1))\boxtimes\varepsilon^{\mathrm{Iw}}_{L,\zeta}(\mathcal{R}_L(\delta_2))$ in the above theorem by the continuous $L$-algebra morphism $f_{\bold{1}}:\mathcal{R}^{\infty}_L(\Gamma)\rightarrow L:[\gamma']\mapsto 1$ ($\gamma'\in\Gamma$), we obtain an equality 
$$\varepsilon_{L,\zeta}(D)=\varepsilon_{L,\zeta}(\mathcal{R}_L(\delta_1))\boxtimes \varepsilon_{L,\zeta}(\mathcal{R}_L(\delta_2)).$$
Then, the corollary follows from the equalities 
$\varepsilon_{L,\zeta}(\mathcal{R}_L(\delta_i))=\varepsilon_{L,\zeta}^{\mathrm{dR}}(\mathcal{R}_L(\delta_i))$ for $i=1,2$ (Theorem 3.13 of \cite{Na14b}) and 
$\varepsilon_{L,\zeta}^{\mathrm{dR}}(D)=\varepsilon_{L,\zeta}^{\mathrm{dR}}(D_{\mathrm{rig}})
=\varepsilon_{L,\zeta}^{\mathrm{dR}}(\mathcal{R}_L(\delta_1))\boxtimes \varepsilon_{L,\zeta}^{\mathrm{dR}}(\mathcal{R}_L(\delta_2))$ (Lemma 3.9 of \cite{Na14b}).

\end{proof}

We next prove a lemma concerning the explicit description of the 
extended Iwasawa pairing 
$$\{-,-\}_{0,\mathrm{Iw}}:(\mathcal{R}_L(\delta)^{*})^{\psi=0,\iota}\times \mathcal{R}_L(\delta)^{\psi=0}
\rightarrow \mathcal{R}_L(\Gamma)$$ 
for $M=\mathcal{R}_L(\delta)$.
We identify $\mathcal{R}_L(\delta^{-1})$ with $\mathcal{R}_L(\delta)^{\vee}$ via 
the isomorphism $$\mathcal{R}_L(\delta^{-1})\isom \mathcal{R}_L(\delta)^{\vee}:f\bold{e}_{\delta^{-1}}\mapsto [g\bold{e}_{\delta}\mapsto fg].$$

\begin{lemma}\label{3.3}
The extended Iwasawa pairing 
$$\{-,-\}_{0,\mathrm{Iw}}:(\mathcal{R}_L(\delta)^*)^{\psi=0, \iota}
\times \mathcal{R}_L(\delta)^{\psi=0}
\rightarrow \mathcal{R}_L(\Gamma)$$ satisfies the equality
$$\{\lambda_1\cdot_{\iota}((1+X)^{-1}\bold{e}_{\delta^{-1}}\otimes\bold{e}_{1}), \lambda_2\cdot((1+X)\bold{e}_{\delta})\}_{0,\mathrm{Iw}}=\lambda_1\lambda_2$$
for any $\lambda_1,\lambda_2\in \mathcal{R}_L(\Gamma)$.

\end{lemma}
\begin{proof}
We first remark that the isomorphism 
$$\varepsilon^{\mathrm{Iw}}_{L,\zeta}(\mathcal{R}_L(\delta))\otimes\mathrm{id}_{\mathcal{R}_L(\Gamma)}:\bold{1}_{\mathcal{R}_L(\Gamma)}\isom 
\Delta^{\mathrm{Iw}}_L(\mathcal{R}_L(\delta))\otimes_{\mathcal{R}_L^{\infty}(\Gamma)}\mathcal{R}_L(\Gamma)$$ is equal to the one induced by the isomorphism 
$$\theta_{\zeta}(\mathcal{R}_L(\delta)):\mathcal{R}_L(\Gamma)\otimes_LL(\delta)\isom \mathcal{R}_L(\delta)^{\psi=0}:\lambda\otimes \bold{e}_{\delta}\mapsto \lambda((1+X)^{-1}\bold{e}_{\delta})$$
and the isomorphism $\Delta_{L}^{\mathrm{Iw}}(\mathcal{R}_L(\delta))
\isom (\mathcal{R}_L(\delta)^{\psi=0}\otimes_LL(\delta)^{\vee}, 0)^{-1}$. This follows easily from the definition of $\varepsilon^{\mathrm{Iw}}_L(\mathcal{R}_L(\delta))$ given in \S4.1 of \cite{Na14b}.

Since one has 
$$\varepsilon_{L,\zeta}^{\mathrm{Iw}}(\mathcal{R}_L(\delta)^*)^{\iota}=\varepsilon_{\mathcal{R}_L^{\infty}(\Gamma),\zeta}(\bold{Dfm}(\mathcal{R}_L(\delta))^*)$$
under the canonical isomorphism 
$$\Delta_{L}^{\mathrm{Iw}}(\mathcal{R}_L(\delta)^*)^{\iota}
\isom\Delta_{\mathcal{R}_L^{\infty}(\Gamma)}(\bold{Dfm}(\mathcal{R}_L(\delta))^{*}),$$
the isomorphism $\varepsilon_{\mathcal{R}_L^{\infty}(\Gamma),\zeta}(\bold{Dfm}(\mathcal{R}_L(\delta))^*)\otimes\mathrm{id}_{\mathcal{R}_L(\Gamma)}$ is equal to the one induced by 
 the isomorphism
$$\theta_{\zeta}(\mathcal{R}_L(\delta)^*)^{\iota}:\mathcal{R}_L(\Gamma)\otimes_L L(\delta^{-1})(1)
\isom (\mathcal{R}_L(\delta)^*)^{\psi=0,\iota}:\lambda \otimes(\bold{e}_{\delta^{-1}}\otimes\bold{e}_1)\mapsto 
\lambda\cdot_{\iota}((1+X)^{-1}\bold{e}_{\delta^{-1}}\otimes \bold{e}_1).$$
Under the canonical isomorphism 
$$\Delta_{L}^{\mathrm{Iw}}(\mathcal{R}_L(\delta))\isom (\Delta_L^{\mathrm{Iw}}(\mathcal{R}_L(\delta)^*)^{\iota})^{\vee}\boxtimes(\mathcal{R}_L^{\infty}(\Gamma)\otimes_L L(1), 0)$$ 
defined by the Tate duality (see \S 3.2 of \cite{Na14b}), one has
$$\varepsilon_{L,\zeta^{-1}}^{\mathrm{Iw}}(\mathcal{R}_L(\delta))^{-1}=(\varepsilon_{L,\zeta}^{\mathrm{Iw}}(\mathcal{R}_L(\delta)^{*})^{\iota})^{\vee}\boxtimes[\bold{e}_1\mapsto 1]$$
by the condition (4) of Conjecture \ref{2.2} for $\bold{Dfm}(\mathcal{R}_L(\delta))$ (which is proved in
Theorem 3.13 of \cite{Na14b}).
 Using the isomorphisms $\theta_{\zeta^{-1}}(\mathcal{R}_A(\delta))$ and $\theta_{\zeta}(\mathcal{R}_A(\delta)^*)^{\iota}$, we obtain from this equality the following commutative diagram of $\mathcal{R}_L(\Gamma)$-bilinear pairings:

 \begin{equation*}
 \begin{CD}
  \mathcal{R}_L(\Gamma)\otimes_LL(\delta^{-1})(1)\times  \mathcal{R}_L(\Gamma)\otimes_LL(\delta)
 @>  \theta_{\zeta}(\mathcal{R}_L(\delta)^*)^{\iota}\times  \theta_{\zeta^{-1}}(\mathcal{R}_L(\delta))>>
   (\mathcal{R}_L(\delta)^*)^{\psi=0,\iota}\times  \mathcal{R}_L(\delta)^{\psi=0}\\
 @ VV (1\otimes\bold{e}_{\delta^{-1}}\otimes \bold{e}_1, 1\otimes\bold{e}_{\delta})\mapsto 1\otimes\bold{e}_1V @VV \{-,-\}_{0,\mathrm{Iw}} V \\
 \mathcal{R}_L(\Gamma)\otimes_L L(1) @>\bold{e}_1\mapsto 1>> \mathcal{R}_L(\Gamma).
   \end{CD}
 \end{equation*}
Since one has $(1+X_{\zeta^{-1}})=(1+X_{\zeta})^{-1}=(1+X)^{-1}$, the lemma follows from the commutativity of this diagram.

\end{proof}

Using these preliminaries, we prove the theorem as follows. As we show in the proof, a result of 
Dospinescu \cite{Do11} on the explicit description of the action of $w_{\delta_D}$ on $D_{\mathrm{rig}}^{\psi=0}$, which is intimately related with the action of $w=\begin{pmatrix}0 & 1\\ 1& 0\end{pmatrix}$ on the locally analytic vectors, is crucial for the proof.
\begin{proof}(of Theorem \ref{3.1})
We first show the theorem when $D$ is absolutely  irreducible. 
Since the canonical map $\mathcal{R}_L^{\infty}(\Gamma)\rightarrow \mathcal{R}_L(\Gamma)$ is injective, 
it suffices to show the equality after the base change to $\mathcal{R}_L(\Gamma)$. 

By the results in V.2 of \cite{Co10b}, the involution $w_{\delta_D}:D^{\psi=0}\isom D^{\psi=0}$ 
(first descends to $w_{\delta_D}:D^{\dagger,\psi=0}\isom D^{\dagger,\psi=0}$, and then) uniquely extends to 
$w_{\delta_D}:D_{\mathrm{rig}}^{\psi=0}\isom D_{\mathrm{rig}}^{\psi=0}$, and the isomorphism 
$$\varepsilon_{L,\zeta}^{\mathrm{Iw}}(D)\otimes\mathrm{id}_{\mathcal{R}_L(\Gamma)}: \bold{1}_{\mathcal{R}_L(\Gamma)}\isom
\Delta_{L}^{\mathrm{Iw}}(D)\otimes_{\Lambda_L(\Gamma)} \mathcal{R}_L(\Gamma)(\isom (\mathrm{det}_{\mathcal{R}_L(\Gamma)}D^{\psi=0}_{\mathrm{rig}}\otimes_L \mathcal{L}_L(D)^{\vee}, 0)^{-1})$$ is the one which is naturally 
induced by the isomorphism 
\begin{multline*}
\theta_{\zeta}(D_{\mathrm{rig}}):\mathrm{det}_{\mathcal{R}_L(\Gamma)}D_{\mathrm{rig}}^{\psi=0}\isom \mathcal{R}_L(\Gamma)\otimes_L \mathcal{L}_L(D_{\mathrm{rig}}):\\
(x\wedge y) \mapsto [\sigma_{-1}]\cdot
\{w_{\delta_D}(x)\otimes z^{\vee}\otimes \bold{e}_1, y\}_{0,\mathrm{Iw}}\otimes z
\end{multline*}
for any $z\in \mathcal{L}_L(D_{\mathrm{rig}})^{\times}$.
By the explicit descriptions of $\varepsilon^{\mathrm{Iw}}_{L,\zeta}(\mathcal{R}_L(\delta_i))\otimes\mathrm{id}_{\mathcal{R}_L(\Gamma)}$ and $\varepsilon^{\mathrm{Iw}}_{L,\zeta}(D)\otimes\mathrm{id}_{\mathcal{R}_L(\Gamma)}$, it suffices to show that the following diagram is 
commutative:
 \begin{equation*}
 \begin{CD}
  \mathcal{R}_L(\delta_1)^{\psi=0}\otimes_{\mathcal{R}_L(\Gamma)}\mathcal{R}_L(\delta_2)^{\psi=0}@> x\otimes y\mapsto x\wedge\widetilde{y}>>
  \mathrm{det}_{\mathcal{R}_L(\Gamma)}D_{\mathrm{rig}}^{\psi=0}\\
 @ V \theta_{\zeta}(\mathcal{R}_L(\delta_1))^{-1}\otimes\theta_{\zeta}(\mathcal{R}_L(\delta_2))^{-1}VV @VV \theta_{\zeta}(D_{\mathrm{rig}})V \\
  (\mathcal{R}_L(\Gamma)\otimes_L L(\delta_1))\otimes_{\mathcal{R}_L(\Gamma)}(\mathcal{R}_L(\Gamma)\otimes_LL(\delta_2))@> \bold{e}_{\delta_1}\otimes \bold{e}_{\delta_2}\mapsto 
  \bold{e}_{\delta_1}\wedge\widetilde{\bold{e}}_{\delta_2}>>\mathcal{R}_L(\Gamma)\otimes_L\mathcal{L}_L(D_{\mathrm{rig}}).
   \end{CD}
 \end{equation*}
 Here $\widetilde{y}\in D_{\mathrm{rig}}^{\psi=0}$ (resp. $\widetilde{\bold{e}}_{\delta_2}\in D_{\mathrm{rig}}$) is a lift of $y\in \mathcal{R}_L(\delta_2)^{\psi=0}$ (resp. $\bold{e}_{\delta_2}\in 
 \mathcal{R}_L(\delta_2)$).

By definitions of $\theta_{\zeta}(\mathcal{R}_L(\delta_i))$ and $\theta_{\zeta}(D_{\mathrm{rig}})$, and the $\mathcal{R}_L(\Gamma)$-bilinearity of the pairings in the  diagram above, it suffices to show the equality
\begin{equation}\label{equality}
[\sigma_{-1}]\cdot\{w_{\delta_D}((1+X)^{-1}\bold{e}_{\delta_1})
\otimes  (\bold{e}_{\delta_1}\wedge
\widetilde{\bold{e}}_{\delta_2})^{\vee}\otimes\bold{e}_1, \delta_2(p)^{-1}\cdot(1+X)^{-1}\varphi(\widetilde{\bold{e}}_{\delta_2})\}_{0,\mathrm{Iw}}=1.
\end{equation}

Since one has an equality
$$w_{\delta_D}((1+X)\bold{e}_{\delta_1})=\delta_1(-1)\cdot (1+X)\bold{e}_{\delta_1}$$
by (the proof of) Proposition 3.2 of \cite{Do11}, one also has
\begin{multline}
w_{\delta_D}((1+X)^{-1}\bold{e}_{\delta_1})=\delta_1(-1)\cdot w_{\delta_D}(\sigma_{-1}((1+X)\bold{e}_{\delta_1}))=\delta_D(-1)\cdot \delta_1(-1)\cdot \sigma_{-1}(w_{\delta_D}((1+X)\bold{e}_{\delta_1}))\\
=\delta_D(-1)\cdot \delta_1(-1)\cdot \sigma_{-1}(\delta_1(-1)\cdot (1+X)\bold{e}_{\delta_1})
=\delta_D(-1)\cdot \delta_1(-1)\cdot (1+X)^{-1}\bold{e}_{\delta_1}
\end{multline}
since one has $w_{\delta_D}\circ\sigma_a=\delta_D(a)\cdot \sigma_a^{-1}\circ w_{\delta_D}$ ($a\in \mathbb{Z}_p^{\times}$). 
Using this equality and the equality $\bold{e}_{\delta_1}\otimes (\bold{e}_{\delta_1}\wedge
\widetilde{\bold{e}}_{\delta_2})^{\vee}=-\bold{e}_{\delta_2^{-1}}$ in 
$\mathcal{R}_L(\delta_2^{-1})\subseteq D_{\mathrm{rig}}^{\vee}$, the left hand side of (\ref{equality}) is equal to 
\begin{multline*}
-\delta_D(-1)\cdot \delta_1(-1)\cdot[\sigma_{-1}]\cdot\{(1+X)^{-1}\bold{e}_{\delta_2^{-1}}\otimes \bold{e}_1, \delta_2(p)^{-1}\cdot(1+X)^{-1}\varphi(\widetilde{\bold{e}}_{\delta_2})\}_{0,\mathrm{Iw}}\\
=-\delta_D(-1)\cdot \delta_1(-1)\cdot[\sigma_{-1}]\cdot\{(1+X)^{-1}\bold{e}_{\delta_2^{-1}}\otimes \bold{e}_1, (1+X)^{-1}\bold{e}_{\delta_2}\}_{0,\mathrm{Iw}}\\
=-\delta_D(-1)\cdot \delta_1(-1)\cdot\delta_2(-1)\cdot[\sigma_{-1}]\cdot\{(1+X)^{-1}\bold{e}_{\delta_2^{-1}}\otimes \bold{e}_1, [\sigma_{-1}]\cdot((1+X)\bold{e}_{\delta_2})\}_{0,\mathrm{Iw}}\\
=\{(1+X)^{-1}\bold{e}_{\delta_2^{-1}}\otimes \bold{e}_1, (1+X)\bold{e}_{\delta_2}\}_{0,\mathrm{Iw}}=1,
\end{multline*}
where the last equality follows from the equality $-\delta_D(-1)\cdot\delta_1(-1)\cdot \delta_2(-1)=1$ and Lemma \ref{3.3}.

  When $D$ is not absolutely irreducible, then (after extending scalars $L$), we have an 
  exact sequence 
  $$0\rightarrow \mathcal{E}_{L}(\delta_1)
  \rightarrow D\rightarrow \mathcal{E}_{L}(\delta_2)\rightarrow 0$$
  for some continuous homomorphisms $\delta_1,\delta_2:\mathbb{Q}_p^{\times}\rightarrow 
  \mathcal{O}_L^{\times}$. 
  Then, the involution $w_{\delta_D}$ acts on $\mathcal{E}_L(\delta_1)^{\psi=0}$ (in fact, it acts on any \'etale $(\varphi,\Gamma)$-modules), and one can directly  check that one has $w_{\delta_D}((1+X)\bold{e}_{\delta_1}))=\delta_1(-1)(1+X)\bold{e}_{\delta_1}$. Then, the theorem follows by the same argument as in the absolutely 
  irreducible case.
  
\end{proof}

\begin{rem}
In the last paragraph of the proof above, for any exact sequence 
$0\rightarrow \mathcal{E}_R(\delta_1)\rightarrow D\rightarrow \mathcal{E}_R(\delta_2)\rightarrow 0$ of 
\'etale $(\varphi,\Gamma)$-modules over $\mathcal{E}_R$, we show the equality 
$$\varepsilon^{\mathrm{Iw}}_{R,\zeta}(D)=\varepsilon_{R,\zeta}^{\mathrm{Iw}}(\mathcal{E}_R(\delta_1))\boxtimes \varepsilon_{R,\zeta}^{\mathrm{Iw}}(\mathcal{E}_R(\delta_2))$$ 
under the canonical isomorphism 
$\Delta_R^{\mathrm{Iw}}(D)\isom \Delta_R^{\mathrm{Iw}}(\mathcal{E}_R(\delta_1))\boxtimes
\Delta_R^{\mathrm{Iw}}(\mathcal{E}_R(\delta_2))$, which shows that our $\varepsilon$-isomorphism
 satisfies the condition (2) in Conjecture \ref{2.2}.
\end{rem}



\subsection{The verification of the de Rham condition: the non-trianguline case}
In this subsection, we treat the non-trianguline case. Fix an algebraic closure $\overline{L}$ of $L$. We prove the following theorem. 
\begin{thm}\label{nontrianguline}
Let $D$ be an \'etale $(\varphi,\Gamma)$-module over $\mathcal{E}_L$ of rank two which is de Rham 
and non-trianguline. Assume that the Hodge-Tate weights of $D$ are  $\{0,k\}$ for some $k\geqq 1$. Then, we have 
$$\varepsilon_{\overline{L},\zeta}(D(-r)(\delta))=\varepsilon_{\overline{L},\zeta}^{\mathrm{dR}}(D(-r)(\delta))$$
for any integer $0\leqq r\leqq k-1$ and any character $\delta:\Gamma\rightarrow 
\overline{L}^{\times}$ with finite image。

\end{thm}

We first reduce the proof of this theorem to Proposition \ref{key} below by explicitly describing the both sides of the equality in the theorem, and then in the last part of this subsection we prove this key proposition using the Colmez's theory of Kirillov model of locally algebraic vectors 
$\Pi(D)^{\mathrm{alg}}$ of $\Pi(D)$ and the Emerton's theorem on the compatibility of the $p$-adic and the classical local Langlands correspondence. 

Many arguments in this subsection are also valid for the absolutely irreducible and trianguline case, some of which we need in the next section for the application to a functional equation of Kato's Euler system. Hence, in this subsection, we treat $D$ of rank two which satisfies the following conditions:
\begin{itemize}
\item[(i)]$D$ is an absolutely irreducible \'etale $(\varphi,\Gamma)$-module over $\mathcal{E}_L$  which is de Rham with Hodge-Tate weights $\{0,k\}$ for some $k\geqq 1$. 
\item[(ii)]$\bold{D}_{\mathrm{cris}}(D_{\mathrm{rig}}(-r)(\delta))^{\varphi=1}=\bold{D}_{\mathrm{cris}}(D_{\mathrm{rig}}^*(r)(\delta^{-1}))^{\varphi=1}=0$ for any $0\leqq r\leqq k-1$ and any character $\delta:\Gamma\rightarrow 
\overline{L}^{\times}$ with finite image. 
\end{itemize}
We remark that the assumptions (i) and (ii) hold for the non-trianguline case since one has $\bold{D}_{\mathrm{cris}}(D_{\mathrm{rig}}(-r)(\delta))=\bold{D}_{\mathrm{cris}}(D_{\mathrm{rig}}^*(r)(\delta^{-1}))=0$ for any $(r,\delta)$ in this case. The condition $\bold{D}_{\mathrm{cris}}(D_{\mathrm{rig}}(-r)(\delta))^{\varphi=1}\not=0$ or $\bold{D}_{\mathrm{cris}}(D_{\mathrm{rig}}^*(r)(\delta^{-1}))^{\varphi=1}\not=0$ 
for some ($r,\delta)$ corresponds to the exceptional zero case, for which we need some more additional arguments to 
obtain the (both local and global) similar results, which we will study in the next article \cite{NY}.

Here, we explicitly describe the $\varepsilon$-isomorphism and the de Rham $\varepsilon$-isomorphism under the above assumptions. To simplify the notation, we assume that a character $\delta$ as above takes values in $L^{\times}$ (of course, the general case follows from this case by extending scalars). For any such $\delta$, we fix an $L$-basis $\bold{f}_{\delta}=\alpha_{\delta}\bold{e}_{\delta}$ of $L_{\infty}(\delta)^{\Gamma}$ in such a way that $\alpha_{\delta}\cdot\alpha_{\delta^{-1}}=1$ holds
for any $\delta$.  Under the above assumption on the Hodge-Tate weights, one has $\mathrm{dim}_L\bold{D}^i_{\mathrm{dR}}(D_{\mathrm{rig}})=1$ if and only if $-(k-1)\leqq i\leqq 0$.
 We fix a basis $\{f_1,f_2\}$ of $\bold{D}_{\mathrm{dR}}(D_{\mathrm{rig}})$ over $L$ such that $f_1\in \bold{D}^0_{\mathrm{dR}}(D_{\mathrm{rig}})$. Since we assume that $D$ is 
 absolutely irreducible, we have 
 $$\mathrm{H}^0_{\varphi,\gamma}(D_{\mathrm{rig}}(-r)(\delta))=\mathrm{H}^2_{\varphi,\gamma}(D_{\mathrm{rig}}(-r)(\delta))=0, \,\, \mathrm{dim}_L\mathrm{H}^1_{\varphi,\gamma}(D_{\mathrm{rig}}(-r)(\delta))=2$$ 
 and the specialization map 
 $$\iota_{\chi^{-r}\delta}:D^{\psi=1}=(D^{\dagger})^{\psi=1}\rightarrow \mathrm{H}^1_{\varphi,\gamma}(D_{\mathrm{rig}}(-r)(\delta))$$
 is surjective for any pair $(r,\delta)$ as above (since the cokernel is contained in $D/(\psi-1)$ which is zero by the assumption). Hence, we obtain a canonical isomorphism $$\Delta_{L,1}(D_{\mathrm{rig}}(-r)(\delta))\isom (\mathrm{det}_{L}\mathrm{H}^1_{\varphi,\gamma}(D_{\mathrm{rig}}(-r)(\delta)), 2)^{-1}.$$ We also fix a base $\bold{e}_{D}$ of $\mathcal{L}_L(D)=\mathcal{L}_L(D_{\mathrm{rig}})$. 
 Since we have $$\mathrm{det}_L\bold{D}_{\mathrm{dR}}(D_{\mathrm{rig}})=\bold{D}_{\mathrm{dR}}(\mathrm{det}_{\mathcal{R}_L}D_{\mathrm{rig}})=(\frac{1}{t^k}L_{\infty}\bold{e}_{D})^{\Gamma},$$ there exists a unique $\Omega\in L_{\infty}^{\times}$ such that $f_1\wedge f_2=\frac{1}{\Omega\cdot t^k}\bold{e}_D.$

For any pair $(r,\delta)$ as above, we have 
$$\bold{D}^0_{\mathrm{dR}}(D_{\mathrm{rig}}(-r)(\delta))=Lf_1\otimes t^r\bold{e}_{-r}\otimes \bold{f}_{\delta},$$ 
and 
$$t_{D_{\mathrm{rig}}(-r)(\delta)}:=\bold{D}_{\mathrm{dR}}(D_{\mathrm{rig}}(-r)(\delta))/\bold{D}^0_{\mathrm{dR}}(D_{\mathrm{rig}}(-r)(\delta))=L\overline{f}_2\otimes t^r\bold{e}_{-r}\otimes \bold{f}_{\delta}$$ where $\overline{f}_2\in t_{D_{\mathrm{rig}}}$ is the image of  $f_2$. Then, the Bloch-Kato's exact sequence for $D_{\mathrm{rig}}$ becomes as follows
\begin{multline*}
0\rightarrow \bold{D}_{\mathrm{cris}}(D_{\mathrm{rig}}(-r)(\delta))\xrightarrow{(a)}\bold{D}_{\mathrm{cris}}(D_{\mathrm{rig}}(-r)(\delta))\oplus 
t_{D_{\mathrm{rig}}(-r)(\delta)}\xrightarrow{\mathrm{exp}_f\oplus \mathrm{exp}}\mathrm{H}^1_{\varphi,\gamma}(D_{\mathrm{rig}}(-r)(\delta))\\
\xrightarrow{\mathrm{exp}_f^{\vee}\oplus\mathrm{exp}^*}\bold{D}_{\mathrm{cris}}(D_{\mathrm{rig}}^*(r)(\delta^{-1}))^{\vee}\oplus\bold{D}^0_{\mathrm{dR}}(D_{\mathrm{rig}}(-r)(\delta))\xrightarrow{(b)} \bold{D}_{\mathrm{cris}}(D_{\mathrm{rig}}^*(r)(\delta^{-1}))\rightarrow 0,
\end{multline*}
where the map (a) is defined by $x\mapsto ((1-\varphi)x,\overline{x})$ and the map 
$\bold{D}_{\mathrm{cris}}(D^*(r)(\delta^{-1}))^{\vee}\rightarrow\bold{D}_{\mathrm{cris}}(D^*(r)(\delta^{-1}))^{\vee}$ in $(b)$ is the dual of $1-\varphi$.
Since $1-\varphi$ is isomorphism on $\bold{D}_{\mathrm{cris}}(D_{\mathrm{rig}}(-r)(\delta))$ and $\bold{D}_{\mathrm{cris}}(D_{\mathrm{rig}}^*(r)(\delta^{-1}))$ by the assumption (ii), $\mathrm{det}_{L}\mathrm{H}^1_{\varphi,\gamma}(D_{\mathrm{rig}}(-r)(\delta))$ has a basis 
of the form $y\wedge\mathrm{exp}(\overline{f}_2\otimes t^r\bold{e}_{-r}\otimes\bold{f}_{\delta})$ such that $\mathrm{exp}^*(y)\not=0$. 
Using these fixed datum, we define a map
\begin{equation}\label{alpha}
\alpha_{(r,\delta)}:D^{\psi=1}\xrightarrow{\iota_{\chi^{-r}\delta}}\mathrm{H}^1_{\varphi,\gamma}(D_{\mathrm{rig}}(-r)(\delta))\rightarrow L
\end{equation}
by the formula 
$$\mathrm{exp}^*(\iota_{\chi^{-r}\delta}(x)):=\alpha_{(r,\delta)}(x)f_1\otimes t^r\bold{e}_{-r}\otimes\bold{f}_{\delta}$$
for $x\in D^{\psi=1}$. 

For $x\in D^{\psi=1}$, we set $x_{-r,\delta}:=\iota_{\chi^{-r}\delta}(x)\in \mathrm{H}^1_{\varphi,\gamma}(D_{\mathrm{rig}}(-r)(\delta))$. Set 
$$L(D):=\mathrm{det}_L(1-\varphi|\bold{D}_{\mathrm{cris}}(D_{\mathrm{rig}}))$$ for any $D$.
Under this situation, the de Rham $\varepsilon$-isomorphism 
$\varepsilon_{L,\zeta}^{\mathrm{dR}}(D(-r)(\delta))=\varepsilon_{L,\zeta}^{\mathrm{dR}}(D_{\mathrm{rig}}(-r)(\delta))$ is the isomorphism defined as the composite of the following isomorphisms
\begin{multline*}
\varepsilon_{L,\zeta}^{\mathrm{dR}}(D(-r)(\delta)):
\bold{1}_L\xrightarrow{\theta_{1}(D_{\mathrm{rig}}(-r)(\delta))}
\Delta_{L,1}(D_{\mathrm{rig}}(-r)(\delta))\boxtimes \mathrm{Det}_L(\bold{D}_{\mathrm{dR}}(D_{\mathrm{rig}}(-r)(\delta)))\\
\xrightarrow{\mathrm{id}\boxtimes\theta_{2,\zeta}(D_{\mathrm{rig}}(-r)(\delta))}
\Delta_{L,1}(D_{\mathrm{rig}}(-r)(\delta))\boxtimes \Delta_{L,2}(D_{\mathrm{rig}}(-r)(\delta))
=\Delta_L(D_{\mathrm{rig}}(-r)(\delta))
\end{multline*}
where the isomorphisms $\theta_{1}(D_{\mathrm{rig}}(-r)(\delta))$ and $\theta_{2,\zeta}(D_{\mathrm{rig}}(-r)(\delta))$ are respectively induced by the isomorphisms
defined by (for $x\in D^{\psi=1}$ such that $\alpha_{(r,\delta)}(x)\not=0$)
\begin{multline*}
\theta_{1}(D_{\mathrm{rig}}(-r)(\delta)):\mathrm{det}_L\mathrm{H}_{\varphi,\gamma}^1(D_{\mathrm{rig}}(-r)(\delta))\isom \mathrm{det}_L\bold{D}_{\mathrm{dR}}(D_{\mathrm{rig}}(-r)(\delta)):x_{-r,\delta}\wedge \mathrm{exp}(f_2\otimes \bold{e}_{-r}\otimes \bold{f}_{\delta})\\
\mapsto 
\frac{1}{(k-r-1)!}\cdot\frac{r!}{(-1)^r}\cdot\frac{L(D(-r)(\delta))}{L(D^*(r)(\delta^{-1}))}\cdot\mathrm{exp}^*(x_{-r,\delta})\wedge (f_2\otimes t^r\bold{e}_{-r}\otimes\bold{f}_{\delta})
\end{multline*}
and 
\begin{multline*}
\theta_{2,\zeta}(D_{\mathrm{rig}}(-r)(\delta))^{-1}:\mathcal{L}_L(D_{\mathrm{rig}}(-r)(\delta))\isom \mathrm{det}_L
\bold{D}_{\mathrm{dR}}(D_{\mathrm{rig}}(-r)(\delta))\\
:\bold{e}_D\otimes\bold{e}_{-2r}\otimes\bold{e}^{\otimes 2}_{\delta}\mapsto \frac{1}{\varepsilon_L(W(D_{\mathrm{rig}}(-r)(\delta)),\zeta)}\cdot\frac{1}{t^{k-2r}}\cdot\bold{e}_D\otimes \bold{e}_{-2r}\otimes\bold{e}^{\otimes 2}_{\delta}.
\end{multline*}
Hence, using $\alpha_{(r,\delta)}$ and $\Omega$, the isomorphism $\eta_{\zeta}(D(-r)(\delta)):=\theta_{2,\zeta}(D_{\mathrm{rig}}(-r)(\delta))\circ\theta_{1}(D_{\mathrm{rig}}(-r)(\delta))$ is 
explicitly described as follows:
\begin{multline}\label{3e}
\eta_{\zeta}(D(-r)(\delta))(x_{-r,\delta}\wedge \mathrm{exp}(f_2\otimes t^r\bold{e}_{-r}\otimes\bold{f}_{\delta}))\\
= \frac{1}{(k-r-1)!}\cdot\frac{r!}{(-1)^r}\cdot \frac{L(D(-r)(\delta))}{L(D^*(r)(\delta^{-1}))}\cdot \frac{\varepsilon_L(W(D_{\mathrm{rig}}(-r)(\delta)),\zeta)\cdot\alpha_{\delta}^2}{\Omega}\cdot\alpha_{(r,\delta)}(x)\cdot\bold{e}_D\otimes \bold{e}_{-2r}\otimes\bold{e}^{\otimes 2}_{\delta}.
\end{multline}

We next consider the isomorphism $\varepsilon_{L,\zeta}(D(-r)(\delta))$. Let 
$$[-,-]_{\mathrm{dR}}: \bold{D}_{\mathrm{dR}}(D^*_{\mathrm{rig}}(r)(\delta^{-1}))\times \bold{D}_{\mathrm{dR}}(D_{\mathrm{rig}}(-r)(\delta))\rightarrow L$$ 
be the canonical dual pairing. We remarked in the proof of 
Proposition \ref{3.3.1} that, under the assumption that $D$ is absolutely irreducible, the natural map $1-\varphi:D^{\psi=1}\rightarrow D^{\psi=0}$ is injective, by which we identify $D^{\psi=1}$ with $\mathcal{C}(D)$ (and similarly for $D^*$), and one has $\omega_{\delta_D}(\mathcal{C}(D))\otimes_L \mathcal{L}_{L}(D)^{\vee}=\mathcal{C}(D^{\vee})$ under the canonical isomorphism $D\otimes_L\mathcal{L}_L(D)^{\vee}\isom D^{\vee}$. 
By this fact and the definition of $\varepsilon_{L,\zeta}(D(-r)(\delta))$, $\varepsilon_{L,\zeta}(D(-r)(\delta))$ is the isomorphism which is naturally induced by the isomorphism 
$$\eta'_{\zeta}(D(-r)(\delta)):\mathrm{det}_L\mathrm{H}_{\varphi,\gamma}^1(D_{\mathrm{rig}}(-r)(\delta))\isom 
\mathcal{L}_L(D_{\mathrm{rig}}(-r)(\delta))=
((L\bold{e}_D\otimes \bold{e}_{-2r}\otimes\bold{e}^{\otimes 2}_{\delta})^{\vee})^{\vee}$$
defined by the following formula (for $x\in D^{\psi=1}$ such that $\alpha_{(r,\delta)}(x)\not=0$)
\begin{multline}\label{4e}
\eta'_{\zeta}(D(-r)(\delta))(x_{-r,\delta}\wedge\mathrm{exp}(f_2\otimes t^r\bold{e}_{-r}\otimes\bold{f}_{\delta}))((\bold{e}_D\otimes \bold{e}_{-2r}\otimes\bold{e}^{\otimes 2}_{\delta})^{\vee})
\\
=\langle(\sigma_{-1}(\omega_{\delta_D}(x)\otimes\bold{e}_D^{\vee}\otimes \bold{e}_1))_{r,\delta^{-1}}, \mathrm{exp}(f_2\otimes t^r\bold{e}_{-r}\otimes\bold{f}_{\delta})\rangle_{\mathrm{Tate}}\\
=-[\mathrm{exp}^*((\sigma_{-1}(\omega_{\delta_D}(x)\otimes\bold{e}_D^{\vee}\otimes \bold{e}_1))_{r,\delta^{-1}}), f_2\otimes t^r\bold{e}_{-r}\otimes\bold{f}_{\delta}]_{\mathrm{dR}}=:(*),
\end{multline}
where the second equality follows from the definition of $\mathrm{exp}^*$ (see Proposition 2.16 of \cite{Na14a}). Using the canonical isomorphism 
$$\bold{D}_{\mathrm{dR}}^0(D_{\mathrm{rig}}^*(r)(\delta^{-1}))=Lf_1\otimes\Omega t^{k}\bold{e}_D^{\vee}\otimes \frac{1}{t^{r+1}}\bold{e}_{r+1}\otimes\bold{f}_{\delta^{-1}}$$ induced by the canonical isomorphism 
$D\otimes_L\mathcal{L}_{L}(D)^{\vee}\isom D^{\vee}$, 
we define a map 
\begin{equation}\label{beta}
\beta_{(r,\delta)}:D^{\delta_D(p)\psi=1}\xrightarrow{y\mapsto (\sigma_{-1}(y\otimes\bold{e}_D^{\vee}\otimes\bold{e}_1))_{r,\delta^{-1}}}\mathrm{H}^1_{\varphi,\gamma}(D_{\mathrm{rig}}^*(r)(\delta^{-1}))\rightarrow L
\end{equation} by the formula 
$$\mathrm{exp}^*((\sigma_{-1}(y\otimes\bold{e}_D^{\vee}\otimes \bold{e}_1))_{r,\delta^{-1}})
:=\beta_{(r,\delta)}(y)f_1\otimes\Omega t^{k}\bold{e}_D^{\vee}\otimes \frac{1}{t^{r+1}}\bold{e}_{r+1}\otimes\bold{f}_{\delta^{-1}}$$
for $y\in D^{\delta_D(p)\psi=1}$.
Using $\beta_{(r,\delta)}$, the last term $(*)$ in the equalities (\ref{4e}) is equal to 
\begin{multline} \label{44f}
(*)=-[\beta_{(r,\delta)}(w_{\delta_D}(x))f_1\otimes \Omega t^k\bold{e}_D^{\vee}\otimes \frac{1}{t^{r+1}}\bold{e}_{r+1}\otimes\bold{f}_{\delta^{-1}}, f_2\otimes t^r \bold{e}_{-r}\otimes\bold{f}_{\delta}]_{\mathrm{dR}}\\
=\beta_{(r,\delta)}(w_{\delta_D}(x)).
\end{multline}
We see from the formulae \eqref{4e} and \eqref{44f} that the isomorphism 
$$\eta_{\zeta}'(D(-r)(\delta)):\mathrm{det}_L\mathrm{H}^1_{\varphi,\gamma}(D_{\mathrm{rig}}(-r)(\delta))\isom \mathcal{L}_L(D_{\mathrm{rig}}(-r)(\delta))$$ is explicitly described as follows:
\begin{equation}\label{6e}
\eta_{\zeta}'(D(-r)(\delta))(x_{-r,\delta}\wedge \mathrm{exp}(f_2\otimes t^r\bold{e}_{-r}\otimes\bold{f}_{\delta}))
=\beta_{(r,\delta)}(w_{\delta_D}(x))\bold{e}_{D}\otimes \bold{e}_{-2r}
\otimes\bold{e}^{\otimes 2}_{\delta}.
\end{equation}

The formulae (\ref{3e}) and (\ref{6e}) show that the equality $\varepsilon_{L,\zeta}(D(-r)(\delta))
=\varepsilon_{L,\zeta}^{\mathrm{dR}}(D(-r)(\delta))$ follows from the following 
key proposition. Thus the proof of Theorem \ref{nontrianguline} is reduced to this proposition.
\begin{prop}\label{key}
For any $x\in D^{\psi=1}$ and any pair $(r,\delta)$ such that $0\leqq r\leqq k-1$ and $\delta:\Gamma\rightarrow L^{\times}$ with 
finite image, we have 
$$\beta_{(r,\delta)}(w_{\delta_D}(x))=\frac{1}{(k-r-1)!}\cdot\frac{r!}{(-1)^r}\cdot\frac{L(D(-r)(\delta))}{L(D^*(r)(\delta^{-1}))}\cdot \frac{\varepsilon_L(W(D_{\mathrm{rig}}(-r)(\delta)),\zeta)\cdot \alpha_{\delta}^2}{\Omega}\cdot \alpha_{(r,\delta)}(x).$$

\end{prop}

Of course, we have already proved this proposition for the trianguline case as a consequence of Corollary \ref{3.9.1}. 
In the rest of this subsection, we prove this key proposition in the non-trianguline case (the proof is given in the last part of this subsection). Our proof heavily depends on the Colmez's theory of Kirillov model of locally algebraic vectors $\Pi(D)^{\mathrm{alg}}$ of $\Pi(D)$ and the Emerton's theorem on the compatibility of the $p$-adic and the classical local Langlands correspondence, which we recall below. 
Since many arguments below are also valid for the trianguline case, we keep the assumptions (i) and (ii) on $D$ at the beginning of this subsection.

Set $G:=\mathrm{GL}_2(\mathbb{Q}_p)$, $B:=\left\{\begin{pmatrix}* & *\\ 0 & *\end{pmatrix}\in G\right\}$, $P:=\begin{pmatrix}\mathbb{Q}_p^{\times} & \mathbb{Q}_p \\ 
0 &1\end{pmatrix}$, $P^+:=\begin{pmatrix}\mathbb{Z}_p\setminus\{0\} & \mathbb{Z}_p\\ 0 & 1\end{pmatrix}$ and $Z:=\left\{\left.\begin{pmatrix} a & 0 \\ 0 & a\end{pmatrix}\right|a\in \mathbb{Q}_p^{\times}\right\}$.  We identify $Z$ with $\mathbb{Q}_p^{\times}$ via the isomorphism $\mathbb{Q}_p^{\times}\isom Z$ given by $a\mapsto \begin{pmatrix} a & 0 \\ 0 & a\end{pmatrix}$. 

Let us briefly recall the construction of the representation $\Pi(D)$ of $G$ for an absolutely irreducible $D$ of rank two (see \cite{Co10b} for details).  Let the monoid $P^+$ act on $D$ by the rule $\begin{pmatrix}p^na& b \\ 0 & 1\end{pmatrix}\cdot x:=(1+X)^b\cdot \varphi^n(\sigma_a(x))$ for $n\geqq 0$, $a\in \mathbb{Z}_p^{\times}$, $b\in\mathbb{Z}_p$ and $x\in D$. Using the involution $w_{\delta_D}:D^{\psi=0}\isom D^{\psi=0}$, we define a topological $L$-vector space
$$D\boxtimes_{\delta_D}\bold{P}^1:=\{(z_1,z_2)\in D\times D\ |\  w_{\delta_D}((1-\varphi\psi)z_1)
=(1-\varphi\psi)z_2\}$$
and an $L$-linear map 
$$\mathrm{Res}_{\mathbb{Z}_p}:D\boxtimes_{\delta_D}\bold{P}^1\rightarrow D:(z_1,z_2)\mapsto z_1.$$
By II of \cite{Co10b}, one can define a continuous action of $G$ on 
$D\boxtimes_{\delta_D}\bold{P}^1$ with the central character $\delta_D$ such that $\begin{pmatrix}0 & 1 \\ 1 & 0\end{pmatrix}\cdot (z_1,z_2)=(z_2,z_1)$ and the map $\mathrm{Res}_{\mathbb{Z}_p}$ is $P^+$-equivariant. We denote by $D\boxtimes\mathbb{Q}_p$ the topological $L$-vector space consisting of the 
sequences $(z_n)_{n\geqq 0}$ such that $\psi(z_{n+1})=z_n$ for all $n\geqq 0$.  One can define a continuous action of $P$ on  $D\boxtimes\mathbb{Q}_p$ by 
$$\begin{pmatrix}a& 0 \\ 0 & 1\end{pmatrix}\cdot (z_n)_{n\geqq 0}:=(\sigma_a(x_n))_{n\geqq 0}, \,\,\begin{pmatrix}p& 0 \\ 0 & 1\end{pmatrix}\cdot (x_n)_{n\geqq 0}:=(x_{n+1})_{n\geqq 0}$$
and $$\begin{pmatrix}1& b/p^k \\ 0 & 1\end{pmatrix}\cdot (z_n)_{n\geqq 0}:=(\psi^{k}((1+X)^{p^{n}b}\cdot z_{n+k}))_{n\geqq 0}$$ 
for $a\in \mathbb{Z}_p^{\times}, b\in \mathbb{Z}_p$ and $k\geqq 0$. 

Take an 
\'etale $(\varphi,\Gamma)$-submodule $D_0\subseteq D$ over $\mathcal{E}_{\mathcal{O}_L}$ such that $D_0[1/p]=D$. By \cite{Co10a}, there exists the smallest  $\psi$-stable compact $\mathcal{O}_L[[X]]$-submodules of $D_0$, which we denote by $D_0^{\natural}\subseteq D_0$. One also has 
the largest $\psi$-stable compact $\mathcal{O}_L[[X]]$-submodule $D_0^{\sharp}\subseteq D_0$ on which $\psi$ is surjective. We set $D^{\natural}:=D_0^{\natural}[1/p]$ and $D^{\sharp}:=D_0^{\sharp}[1/p]$, which are independent of the choice of $D_0$. We note that one has $D^{\natural}=D^{\sharp}$ under our assumption that $D$ is absolutely irreducible (Corollaire II.5.21 of \cite{Co10a}). One also has $(D^{\natural})^{\psi=1}=(D^{\sharp})^{\psi=1}=D^{\psi=1}$, where the second equality follows from Proposition II.5.6 of \cite{Co10a}. Define a sub $L[B]$-module $D^{\natural}\boxtimes_{\delta_D}\bold{P}^1$ of $D\boxtimes_{\delta_D}\bold{P}^1$ by 
$$D^{\natural}\boxtimes_{\delta_D}\bold{P}^1:=
\left\{z\in D\boxtimes_{\delta_D}\bold{P}^1\ \left|\ \mathrm{Res}_{\mathbb{Z}_p}\left(\begin{pmatrix}p^n & 0 \\ 0 &1 \end{pmatrix}\cdot z\right)\in D^{\natural} \text{ for all } n\geqq 0\right.\right\}.$$
One of the deepest results in the theory of the $p$-adic local Langlands correspondence for $\mathrm{GL}_2(\mathbb{Q}_p)$ is that 
the pair $(D,\delta_D)$ is $G$-compatible, which means that $D^{\natural}\boxtimes_{\delta_D}\bold{P}^1$ is stable under the action of $G$ (Th\'eor\`eme II.3.1 of \cite{Co10b}, Proposition 10.1 of \cite{CDP14a}). Finally, one defines 
$$\Pi(D):=D\boxtimes_{\delta_D}\bold{P}^1/D^{\natural}\boxtimes_{\delta_D}\bold{P}^1$$ which is a topologically irreducible unitary $L$-Banach admissible representation of $G$.

We next recall in detail the Colmez's theory of the Kirillov model of the locally algebraic vectors $\Pi(D)^{\mathrm{alg}}$ of $\Pi(D)$. We set $L_{\infty}[[t]]:=\bigcup_{n\geqq 1}L_n[[t]]$. For the fixed $\zeta=\{\zeta_{p^n}\}_{n\geqq 1}\in \mathbb{Z}_p(1)$, we define a homomorphism 
$$[\overline{\zeta}^{(-)}]:\mathbb{Q}_p^{\times}\rightarrow ((\widetilde{\bold{B}}^{+})^{\times})^{H_{\mathbb{Q}_p}}:a\mapsto [\overline{\zeta}^a].$$ For $V:=V(D)$, we 
set 
$$\widetilde{D}^+:=(V\otimes_{\mathbb{Q}_p}\widetilde{\bold{B}}^+)^{H_{\mathbb{Q}_p}}, \,\,\widetilde{D}:=(V\otimes_{\mathbb{Q}_p}\widetilde{\bold{B}})^{H_{\mathbb{Q}_p}} \text{ and }
 \widetilde{D}^+_{\mathrm{dif}}:=(V\otimes_{\mathbb{Q}_p}\bold{B}^+_{\mathrm{dR}})^{H_{\mathbb{Q}_p}}.$$ One has a canonical isomorphism 
 $$\widetilde{D}^+_{\mathrm{dif}}\isom\bold{D}^+_{\mathrm{dif},\infty}(D_{\mathrm{rig}})\otimes_{\mathbb{Q}_{p, \infty}[[t]]}(\bold{B}_{\mathrm{dR}}^{+})^{ H_{\mathbb{Q}_p}}.$$ The natural inclusion $\iota_0:\widetilde{\bold{B}}^+\hookrightarrow \bold{B}^+_{\mathrm{dR}}$ induces a canonical $\Gamma$-equivariant inclusion 
 $$\iota_0:\widetilde{D}^+\hookrightarrow \widetilde{D}^+_{\mathrm{dif}}.$$ The group $B$ acts on 
 $\widetilde{D}^+$, $\widetilde{D}$ and $\widetilde{D}/\widetilde{D}^+$ by the rule
 (for $z\in  \widetilde{D}^+, \widetilde{D}, \widetilde{D}/\widetilde{D}^+$)
  $$\begin{pmatrix}a& 0 \\ 0 & a\end{pmatrix}\cdot z:=\delta_D(a)z, \,\,
 \begin{pmatrix}p^nb& 0 \\ 0 & 1\end{pmatrix}\cdot z:=\varphi^n(\sigma_b(z)),$$
 $$\begin{pmatrix}1& c \\ 0 & 1\end{pmatrix}\cdot z:=[\overline{\zeta}^{c}] z$$
 for $a\in \mathbb{Q}_p^{\times}$, $b\in \mathbb{Z}_p^{\times}$, $n\in \mathbb{Z}$ and $c\in \mathbb{Q}_p$. 
  
 We denote by $$\mathrm{LP}\left(\mathbb{Q}_p^{\times}, \frac{1}{t^k}\widetilde{D}^+_{\mathrm{dif}}/\widetilde{D}^+_{\mathrm{dif}}\right)^{\Gamma}$$ the $L$-vector space consisting of functions $\phi:\mathbb{Q}_p^{\times}\rightarrow \frac{1}{t^k}\widetilde{D}^+_{\mathrm{dif}}/\widetilde{D}^+_{\mathrm{dif}}$ 
 such that the support is compact in $\mathbb{Q}_p$ (i.e. $\phi(\frac{1}{p^n}\mathbb{Z}_p^{\times})=0$ for any sufficiently large $n$) and 
 $\sigma_a(\phi(x))=\phi(ax)$ for any $a\in \mathbb{Z}_p^{\times}$ and $x\in \mathbb{Q}_p^{\times}$. 
 We equip this space with an action of $B$ by 
 $$\left(\begin{pmatrix}a& 0 \\ 0 & a\end{pmatrix}\cdot \phi\right)(x):=\delta_D(a)\phi(x), \,\,
 \left(\begin{pmatrix}a& 0 \\ 0 & 1\end{pmatrix}\cdot \phi\right)(x):=\phi(ax),$$
 $$
 \left(\begin{pmatrix}1& b \\ 0 & 1\end{pmatrix}\cdot \phi\right)(x):=\iota_0([\overline{\zeta}^{bx}])\cdot\phi(x)$$
 for $a\in \mathbb{Q}_p^{\times}$ and $b\in \mathbb{Q}_p$. Remark that, for $a=\frac{b}{p^n}\in \mathbb{Q}_p^{\times}$ such that $b\in \mathbb{Z}_p$, $n\geqq 0$, one has $\iota_0([\overline{\zeta}^a])=\zeta_{p^n}^b\mathrm{exp}(at)\in L_{\infty}[[t]]^{\times}$.

 For $z\in \bigcup_{n\geqq 0}\frac{1}{\varphi^n(X)^k}\widetilde{D}^+/\widetilde{D}^+$ (this is a $B$-stable subspace of $\widetilde{D}/\widetilde{D}^+$), define a function $\phi_z\in \mathrm{LP}\left(\mathbb{Q}_p^{\times}, \frac{1}{t^k}\widetilde{D}^+_{\mathrm{dif}}/\widetilde{D}^+_{\mathrm{dif}}\right)^{\Gamma}$ by 
 $$\phi_z(x):=\iota_0\left(\begin{pmatrix} x& 0 \\ 0 & 1\end{pmatrix}\cdot z\right)$$
 for $x\in \mathbb{Q}_p^{\times}$.  By Lemme VI.5.4 (i) of \cite{Co10b}, this correspondence induces a $B$-equivariant inclusion 
 $$\bigcup_{n\geqq 0}\frac{1}{\varphi^n(X)^k}\widetilde{D}^+/\widetilde{D}^+\hookrightarrow 
 \mathrm{LP}\left(\mathbb{Q}_p^{\times}, \frac{1}{t^k}\widetilde{D}^+_{\mathrm{dif}}/\widetilde{D}^+_{\mathrm{dif}}\right)^{\Gamma}:z\mapsto \phi_z.$$
 
  Let us write $$\bold{N}^+_{\mathrm{dif},*}(D_{\mathrm{rig}}):=\bold{D}_{\mathrm{dR}}(D_{\mathrm{rig}})\otimes_LL_{*}[[t]]$$ for $*=n\geqq n(D_{\mathrm{rig}})$ or $*=\infty$. We set 
   $$X_{\infty}^{-}:=\bold{N}^+_{\mathrm{dif},\infty}(D)/\bold{D}^+_{\mathrm{dif},\infty}(D).$$
Since $\bold{D}^+_{\mathrm{dif},\infty}(D_{\mathrm{rig}})=L_{\infty}[[t]]f_1\oplus L_{\infty}[[t]](t^kf_2)$, one has 
  $$X_{\infty}^{-}=(L_{\infty}[[t]]/t^kL_{\infty}[[t]])\otimes_LL\overline{f}_2\subseteq \frac{1}{t^k}\widetilde{D}^+_{\mathrm{dif}}/\widetilde{D}^+_{\mathrm{dif}}.$$ We denote by 
  $$\mathrm{LP}(\mathbb{Q}_p^{\times}, X_{\infty}^-)^{\Gamma}$$ 
  the $B$-stable $L$-subspace of $\mathrm{LP}\left(\mathbb{Q}_p^{\times}, \frac{1}{t^k}\widetilde{D}^+_{\mathrm{dif}}/\widetilde{D}^+_{\mathrm{dif}}\right)^{\Gamma}$ consisting of functions 
 $\phi$ with values in $X^-_{\infty}$, in other words, consisting of functions $\phi:\mathbb{Q}_p^{\times}\rightarrow X^-_{\infty}:x\mapsto \sum_{
i=0}^{k-1}\phi_i(x)(xt)^{i}\otimes \overline{f}_2$ such that, for any $0\leqq i\leqq k-1$,  the function 
$\phi_i:\mathbb{Q}_p^{\times}\rightarrow L_{\infty}$ is locally constant with
 compact support in $\mathbb{Q}_p$ and $\phi_i(ax)=\sigma_a(\phi_i(x))$ for any $a\in \mathbb{Z}_p^{\times}$ and $x\in \mathbb{Q}_p^{\times}$. We denote by  $$\mathrm{LP}_{\mathrm{c}}(\mathbb{Q}_p^{\times}, X_{\infty}^-)^{\Gamma}$$ the $B$-stable $L$-subspace of $\mathrm{LP}(\mathbb{Q}_p^{\times}, X_{\infty}^-)^{\Gamma}$ consisting of functions $\phi$ with compact support in $\mathbb{Q}_p^{\times}$, i.e. $\phi_i(p^{\pm n}\mathbb{Z}_p^{\times})=0$ for 
 sufficiently large $n$.

 By Corollaire  II.2.9 (ii) of \cite{Co10b}, one has a canonical $B$-equivariant topological isomorphism 
 $$\widetilde{D}/\widetilde{D}^+\isom \Pi(D)$$ (under the assumption that $D$ is absolutely irreducible), by which we identify the both sides with each other. We denote by $\Pi(D)^{\mathrm{alg}}$  the $G$-stable $L$-subspace of $\Pi(D)$ consisting of locally algebraic vectors, which is non zero due to Th\'eor\`eme VI.6.18 of \cite{Co10b}. By Lemme VI.5.3, Corollaire VI.5.9 of \cite{Co10b}, one has $$\Pi(D)^{\mathrm{alg}}\subseteq \bigcup_{n\geqq 0}\frac{1}{\varphi^n(X)^k}\widetilde{D}^+/\widetilde{D}^+,$$ and the map $z\mapsto \phi_z$ defined above induces a $B$-equivariant injection
  \begin{equation}\label{8s}
  \Pi(D)^{\mathrm{alg}}\hookrightarrow \mathrm{LP}(\mathbb{Q}_p^{\times}, X^-_{\infty})^{\Gamma}
  \end{equation}
  whose image contains 
 $\mathrm{LP}_{\mathrm{c}}(\mathbb{Q}_p^{\times}, X^-_{\infty})^{\Gamma}$. Hence, if we write $\Pi(D)_{\mathrm{c}}^{\mathrm{alg}}\subseteq \Pi(D)^{\mathrm{alg}}$ for the inverse image of $\mathrm{LP}_{\mathrm{c}}(\mathbb{Q}_p^{\times}, X^-_{\infty})^{\Gamma}$, then one obtains a $B$-equivariant  isomorphism 
 \begin{equation}\label{8e}
 \Pi(D)_{\mathrm{c}}^{\mathrm{alg}}\isom \mathrm{LP}_{\mathrm{c}}(\mathbb{Q}_p^{\times}, X^-_{\infty})^{\Gamma}.
 \end{equation}
 
We denote by $$\mathrm{LC}_{\mathrm{c}}(\mathbb{Q}_p,L_{\infty})^{\Gamma}$$ the $L$-vector space consisting of locally constant 
functions $\phi:\mathbb{Q}_p^{\times}\rightarrow L_{\infty}$ such that the support of $\phi$ is compact in $\mathbb{Q}_p^{\times}$ and that $\sigma_a(\phi(x))=\phi(ax)$ for any $a\in \mathbb{Z}_p^{\times}$ and $x\in \mathbb{Q}_p^{\times}$. We similarly define an action of $B$ on this space by the rule
$$\left(\begin{pmatrix}a& 0 \\ 0 & a\end{pmatrix}\cdot \phi\right)(x):=\frac{\delta_D(a)}{a^{k-1}}\phi(x), \,\,
 \left(\begin{pmatrix}a& 0 \\ 0 & 1\end{pmatrix}\cdot \phi\right)(x):=\phi(ax),$$
 $$
 \left(\begin{pmatrix}1& b \\ 0 & 1\end{pmatrix}\cdot \phi\right)(x):=\psi_{\zeta}(bx)\cdot\phi(x)$$
 for $a\in \mathbb{Q}_p^*$ and $b\in \mathbb{Q}_p$, where $\psi_{\zeta}:\mathbb{Q}_p 
 \rightarrow L_{\infty}^{\times}$ is the additive character associated to $\zeta$  (i.e. we define 
 $\psi_{\zeta}(a):=\zeta_{p^n}^b\in L_{\infty}^{\times}$
 for $a=\frac{b}{p^n}\in \mathbb{Q}_p$ with $b\in \mathbb{Z}_p$ and $n\geqq 0$). Let $\mathrm{Sym}^{k-1}L^2$ be 
 the $(k-1)$-th symmetric power of the standard representation $L^2$ of $G$, i.e. 
 $\mathrm{Sym}^{k-1}L^2:=\oplus_{i=0}^{k-1}Le_1^{i}e_2^{k-1-i}$ on which $G$ acts by 
 $$\begin{pmatrix}a& b \\ c & d\end{pmatrix}\cdot e_1^ie_2^{k-1-i}:=(ae_1+ce_2)^i
 \cdot (be_1+de_2)^{k-1-i}.$$ 
 Then, one has a canonical (up to the choice of $f_2$) $B$-equivariant isomorphism 
 \begin{multline}\label{9e}
 \mathrm{LC}_{\mathrm{c}}(\mathbb{Q}_p^{\times},L_{\infty})^{\Gamma}
 \otimes_L \mathrm{Sym}^{k-1}L^2\isom \mathrm{LP}_{\mathrm{c}}(\mathbb{Q}_p^{\times}, 
 X^-_{\infty})^{\Gamma}:\\
 \sum_{i=0}^{k-1}\phi_i\otimes e_1^ie_2^{k-1-i}\mapsto 
 \left[x\mapsto \sum_{i=0}^{k-1}(k-1-i)!\cdot\phi_i(x)(xt)^i\otimes \overline{f}_2\right].
 \end{multline}
 
 Therefore, as the composite of isomorphisms (\ref{8e}) and (\ref{9e}), 
 one obtains a $B$-equivariant isomorphism 
 \begin{equation}\label{10e}
 \Pi(D)^{\mathrm{alg}}_{\mathrm{c}}\isom \mathrm{LC}_{\mathrm{c}}(\mathbb{Q}_p^{\times}, L_{\infty})^{\Gamma}
 \otimes_L\mathrm{Sym}^{k-1}L^2.
 \end{equation}
 
 We next consider the quotient $$J(\Pi(D)^{\mathrm{alg}}):=\Pi(D)^{\mathrm{alg}}/\Pi(D)^{\mathrm{alg}}_{\mathrm{c}}.$$ Set $$D_{\mathrm{crab}}:=\bigcup_{n\geqq 0}D_{\mathrm{rig}}[1/t]^{\Gamma_n},$$ which is a finite dimensional $L$-vector space equipped with the action of $\varphi$ and (a smooth action of) $\Gamma$. Then, $D$ is trianguline if and only if $D_{\mathrm{crab}}\not =0$. If $D$ is trianguline, then $\mathrm{dim}_LD_{\mathrm{crab}}=2$ if $D$ is potentially crystalline, and $\mathrm{dim}_LD_{\mathrm{crab}}=1$ if $D$ is potentially semi-stable with $N\not=0$ on $W(D_{\mathrm{rig}})$. For a locally constant homomorphism $\delta:\mathbb{Q}_p^{\times}\rightarrow L^{\times}$, define an action of $\varphi$ and $\Gamma$ on $L(\delta):=L\bold{e}_{\delta}$ by $\varphi(\bold{e}_{\delta}):=\delta(p), \sigma_a(\bold{e}_{\delta}):=\delta(a)\bold{e}_{\delta}$ for $a\in \mathbb{Z}_p^{\times}$. If $D$ is potentially crystalline, then $D_{\mathrm{crab}}=L(\delta_1)\oplus L(\delta_2)$, or $D_{\mathrm{crab}}$ is the nontrivial extension $0\rightarrow 
L(\delta_1)\rightarrow D_{\mathrm{crab}}\rightarrow L(\delta_1)\rightarrow 0$ for some characters
$\delta_1,\delta_2$. If we define a $B$-action on $D_{\mathrm{crab}}$ by 
$\begin{pmatrix}p^na & 0 \\ 0 &1\end{pmatrix}x:=\varphi^n(\sigma_a(x))$ for $n\in \mathbb{Z}, a\in \mathbb{Z}_p^{\times}$, $\begin{pmatrix}b & 0 \\ 0 &b\end{pmatrix}x:=\frac{\delta_D(b)}{b^{k-1}}x$ 
for $b\in \mathbb{Q}_p^{\times}$ and $\begin{pmatrix}1 & c\\ 0 &1\end{pmatrix}x:=x$ for 
any $c\in \mathbb{Q}_p$, then one has a $B$-equivariant isomorphism 
\begin{equation}\label{Jacquet}
D_{\mathrm{crab}}\otimes_L\mathrm{Sym}^{k-1}L^2\isom J(\Pi(D)^{\mathrm{alg}})
\end{equation} by Th\'eor\`eme VI.6.30 of \cite{Co10b}. In particular, $D$ is non-trianguline if and only if one has
$$\Pi(D)^{\mathrm{alg}}_{\mathrm{c}}=\Pi(D)^{\mathrm{alg}}.$$

Using the map $z\mapsto \phi_z$, we define a 
$\Gamma\isom\begin{pmatrix}\mathbb{Z}_p^{\times} & 0 \\ 0 &1\end{pmatrix}$-equivariant map
$$\iota_i^{-}: \Pi(D)^{\mathrm{alg}}\rightarrow X_{\infty}^-:
z\mapsto \phi_z(p^{-i})(=\iota_0(\varphi^{-i}(z)))$$ 
for each $i\in \mathbb{Z}$.

Set $$X_n^+:=\bold{D}^+_{\mathrm{dif},n}(D_{\mathrm{rig}})/t^k\bold{N}^+_{\mathrm{dif},n}(D_{\mathrm{rig}})\isom (L_n[[t]]/t^kL_n[[t]])\otimes_LL\overline{f}_1$$ for each $n\geqq n(D_{\mathrm{rig}})$, and set 
$$X^+\boxtimes \mathbb{Q}_p:=\varprojlim_{n}X_n^+$$ where the transition maps are the maps induced by $$\frac{1}{p}\mathrm{Tr}_{L_{n+1}/L_{n}}:L_{n+1}((t))f_i\rightarrow L_{n}((t))f_i:\sum_{m\in\mathbb{Z}}a_mt^mf_i\mapsto \sum_{m\in\mathbb{Z}}\frac{1}{p}\mathrm{Tr}_{L_{n+1}/L_{n}}(a_m)t^mf_i$$ for $i=1,2$. Set $g_p:=\begin{pmatrix} p & 0\\ 0 &1\end{pmatrix}$.
For each $i\in \mathbb{Z}$ and $n\geqq n(D_{\mathrm{rig}})$, define a $\Gamma\isom \begin{pmatrix}\mathbb{Z}_p^{\times} & 0 \\ 0 &1\end{pmatrix}$-equivariant map
$$\iota^+_{i,n}:D^{\natural}\boxtimes_{\delta_D}\bold{P}^1\rightarrow X^+_n:z\mapsto \iota_n(\mathrm{Res}_{\mathbb{Z}_p}(g_p^{n-i}\cdot z))
\in X_n^+,$$
where $\iota_n:D_{\mathrm{rig}}^{(n)}\hookrightarrow \bold{D}_{\mathrm{dif},n}^+(D_{\mathrm{rig}})$ is the canonical map (remark that we have $D^{\natural}\subseteq D_{\mathrm{rig}}$ by Corollaire II.7.2 of \cite{Co10a}), which also induces a $\Gamma$-equivariant map
$$\iota_i^+:D^{\natural}\boxtimes \bold{P}^1\rightarrow X_{\infty}^+:z\mapsto (\iota^+_{i,n}(z))_{n\geqq n(D_{\mathrm{rig}})}.$$

Let $$\langle-,-\rangle:D^*\times D\rightarrow \mathcal{E}_L(1)$$ be the canonical $\mathcal{E}_L$-bilinear pairing. Since we have $\bold{D}^+_{\mathrm{dif},n}(\mathcal{R}_L(1))=L_n[[t]]\bold{e}_1$, this pairing also induces an $L_n((t))$-bilinear pairing 
$$\langle-,-\rangle:\bold{D}_{\mathrm{dif},n}(D_{\mathrm{rig}}^*)\times \bold{D}_{\mathrm{dif},n}(D_{\mathrm{rig}})\rightarrow 
L_n((t))\bold{e}_1,$$
by which we identify $\bold{D}^+_{\mathrm{dif},n}(D_{\mathrm{rig}}^*)$ with $\mathrm{Hom}_{L_n[[t]]}(\bold{D}^+_{\mathrm{dif},n}(D_{\mathrm{rig}}),L_n[[t]]\bold{e}_1)$. Then, using the canonical isomorphism 
$\bold{D}^+_{\mathrm{dif},n}(\mathrm{det}_{\mathcal{R}_L}D_{\mathrm{rig}})\isom
\mathcal{L}_{L}(D)\otimes_LL_n[[t]]$, we define a canonical isomorphism
\begin{equation}
\bold{D}^+_{\mathrm{dif},n}(D_{\mathrm{rig}})\otimes_{L}\mathcal{L}_{L}(D)^{\vee}\otimes_LL(1)\isom 
\bold{D}^+_{\mathrm{dif},n}(D_{\mathrm{rig}}^*)
: x\otimes z\otimes \bold{e}_1
\mapsto [y\mapsto z(y\wedge x)\bold{e}_1].
\end{equation} 
Using this isomorphism and the fixed basis $\bold{e}_D\in \mathcal{L}_L(D)$, we define a pairing
$$[-,-]_{\mathrm{dif}}:\bold{D}_{\mathrm{dif},n}(D_{\mathrm{rig}})\times \bold{D}_{\mathrm{dif},n}(D_{\mathrm{rig}}):(x,y)\mapsto \mathrm{res}_L(\langle \sigma_{-1}(x\otimes \bold{e}_D^{\vee}\otimes \bold{e}_1), y\rangle),$$
where $\mathrm{res}_L$ is the map
$$\mathrm{res}_L:L_{n}((t))\bold{e}_1\rightarrow L: 
\sum_{m\in\mathbb{Z}}a_mt^m\bold{e}_1\mapsto \frac{1}{[\mathbb{Q}_{p}(\zeta_{p^n}):\mathbb{Q}_p]}\cdot \mathrm{Tr}_{L_n/L}(a_{-1}).$$ We remark that one has
$$[\sigma_a(x),\sigma_a(y)]_{\mathrm{dif}}=
\delta_D(a)[x,y]_{\mathrm{dif}}$$ for any $a\in \mathbb{Z}_p^{\times}$. This pairing also induces a pairing
$$[-,-]_{\mathrm{dif}}:X^+_{n}\times X^-_n\rightarrow L, $$ 
and,  by taking limits, one also obtains a pairing 
$$[-,-]_{\mathrm{dif}}:X^+\boxtimes\mathbb{Q}_p \times 
X^-_{\infty}\rightarrow L.$$ 

Similarly, using the canonical isomorphism $D\otimes_L\mathcal{L}_L(D)^{\vee}\otimes_LL(1)\isom D^*$, we define a pairing
$$[-,-]:D\times D\rightarrow L 
: (x, y)\mapsto \mathrm{res}_0(\langle\sigma_{-1}(x\otimes \bold{e}_D^{\vee}\otimes \bold{e}_1), y\rangle)$$
using the residue map 
$$\mathrm{res}_0:\mathcal{E}_L(1)\rightarrow L:f(X)\bold{e}_1\mapsto \mathrm{Res}_{X=0}\left(\frac{f(X)}{(1+X)}\right).$$
 This pairing also induces a pairing 
\begin{equation*}
[-,-]_{\bold{P}^1}:D\boxtimes_{\delta_D}\bold{P}^1\times 
D\boxtimes_{\delta_D} \bold{P}^1\rightarrow L
:((z_1,z_2), (z'_1,z'_2))\mapsto [z_1,z'_1]+[\varphi\psi(z_2),\varphi\psi(z_2')],
\end{equation*}
which satisfies $$[gx,gy]_{\bold{P}^1}=\delta_D(\mathrm{det}(g))[x,y]_{\bold{P}^1}$$ for any 
$x,y\in D\boxtimes_{\delta_D}\bold{P}^1$ and $g\in G$ by Th\'eor\`eme II.1.13 of \cite{Co10b}. 
By Th\'eor\`eme II.3.1 of \cite{Co10b}, this pairing $[-,-]_{\bold{P}^1}$ satisfies that $[x, y]_{\bold{P}^1}=0$ for any $x,y\in D^{\natural}\boxtimes_{\delta_D} \bold{P}^1$ and induces a $G$-equivariant topological isomorphism 
\begin{equation}\label{DD}
D^{\natural}\boxtimes_{\delta_D} \bold{P}^1\isom \Pi(D)^{\vee}\otimes_L(\delta_D\circ \mathrm{det}):x\mapsto [y\in \Pi(D) \mapsto [x,y]_{\bold{P}^1}],
\end{equation}
where we set $\Pi(D)^{\vee}:=\mathrm{Hom}^{\mathrm{cont}}_L(\Pi(D), L)$.
Moreover, one also has the following proposition.
\begin{prop}\label{Co}$(\mathrm{Proposition}\,\mathrm{VI}.5.12.(\mathrm{ii}) \,\mathrm{of}$ \cite{Co10b}$)$
For any $x\in D^{\natural}\boxtimes_{\delta_D}\bold{P}^1$ and $y\in \Pi(D)^{\mathrm{alg}}_{\mathrm{c}}$, one has
$$[x,y]_{\bold{P}^1}=\sum_{i\in\mathbb{Z}}\delta_D(p^i)[\iota_i^+(x), \iota^-_i(y)]_{\mathrm{dif}}.$$

\end{prop}

Using these preliminaries, we prove two propositions below (Proposition \ref{formula1} and Proposition \ref{formula2}) which explicitly describe the maps 
$\alpha_{(r,\delta)}$ and $\beta_{(r,\delta)}$ introduced in \eqref{alpha} and \eqref{beta} in terms of the pairing $[-,-]_{\bold{P}^1}$. 

Since $D$ is (assumed to be) absolutely irreducible, one has $D^{\natural}=D^{\sharp}$, and then one has a natural $P$-equivariant isomorphism 
$$D^{\natural}\boxtimes_{w_{\delta_D}}\bold{P}^1\isom D^{\natural}\boxtimes\mathbb{Q}_p:
z\mapsto \left(\mathrm{Res}_{\mathbb{Z}_p}(g_p^n\cdot z)\right)_{n\geqq 0}.$$
Hence, the inverse of the natural isomorphism 
$$D^{\psi=1}=(D^{\sharp})^{\psi=1}=(D^{\natural})^{\psi=1}
\isom(D^{\natural}\boxtimes\mathbb{Q}_p)^{g_p=1}:
z\mapsto (z_n)_{n\geqq 0},$$ 
where $z_n:=z$ for any $n$, induces an isomorphism 
$$ (D^{\natural}\boxtimes_{w_{\delta_D}}\bold{P}^1)^{g_p=1}\isom D^{\psi=1}: z\mapsto \mathrm{Res}_{\mathbb{Z}_p}(z).$$ 
For $x\in D^{\psi=1}$, we denote by $\widetilde{x}\in 
(D^{\natural}\boxtimes_{w_{\delta_D}}\bold{P}^1)^{g_p=1}$ 
the element which corresponds to $x$ via the last isomorphism.
 
For any $m\in \mathbb{Z}$ and any $(r,\delta)$ such that $0\leqq r\leqq k-1$ and $\delta:\Gamma\rightarrow L^{\times}$ with finite image, define 
a function 
$$f_{(r,\delta),m}\in \mathrm{LP}_c(\mathbb{Q}_p^{\times}, X^-_{\infty})^{\Gamma}$$ by 
$$f_{(r,\delta),m}(p^n\cdot a):=\begin{cases}\sigma_a(\alpha^{-1}_{\delta}\cdot\Omega t^{k-1-r})\cdot \overline{f}_2 & \text{ if } n=m\\
0 & \text{ if } n\not=m\end{cases}$$for $n\in \mathbb{Z}$ and $a\in \mathbb{Z}_p^{\times}$. Since we have $\frac{1}{\Omega t^k}\bold{e}_D\in \bold{D}_{\mathrm{dR}}(\mathrm{det}_{\mathcal{R}_L}D_{\mathrm{rig}})$, we have $\sigma_a(\Omega)=\frac{\mathrm{det}_{\mathcal{E}_L}D(\sigma_a)}{a^k}$ for any $a\in \mathbb{Z}_p^{\times}$. Hence, we have $$\sigma_a(\alpha^{-1}_{\delta}\cdot\Omega t^{k-1-r})\cdot \overline{f}_2=\frac{\delta(a)\cdot\mathrm{det}D(\sigma_a)}{a^k}\cdot(\alpha^{-1}_{\delta}\Omega)\cdot(at)^{k-1-r}\cdot \overline{f}_2.$$ 

\begin{prop}\label{formula1}
For any $x\in D^{\psi=1}$, $m\in \mathbb{Z}$ and 
$(r,\delta)$ such that $0\leqq r\leqq k-1$ and $\delta:\Gamma\rightarrow L^{\times}$ with finite image, we have 

$$\alpha_{(r,\delta)}(x)=-\delta_D(p^m)\cdot \frac{p-1}{p}\cdot (-1)^r\delta(-1)\delta_D(-1)\cdot[\widetilde{x}, f_{(r,\delta),m}]_{\bold{P}^1}.$$

\end{prop}
\begin{proof}
Since $\tilde{x}$ is fixed by $g_p$, for any $i\in \mathbb{Z}$ and $n\geqq n(D_{\mathrm{rig}})$, we have 
$$\iota^+_{i,n}(\widetilde{x})=\iota_n\left(\mathrm{Res}_{\mathbb{Z}_p}(g_p^{n-i}\cdot\widetilde{x})\right)=\iota_n(\mathrm{Res}_{\mathbb{Z}_p}(\widetilde{x}))=\iota_n(x).$$ Hence, we have $$\iota^+_{i}(\widetilde{x})=(\iota^+_{i,n}(\tilde{x}))_{n\geqq n(D_{\mathrm{rig}})}=(\iota_n(x))_{n\geqq n(D_{\mathrm{rig}})}\in X^+\boxtimes\mathbb{Q}_p.$$ Then by Proposition \ref{Co} we have
\begin{multline*}
[\widetilde{x}, f_{(r,\delta),m}]_{\bold{P}^1}=\sum_{i\in \mathbb{Z}}
\delta_D(p^i)[(\iota_n(x))_{n\geqq n(D_{\mathrm{rig}})}, f_{(r,\delta),m}(p^{-i})]_{\mathrm{dif}}\\
=\delta_D(p^{-m})[(\iota_n(x))_{n\geqq n(D_{\mathrm{rig}})}, \alpha^{-1}_{\delta}\cdot\Omega t^{k-1-r}\cdot f_2]_{\mathrm{dif}}=\delta_D(p^{-m})[\iota_n(x), \alpha^{-1}_{\delta}\cdot\Omega t^{k-1-r}\cdot f_2]_{\mathrm{dif}}
\end{multline*}
for any $n\geqq n(D_{\mathrm{rig}})$.

For an $L[\Gamma]$-module $N$, we set $\mathrm{H}^1_{\gamma}(N):=N^{\Delta}/(\gamma-1)N^{\Delta}$ using the fixed $\Delta\subseteq \Gamma_{\mathrm{tor}}$ and 
$\gamma\in \Gamma$ in \S 2.2.

By Proposition 2.16 of \cite{Na14a}, one has a commutative  diagram
\begin{equation*}
 \begin{CD}
  \mathrm{H}^1_{\varphi,\gamma}(D_{\mathrm{rig}}(-r)(\delta))@>>>
  \mathrm{H}^1_{\gamma}(\bold{D}_{\mathrm{dif},\infty}(D_{\mathrm{rig}}(-r)(\delta)))\\
 @ V \mathrm{id} VV @AA x\mapsto \mathrm{log}(\chi(\gamma))[x] A\\
  \mathrm{H}^1_{\varphi,\gamma}(D_{\mathrm{rig}}(-r)(\delta))@> \mathrm{exp}^*>>\bold{D}_{\mathrm{dR}}(D_{\mathrm{rig}}(-r)(\delta)),
     \end{CD}
 \end{equation*}
 where the upper horizontal arrow is the map defined by $[(x,y)]\mapsto [\iota_n(x)]$ for any sufficiently large $n\geqq n(D_{\mathrm{rig}})$ (which is independent of $n$). We remark that the right vertical arrow is isomorphism since $D_{\mathrm{rig}}(-r)(\delta)$ is de Rham. We have 
$$[\iota_n(x_{-r,\delta})]=[\mathrm{log}(\chi(\gamma))\cdot \alpha_{(r,\delta)}(x)\cdot f_1\otimes 
t^r\bold{e}_{-r}\otimes\bold{f}_{\delta}]\in \mathrm{H}^1_{\gamma}(\bold{D}_{\mathrm{dif},\infty}(D_{\mathrm{rig}}(-r)(\delta)))$$
for any $n\geqq n(D_{\mathrm{rig}})$ by definition of $\alpha_{(r,\delta)}$.
Since we have 
$$\iota_n(x_{-r,\delta})=\frac{p-1}{p}\cdot\mathrm{log}(\chi(\gamma))\cdot p_{\Delta}(\iota_n(x)\otimes \bold{e}_{-r}\otimes\bold{e}_{\delta}),$$
we have 
\begin{multline*}
[\iota_n(x), \alpha^{-1}_{\delta}\cdot\Omega\cdot t^{k-1-r}\cdot f_2]_{\mathrm{dif}}
=\frac{p}{p-1}\cdot
[\alpha_{\delta}\cdot\alpha_{(r,\delta)}(x)f_1t^r, \alpha_{\delta}^{-1}\Omega t^{k-1-r}f_2]_{\mathrm{dif}}\\
=\frac{p}{p-1}\cdot\mathrm{Res}_L(\langle \sigma_{-1}(\alpha_{\delta}\alpha_{(r,\delta)}(x)f_1t^r\otimes\bold{e}_D^{\vee}\otimes
\bold{e}_1), \alpha_{\delta}^{-1}\Omega t^{k-1-r}f_2\rangle)\\
=\frac{p}{p-1}\cdot (-1)^r\delta(-1)\delta_D(-1)\cdot \alpha_{(r,\delta)}(x)
\mathrm{Res}_L(\langle f_1t^r\otimes\bold{e}_D^{\vee}\otimes
\bold{e}_1, \Omega t^{k-1-r}f_2\rangle)\\
=-\frac{p}{p-1}\cdot (-1)^r\delta(-1)\delta_D(-1)\cdot \alpha_{(r,\delta)}(x),
\end{multline*}
where the first equality follows form Lemma \ref{3.9} below.
Hence, we obtain the equality 
$$\alpha_{(r,\delta)}(x)=-\delta_D(p^m)\cdot \frac{p-1}{p}\cdot (-1)^r\delta(-1)\delta_D(-1)\cdot[\tilde{x}, f_{(r,\delta),m}]_{\bold{P}^1}.$$

\end{proof}

\begin{lemma}\label{3.9}
For each pair $(r,\delta)$ as above, the following hold.
\begin{itemize}

\item[(i)]For $y\otimes \bold{e}_{-r}\otimes\bold{e}_{\delta}\in \bold{D}_{\mathrm{dif},n}(D_{\mathrm{rig}}(-r)(\delta))=\bold{D}_{\mathrm{dif},n}(D_{\mathrm{rig}})\otimes_LL(-r)(\delta)$, 
set $$p_{\Delta}(y\otimes\bold{e}_{-r}\otimes\bold{e}_{\delta})=:y_1\otimes\bold{e}_{-r}\otimes\bold{e}_{\delta}$$
$($recall $p_{\Delta}:=\frac{1}{|\Delta|}\sum_{\sigma\in \Delta}[\sigma]\in L[\Delta]$$)$. Then we have 
$$[y,\alpha^{-1}_{\delta}\cdot\Omega t^{k-1-r}f_2]_{\mathrm{dif}}=[y_1,
\alpha^{-1}_{\delta}\cdot\Omega t^{k-1-r}f_2]_{\mathrm{dif}}.$$
\item[(ii)]For $y\otimes \bold{e}_{-r}\otimes\bold{e}_{\delta}\in \bold{D}_{\mathrm{dif},n}(D_{\mathrm{rig}}(-r)(\delta))$, 
set $$(\gamma-1)(y\otimes \bold{e}_{-r}\otimes\bold{e}_{\delta})=:y_2\otimes\bold{e}_{-r}\otimes\bold{e}_{\delta}.$$ Then, we have 
$$[y_2, \alpha^{-1}_{\delta}\cdot\Omega t^{k-1-r}\cdot f_2]_{\mathrm{dif}}=0.$$

\end{itemize}
 \end{lemma}
\begin{proof}
We first remark that, for any $x,y\in \bold{D}_{\mathrm{dif},n}(D_{\mathrm{rig}})$ and 
$a\in \mathbb{Z}_p^{\times}$, we have 
$$[x\otimes\bold{e}_{-r}\otimes\bold{e}_{\delta}, y\otimes\bold{e}_{-r}\otimes\bold{e}_{\delta}]_{\mathrm{dif}}=(-1)^r\cdot \delta(-1)\cdot[x,y]_{\mathrm{dif}}$$ and 
$$[\sigma_a(x), y]_{\mathrm{dif}}=\delta_D(a)\cdot[x,\sigma_a^{-1}(y)]_{\mathrm{dif}}.$$ 

Using these, for $y$ and $y_1$ as in (i), we have 
\begin{multline*}
[y_1, \alpha_{\delta}^{-1}\cdot\Omega t^{k-1-r} f_2]_{\mathrm{dif}}=(-1)^r\cdot \delta(-1)\cdot
[y_1\otimes \bold{e}_{-r}\otimes\bold{e}_{\delta}, \alpha_{\delta}^{-1}\cdot\Omega t^{k-1-r} f_2\otimes \bold{e}_{-r}\otimes\bold{e}_{\delta}]_{\mathrm{dif}}\\
=(-1)^r\cdot\delta(-1)\cdot[p_{\Delta}(y\otimes\bold{e}_{-r}\otimes\bold{e}_{\delta}), \alpha_{\delta}^{-1}\cdot\Omega t^{k-1-r} f_2\otimes \bold{e}_r]_{\mathrm{dif}}\\
=(-1)^r\cdot\delta(-1)\cdot[y\otimes\bold{e}_{-r}\otimes\bold{e}_{\delta}, p_{\Delta}^{\delta_D(-r)(\delta)}(\alpha_{\delta}^{-1}\cdot\Omega t^{k-1-r} f_2\otimes \bold{e}_{-r}\otimes\bold{e}_{\delta})]_{\mathrm{dif}}\\
=(-1)^r\delta(-1)[y\otimes\bold{e}_{-r}\otimes\bold{e}_{\delta}, \alpha_{\delta}^{-1}\cdot\Omega t^{k-1-r} f_2\otimes \bold{e}_{-r}\otimes\bold{e}_{\delta}]_{\mathrm{dif}}
=[y, \alpha_{\delta}^{-1}\Omega t^{k-1-r} f_2]_{\mathrm{dif}},
\end{multline*}
where we set $p_{\Delta}^{\delta_{D(-r)(\delta)}}:=\frac{1}{|\Delta|}\sum_{\sigma\in \Delta}\delta_{D(-r)(\delta)}(\sigma)[\sigma]^{-1}\in L[\Delta]$, and the fourth equality follows from the fact that 
$\sigma_a(\alpha_{\delta}^{-1}\cdot\Omega t^{k-1-r}\otimes\bold{e}_{-r}\otimes\bold{e}_{\delta})=\delta_{D(-r)(\delta)}(a)\cdot\alpha_{\delta}^{-1}\cdot\Omega t^{k-1-r}\otimes\bold{e}_{-r}\otimes\bold{e}_{\delta}$ for any $a\in \mathbb{Z}_p^{\times}$. 

Similarly, for $y, y_2$ as in (ii), we have
\begin{multline*}
[y_2, \alpha^{-1}_{\delta}\Omega t^{k-1-r} f_2]_{\mathrm{dif}}=(-1)^r\delta(-1)[y_2\otimes \bold{e}_{-r}\otimes\bold{e}_{\delta}, \alpha^{-1}_{\delta}\Omega t^{k-1-r} f_2\otimes \bold{e}_{-r}\otimes\bold{e}_{\delta}]_{\mathrm{dif}}\\
=(-1)^r\delta(-1)[(\gamma-1)(y\otimes\bold{e}_{-r}\otimes\bold{e}_{\delta}), \alpha^{-1}_{\delta}\Omega t^{k-1-r} f_2\otimes\bold{e}_{-r}\otimes\bold{e}_{\delta}]_{\mathrm{dif}}\\
=(-1)^r\delta(-1)[y\otimes\bold{e}_{-r}\otimes\bold{e}_{\delta}, (\delta_{D(-r)(\delta)}(\chi(\gamma))\gamma^{-1}-1)(\alpha^{-1}_{\delta}\Omega t^{k-1-r}f_2\otimes \bold{e}_{-r}\otimes\bold{e}_{\delta})]_{\mathrm{dif}}=0.
\end{multline*}

\end{proof}

We next consider the map $\beta_{(r,\delta)}:D^{\delta_D(p)\psi=1}\rightarrow L$. We first recall that,
under the canonical inclusions $1-\varphi:D^{\psi=1}\hookrightarrow D^{\psi=0}$ and 
$1-\delta_D(p)^{-1}\varphi:D^{\delta_D(p)\psi=1}\hookrightarrow D^{\psi=0}$, one has 
$w_{\delta_D}(D^{\psi=1})=D^{\delta_D(p)\psi=1}$ by Proposition V.2.1 of \cite{Co10b}.

Similarly for the case $D^{\psi=1}$, one has the following isomorphism 
$$(D^{\natural}\boxtimes_{w_{\delta_D}}\bold{P}^1)^{g_p
=\delta_D(p)}
\isom D^{\delta_D(p)\psi=1}:z\mapsto \mathrm{Res}_{\mathbb{Z}_p}(z),$$ 
which induces the commutative diagram 
(set $g_p=\begin{pmatrix}p & 0\\ 0 &1\end{pmatrix}, w=\begin{pmatrix} 0 &1 \\ 1 & 0\end{pmatrix}\in G$)

\begin{equation*}
 \begin{CD}
  (D^{\natural}\boxtimes_{w_{\delta_D}}\bold{P}^1)^{g_p=1}
  @> z\mapsto \mathrm{Res}_{\mathbb{Z}_p}(z)>> D^{\psi=1}
  \\
 @ V z\mapsto w\cdot z VV @VV z\mapsto w_{\delta_D}(z)V\\
   (D^{\natural}\boxtimes_{w_{\delta_D}}\bold{P}^1)^{g_p=\delta_D(p)}
@> z\mapsto \mathrm{Res}_{\mathbb{Z}_p}(z) >> D^{\delta_D(p)\psi=1}
     \end{CD}
 \end{equation*}
 in which all the arrows are isomorphism. For $x\in D^{\delta_D(p)\psi=1}$, we denote by $\widetilde{x}\in (D^{\natural}\boxtimes_{w_{\delta_D}}\bold{P}^1)^{g_p=\delta_D(p)}$ the element such that $\mathrm{Res}_{\mathbb{Z}_p}(\widetilde{x})=x$.

For any $m\in \mathbb{Z}$ and any pair $(r,\delta)$ as above, define a function 
$$h_{(r,\delta),m}\in \mathrm{LP}_{\mathrm{c}}(\mathbb{Q}_p^{\times}, 
X^{-}_{\infty})^{\Gamma}$$ by 
$$h_{(r,\delta),m}(p^m\cdot a):=\begin{cases}\sigma_a(\alpha_{\delta}\cdot t^r \overline{f}_2)=\delta(a)^{-1}\alpha_{\delta}\cdot(at)^r\overline{f}_2 & \text{ if } n=m\\
0 & \text{ if } n\not=0\end{cases}$$
and  for $n\in \mathbb{Z}$ and $a\in \mathbb{Z}_p^{\times}$.

\begin{prop}\label{formula2}
For any $x\in D^{\delta_D(p)\psi=1}$, $m\in \mathbb{Z}$ and $(r,\delta)$ such that 
$0\leqq r\leqq k-1$ and $\delta:\Gamma\rightarrow L^{\times}$ with finite image, we have
$$\beta_{(r,\delta)}(x)=-\frac{p-1}{p}\cdot[\widetilde{x}, h_{(r,\delta),m}]_{\bold{P}^1}.$$

\end{prop}
\begin{proof}
By Proposition \ref{Co}, we have 
$$[\widetilde{x}, h_{(r,\delta),m}]_{\bold{P}^1}=\sum_{i\in \mathbb{Z}}\delta_D(p^i)[\iota_i^+(\widetilde{x}), h_{(r,\delta),m}(p^i)]_{\mathrm{dif}}=\delta_D(p)^m[\iota^+_m(\widetilde{x}), \alpha_{\delta}t^rf_2]_{\mathrm{dif}}.$$
Since we have 
$$\iota^+_{m,n}=\iota_n\left(\mathrm{Res}_{\mathbb{Z}_p}(g_p^{n-m}\cdot \widetilde{x})\right)=\delta_D(p)^{n-m} \cdot\iota_n(x)$$
for any $n\geqq n(D_{\mathrm{rig}})$, we have 
$$[\widetilde{x}, h_{(r,\delta),m}]_{\bold{P}^1}=\delta_D(p)^n[\iota_n(x), \alpha_{\delta}t^r f_2]_{\mathrm{dif}}.$$
On the other hand, since we have 
$$\mathrm{exp}^*((\sigma_{-1}(x\otimes\bold{e}_D^{\vee}\otimes\bold{e}_1))_{r,\delta^{-1}})=\beta_{(r,\delta)}(x)f_1\otimes \Omega\cdot t^k\bold{e}_D^{\vee}\otimes 
\frac{1}{t^{r+1}}\bold{e}_{r+1}\otimes\bold{f}_{\delta^{-1}}$$ by definition of $\beta_{(r,\delta)}$, we obtain 
\begin{multline*}
[\iota_n((\sigma_{-1}(x\otimes\bold{e}_D^{\vee}\otimes\bold{e}_1))_{r,\delta^{-1}})]\\
=\mathrm{log}(\chi(\gamma))\beta_{(r,\delta)}(x)\cdot [\alpha_{\delta}^{-1}f_1\otimes \Omega t^k\bold{e}_D^{\vee} \otimes 
\frac{1}{t^{r+1}}\bold{e}_{r+1}\otimes\bold{e}_{\delta^{-1}}]\in \mathrm{H}^1_{\gamma}(\bold{D}_{\mathrm{dif},\infty}(D_{\mathrm{rig}}^*(r)(\delta^{-1})))
\end{multline*}
by Proposition 2.16 of \cite{Na14a}. 
Since we have 
\begin{multline*}
\iota_n((\sigma_{-1}(x\otimes \bold{e}_D^{\vee}\otimes\bold{e}_1))_{r,\delta^{-1}})
=\frac{p-1}{p}\cdot\mathrm{log}(\chi(\gamma))p_{\Delta}(\iota_n(\sigma_{-1}(x\otimes \bold{e}_D^{\vee}\otimes\bold{e}_1)\otimes \bold{e}_r\otimes\bold{e}_{\delta^{-1}}))\\
=\frac{p-1}{p}\cdot\mathrm{log}(\chi(\gamma))p_{\Delta}((-1)^r\delta(-1)\cdot\delta_D(p)^n\cdot\sigma_{-1}(\iota_n(x)\otimes \bold{e}_D^{\vee}\otimes\bold{e}_{r+1}\otimes\bold{e}_{\delta^{-1}}))\\
=(-1)^r\delta(-1)\delta_D(p)^n\cdot\frac{p-1}{p}\cdot\mathrm{log}(\chi(\gamma))p_{\Delta}(\iota_n(x)\otimes \bold{e}_D^{\vee}\otimes\bold{e}_{r+1}\otimes\bold{e}_{\delta^{-1}}),
\end{multline*}
we obtain
\begin{multline*}
[\iota_n(x), \alpha_{\delta}t^rf_2]_{\mathrm{dif}}=(-1)^r\delta(-1)\delta_D(p)^{-n}\cdot\frac{p}{p-1}\cdot\beta_{(r,\delta)}(x)\cdot [\alpha^{-1}_{\delta}\Omega t^{k-r-1}f_1, \alpha_{\delta}t^rf_2]_{\mathrm{dif}}\\
=(-1)^r\delta(-1)\delta_D(p)^{-n}\cdot\frac{p}{p-1}\cdot\beta_{(r,\delta)}(x)\cdot \mathrm{Res}_L(\langle \sigma_{-1}(\alpha^{-1}_{\delta}t^{-r}f_1\otimes \Omega t^k\bold{e}_D^{\vee}\otimes t^{-1}\bold{e}_1), \alpha_{\delta}t^rf_2\rangle)\\
=\delta_D(p)^{-n}\cdot\frac{p}{p-1}\cdot\beta_{(r,\delta)}(x)\cdot \mathrm{Res}_L(\langle t^{-r}f_1\otimes \Omega t^k\bold{e}_D^{\vee}\otimes t^{-1}\bold{e}_1, t^rf_2\rangle)=-\delta_D(p)^{-n}\cdot\frac{p}{p-1}\cdot\beta_{(r,\delta)}(x),
\end{multline*}
where the first equality follows from Lemma \ref{3.16} below. By this equality, we obtain 
$$[\widetilde{x}, h_{(r,\delta),m}]_{\bold{P}^1}=-\frac{p}{p-1}\cdot \beta_{(r,\delta)}(x),$$
which proves the proposition.

\end{proof}
\begin{lemma}\label{3.16}
For each pair $(r,\delta)$ as above, the following hold.
\begin{itemize}
\item[(i)]For $y\otimes\bold{e}_D^{\vee}\otimes \bold{e}_{r+1}\otimes\bold{e}_{\delta^{-1}}\in \bold{D}_{\mathrm{dif},n}
(D_{\mathrm{rig}}^*(r)(\delta^{-1}))$, set 
$$p_{\Delta}(y\otimes\bold{e}_D^{\vee}\otimes \bold{e}_{r+1}\otimes\bold{e}_{\delta^{-1}})
=y_1\otimes\bold{e}_D^{\vee}\otimes\bold{e}_{r+1}\otimes\bold{e}_{\delta^{-1}},$$ then we have 
$$[y, \alpha_{\delta}t^rf_2]_{\mathrm{dif}}=[y_1,\alpha_{\delta}t^rf_2]_{\mathrm{dif}}.$$
\item[(ii)]For $y\otimes\bold{e}_D^{\vee}\otimes \bold{e}_{r+1}\otimes\bold{e}_{\delta^{-1}}\in \bold{D}_{\mathrm{dif},n}
(D_{\mathrm{rig}}^*(r)(\delta^{-1}))$, set 
$$(\gamma-1)(y\otimes\bold{e}_D^{\vee}\otimes \bold{e}_{r+1}\otimes\bold{e}_{\delta^{-1}})
=y_2\otimes\bold{e}_D^{\vee}\otimes \bold{e}_{r+1}\otimes\bold{e}_{\delta^{-1}},$$ then we have 
$$[y_2, \alpha_{\delta}t^rf_2]_{\mathrm{dif}}=0.$$
\end{itemize}
\end{lemma}
\begin{proof}
We first remark that we have $$\sigma_a(\alpha_{\delta}t^kf_2\otimes\bold{e}_{D}^{\vee}\otimes\bold{e}_{r+1}\otimes\bold{e}_{\delta^{-1}})=\delta_{D^*(r)(\delta^{-1})}(a)\alpha_{\delta}t^kf_2\otimes\bold{e}_{D}^{\vee}\otimes\bold{e}_{r+1}\otimes\bold{e}_{\delta^{-1}}$$ for any $a\in \mathbb{Z}_p^{\times}$

For $y, y_1$ as in (i), then we have 
\begin{multline*}
[y_1,\alpha_{\delta}t^rf_2]_{\mathrm{dif}}=\delta_D(-1)(-1)^r\delta(-1)[y_1\otimes\bold{e}_D^{\vee}\otimes \bold{e}_{r+1}\otimes\bold{e}_{\delta^{-1}},\alpha_{\delta}t^rf_2\otimes\bold{e}_D^{\vee}\otimes \bold{e}_{r+1}\otimes\bold{e}_{\delta^{-1}}]_{\mathrm{dif}}\\
=\delta_D(-1)(-1)^r\delta(-1)[p_{\Delta}(y\otimes\bold{e}_D^{\vee}\otimes \bold{e}_{r+1}\otimes\bold{e}_{\delta^{-1}}),\alpha_{\delta}t^rf_2\otimes\bold{e}_D^{\vee}\otimes \bold{e}_{r+1}\otimes\bold{e}_{\delta^{-1}}]_{\mathrm{dif}}\\
=\delta_D(-1)(-1)^r\delta(-1)[y\otimes\bold{e}_D^{\vee}\otimes \bold{e}_{r+1}\otimes\bold{e}_{\delta^{-1}},p_{\Delta}^{\delta_{D^*(r)(\delta^{-1})}}(\alpha_{\delta}t^rf_2\otimes\bold{e}_D^{\vee}\otimes \bold{e}_{r+1}\otimes\bold{e}_{\delta^{-1}}]_{\mathrm{dif}}\\
=\delta_D(-1)(-1)^r\delta(-1)[y\otimes\otimes\bold{e}_D^{\vee}\otimes \bold{e}_{r+1}\otimes\bold{e}_{\delta^{-1}}, \alpha_{\delta}t^rf_2\otimes\bold{e}_D^{\vee}\otimes \bold{e}_{r+1}\otimes\bold{e}_{\delta^{-1}}]_{\mathrm{dif}}=[y,\alpha_{\delta}t^rf_2]_{\mathrm{dif}}.
\end{multline*}

We can also prove (ii) in the same way, hence we omit the proof.

\end{proof}

Before proceeding to the proof of the key proposition (Proposition \ref{key}), we prove the following corollary, which is crucial in the proof of the main theorem (Theorem \ref{4.5}) in the next section.
\begin{corollary}\label{3.12.123}
The following maps 
$$D^{\psi=1}\rightarrow \prod_{(r,\delta)}\bold{D}^0_{\mathrm{dR}}(D_{\mathrm{rig}}(-r)(\delta)):x\mapsto 
(\mathrm{exp}^*(x_{-r,\delta}))_{(r,\delta)}$$ 
and 
$$(D^*)^{\psi=1}\rightarrow \prod_{(r,\delta)}\bold{D}^0_{\mathrm{dR}}(D_{\mathrm{rig}}^*(r)(\delta^{-1})):y\mapsto 
(\mathrm{exp}^*(y_{r,\delta^{-1}}))_{(r,\delta)}$$ 
are both injective, where the both products are taken with respect to all the pairs $(r,\delta)$ such that 
$0\leqq r\leqq k-1$ and $\delta:\Gamma\rightarrow \overline{L}^{\times}$ with finite image.
\end{corollary}
\begin{rem}
The statement of this corollary is no longer true when $D$ is absolute irreducible and $\bold{D}_{\mathrm{cris}}(D_{\mathrm{rig}}(-r)(\delta))^{\varphi=1}\not=0$ for some $(r,\delta)$ (which corresponds to the exceptional zero case), or when $D$ is reducible. In the reducible case, if we have an exact sequence 
$0\rightarrow D_1\rightarrow D\rightarrow D_2\rightarrow 0$ of \'etale $(\varphi,\Gamma)$-modules such that $D_1$ and $D_2$ are of rank one, then the kernels of the above maps contain respectively $D_1^{\psi=1}$ and $(D_2^*)^{\psi=1}$. In the exceptional zero case, we can make the above maps injective if we add another kind of dual exponential map $$\mathrm{H}^1_{\varphi,\gamma}(D_{\mathrm{rig}})
\rightarrow \bold{D}_{\mathrm{crys}}(D_{\mathrm{rig}})^{\varphi=1/p}$$ (i.e. the dual of $$\mathrm{exp}_f:\bold{D}_{\mathrm{crys}}(D_{\mathrm{rig}}^*)/(1-\varphi)\bold{D}_{\mathrm{crys}}(D_{\mathrm{rig}}^*)\rightarrow \mathrm{H}^1_{\varphi,\gamma}(D_{\mathrm{rig}}^*)).$$ See the next article \cite{NY} for these cases.
\end{rem}
\begin{proof}
We only prove that the first map is injective using Proposition \ref{formula1}, since the proof for the latter map is similar if we use Proposition \ref{formula2} instead.

We assume that $x\in D^{\psi=1}$ satisfies $\mathrm{exp}^*(x_{-r,\delta})=0$ for any pair $(r,\delta)$ 
as above. 
By Proposition \ref{formula1} and definition of $\alpha_{(r,\delta)}$, we have $$[\widetilde{x},f_{(r,\delta), m}]_{\bold{P}^1}=0$$
for any $(r,\delta)$ and $m$. Since we have 
$$\Pi(D)_{\mathrm{c}}^{\mathrm{alg}}\isom \mathrm{LP}_{\mathrm{c}}(\mathbb{Q}_p^{\times}, L_{\infty})^{\Gamma}\subseteq \mathrm{LP}_{\mathrm{c}}(\mathbb{Q}_p^{\times}, \overline{L}_{\infty})^{\Gamma}$$ and $\mathrm{LP}_{\mathrm{c}}(\mathbb{Q}_p^{\times}, \overline{L}_{\infty})^{\Gamma}$ is generated by $\{f_{(r,\delta), m}\}_{(r,\delta),m}$ as $\overline{L}$-vector space, 
the lift $$\widetilde{x}\in (D^{\natural}\boxtimes_{\delta_D}\bold{P}^1)^{g_p=1}\isom (\Pi(D)^{\vee}\otimes \delta_D\circ\mathrm{det})^{g_p=1}$$ of $x$
satisfies $\widetilde{x}|_{\Pi(D)_{\mathrm{c}}^{\mathrm{alg}}}=0.$ We also claim that $\widetilde{x}|_{\Pi(D)^{\mathrm{alg}}}=0$, which proves that $\widetilde{x}=0$ since $\Pi(D)^{\mathrm{alg}}$ is 
dense in $\Pi(D)$ (since $\Pi(D)$ is topologically irreducible), hence $x=\mathrm{Res}_{\mathbb{Z}_p}(\widetilde{x})=0$. 

Hence, it suffices to show this claim. If $D$ is non-trianguline, then we have $\Pi(D)^{\mathrm{alg}}_{\mathrm{c}}=\Pi(D)^{\mathrm{alg}}$, hence there is nothing to prove. Assume that $D$ is trianguline. Then, we have $D_{\mathrm{crab}}\not=0$ and have a $B$-equivariant isomorphism 
$D_{\mathrm{crab}}\otimes_L\mathrm{Sym}^{k-1}L^2\isom \Pi(D)^{\mathrm{alg}}/\Pi(D)^{\mathrm{alg}}_{\mathrm{c}}$. Since we have $\widetilde{x}|_{\Pi(D)_{\mathrm{c}}^{\mathrm{alg}}}=0$, the pairing 
$[\widetilde{x},-]_{\bold{P}^1}$ induces $[\widetilde{x},-]_{\bold{P}^1}:D_{\mathrm{crab}}\otimes_L\mathrm{Sym}^{k-1}L^2\rightarrow L$. Here, we only prove the claim when $D$ is potentially crystalline and $D_{\mathrm{crab}}$ is the non trivial extension $0\rightarrow L(\delta_1)\rightarrow D_{\mathrm{crab}}\rightarrow L(\delta_1)\rightarrow 0$ for a locally constant homomorphism $\delta_1:\mathbb{Q}_p^{\times}\rightarrow L^{\times}$ (since we can similarly prove it in other cases). In this case, we can take a basis $\{\bold{f}_1, \bold{f}_2\}$ of $D_{\mathrm{crab}}$ such that 
$\varphi(\bold{f}_1)=\delta_1(p), \varphi(\bold{f}_2)=\delta_1(p)(\bold{f}_2+\bold{f}_1)$. Then, it suffices to show that we have 
$$[\widetilde{x},\bold{f}_i\otimes e_1^{j}e_2^{k-1-j}]_{\bold{P}^1}=0$$
for any $i=1,2$ and $0\leqq j\leqq k-1$. Since we have $$\left(\begin{pmatrix}p& 0\\ 0& 1\end{pmatrix}-\delta_1(p)p^{j}\right)\cdot\bold{f}_1\otimes e_1^je_2^{k-1-j}=0$$ and $\widetilde{x}\in (\Pi(D)^{\vee}\otimes \delta_D(\mathrm{det}))^{g_p=1}$, we have 
\begin{multline*}
0=\left[\widetilde{x}, \left(\begin{pmatrix}p& 0\\ 0& 1\end{pmatrix}-\delta_1(p)p^{j}\right)\cdot\bold{f}_1\otimes e_1^je_2^{k-1-j}\right]_{\bold{P}^1}\\
=\delta_D(p)\left[\begin{pmatrix}p^{-1}& 0 \\ 0 &1\end{pmatrix}\widetilde{x}, \bold{f}_1\otimes e_1^je_2^{k-1-j}\right]_{\bold{P}^1}-\delta_1(p)p^j[\widetilde{x}, \bold{f}_1\otimes e_1^je_2^{k-1-j}]_{\bold{P}^1}\\
=(\delta_D(p)-\delta_1(p)p^j)[\widetilde{x}, \bold{f}_1\otimes e_1^je_2^{k-1-j}]_{\bold{P}^1}.
\end{multline*}
Since we have $\delta_D(p)=\delta_1(p)^2p^{k}$, we have $\delta_D(p)-\delta_1(p)p^j=\delta_1(p)p^j(\delta_1(p)p^{k-j}-1)$, which is non zero since, for $j=0$, we have $\delta_1(p)\not=p^{-k}$ by the irreducibility of $D$, and, for $1\leqq j\leqq k-1$, we assumed that $1-\varphi$ is isomorphism on $\bold{D}_{\mathrm{cris}}(D(-(k-j)(\delta)))$ for any $\delta$ (in particular, for $\delta=\delta_1^{-1}|_{\mathbb{Z}_p^{\times}}$). Hence, we obtain $[\widetilde{x}, \bold{f}_1\otimes e_1^je_2^{k-1-j}]_{\bold{P}^1}=0$ for any $0\leqq j\leqq k-1$. Finally, using the equality 
$$\left(\begin{pmatrix}p& 0\\ 0& 1\end{pmatrix}-\delta_1(p)p^{j}\right)\cdot\bold{f}_2\otimes e_1^je_2^{k-1-j}=\delta_1(p)p^j\bold{f}_1\otimes e_1^je_2^{k-1-j},$$ we can similarly obtain 
$[\widetilde{x}, \bold{f}_2\otimes e_1^je_2^{k-1-j}]_{\bold{P}^1}=0$, which proves the claim. 
\end{proof}

We next recall the theorem of Emerton on the compatibility of 
the $p$-adic and the classical local Langlands correspondence. 
 Fix an isomorphism $\iota:\overline{L}\isom \mathbb{C}$. Let $\pi'_p(D)$ be the irreducible smooth admissible representation of 
 $G$ defined over $\mathbb{C}$ corresponding to the Weil-Deligne representation $W(D_{\mathrm{rig}})^{\mathrm{ss}}\otimes_{L,\iota}\mathbb{C}$ of $W_{\mathbb{Q}_p}$ over $\mathbb{C}$ of rank two via the unitary normalized local Langlands correspondence, where $W(D_{\mathrm{rig}})^{\mathrm{ss}}$ is the Frobenius semi-simplification of $W(D_{\mathrm{rig}})$. We remark that, under this normalization, the local $L$- and $\varepsilon$-factors attached to $W(D_{\mathrm{rig}})^{\mathrm{ss}}\otimes_{L,\iota}\mathbb{C}$  coincide with those for $\pi'_p(D)$, and the central character of $\pi_p'(D)$ is equal to 
 $$\iota\circ\mathrm{det}_LW(D_{\mathrm{rig}}):Z(\isom \mathbb{Q}_p^{\times})\rightarrow \mathbb{C}^{\times},$$ where we regard $\mathrm{det}_LW(D_{\mathrm{rig}})$ as a character $\mathrm{det}_LW(D_{\mathrm{rig}}):\mathbb{Q}_p^{\times}\rightarrow L^{\times}$ via local class field theory. For our purpose, 
 we need another normalization called Tate's normalization, which we define by 
 $$\pi_p(D)_{\overline{L}}:=(\pi_p'(D)\otimes |\mathrm{det}|_p^{-1/2})\otimes_{\mathbb{C},\iota^{-1}}\overline{L}.$$ 
 Then, it is known that $\pi_p(D)_{\overline{L}}$ does not depend on the choice of $\iota$, and is defined over $L$. Then, we denote $\pi_p(D)$ for the model of $\pi_p(D)_{\overline{L}}$ defined over $L$. Let $\omega_{\pi_p(D)}:\mathbb{Q}_p^{\times}\rightarrow L^{\times}$ be the central character of $\pi_p(D)$.  Since 
 we have $\mathrm{det}_LW(D_{\mathrm{rig}})=\mathrm{det}_{\mathcal{E}_L}D\cdot x^{-k}$, then one has an equality 
 \begin{equation}
 \omega_{\pi_p(D)}=\mathrm{det}_LW(D_{\mathrm{rig}})\cdot |-|_p^{-1}=\delta_D\cdot x^{-(k-1)},
 \end{equation}
 where we define $x^{i}:\mathbb{Q}_p^{\times}\rightarrow L^{\times}:y\mapsto y^{i}$ for $i\in \mathbb{Z}$.
Finally, we set $\pi_p^{\mathrm{m}}(D):=\mathrm{Ind}^G_B(\delta\cdot |-|_p\otimes \delta)_{\mathrm{sm}}$ (which is not irreducible) if 
 $\pi_p(D)=\delta\circ\mathrm{det}:G\rightarrow L^{\times}$ is one dimensional defined by a character $\delta:\mathbb{Q}_p^{\times}\rightarrow L^{\times}$, and $\pi_p^{\mathrm{m}}(D):=\pi_p(D)$ otherwise. Under this situation, one has the following theorem, and the non-trianguline case of this theorem is crucial for our proof of Proposition \ref{key}. 
 
 \begin{thm}
 Under the situation above, there exists a $G$-equivariant isomorphism 
  \begin{equation}\label{Em}
 \Pi(D)^{\mathrm{alg}}\isom \pi^{\mathrm{m}}_p(D)\otimes_L\mathrm{Sym}^{k-1}L^2.
 \end{equation}
 \end{thm}
 \begin{proof}
 This is Th\'eor\`eme VI.6.50 of \cite{Co10b} for the trianguline case, and 
 Theorem 3.3.22 of \cite{Em} for the non-trianguline case.
 \end{proof}
 \begin{rem}
 The proof of Theorem 3.3.22 of \cite{Em} is done by a global method using the complete cohomology of 
 modular curves. No purely local proof of this theorem has been known (at least to the author) until now. As we can easily see from the proof below, this theorem is in fact equivalent to Proposition \ref{key}. Hence, our proof of Proposition \ref{key} given below also depends on the global method.
 \end{rem}
 
 We next recall a formula of the action of $w:=\begin{pmatrix}0 & 1\\ 1 & 0\end{pmatrix}\in G$ 
 on the Kirillov model of supercuspidal representations of $G$ following the book of 
 Bushnell-Henniart \cite{BH}. From now on until the end of this subsection, we furthermore assume that $D$ is non-trianguline, which is equivalent to that $\pi^{\mathrm{m}}_p(D)=\pi_p(D)$ and this is a supercuspidal representation of $G$. By the classical theory of Kirillov model, then there exists a $B$-equivariant isomorphism $$\pi_p(D)\isom \mathrm{LC}_{\mathrm{c}}(\mathbb{Q}_p, L_{\infty})^{\Gamma}$$ which is unique up to $L^{\times}$ (see, for example, VI.4 of  \cite{Co10b}). Using this isomorphism, we can uniquely extend the action of $B$ on $\mathrm{LC}_{\mathrm{c}}(\mathbb{Q}_p, L_{\infty})^{\Gamma}$ to that of $G$ such that this isomorphism 
is $G$-equivariant, which we denote by $\pi_p(g)\cdot f$ for any $g\in G$ and $f\in \mathrm{LC}_{\mathrm{c}}(\mathbb{Q}_p^{\times}, L_{\infty})^{\Gamma}$. We now recall a formula on the action of $\pi_p(w)$ on $\mathrm{LC}_{\mathrm{c}}(\mathbb{Q}_p, L_{\infty})^{\Gamma}$ using the $\varepsilon$-factor associated to $\pi_p(D)$. Decompose $L_{\infty}=\prod_{\tau}L_{\tau}$ into a finite product of fields $L_{\tau}$. 
For any $\tau$, fix an isomorphism $\iota_{\tau}:\overline{L}_{\tau}\isom \mathbb{C}$. For any $\tau$, let $$\varepsilon(\pi_p(D)\otimes_{L,\iota_{\tau}}\mathbb{C},s, \iota_{\tau}\circ\psi_{\zeta})\,\,\,(s\in \mathbb{C})$$ be the 
$\varepsilon$-factor associated to $\pi_p(D)\otimes_{L,\iota_{\tau}}\mathbb{C}$ with respect to the additive character 
$\iota_{\tau}\circ\psi_{\zeta}:\mathbb{Q}_p\rightarrow \mathbb{C}^{\times}$. Since $(\pi_p(D)\otimes_{L,\iota_{\tau}}\mathbb{C})\otimes |\mathrm{det}|^{1/2}$ corresponds to $W(D_{\mathrm{rig}})\otimes_{L,\iota_{\tau}}\mathbb{C}$ via the unitary local Langlands correspondence, we have 
\begin{multline*}
\varepsilon(\pi_p(D)\otimes_{L,\iota_{\tau}}\mathbb{C},\frac{1}{2}, \iota_{\tau}\circ\psi_{\zeta})=\varepsilon((\pi_p(D)\otimes_{L,\iota_{\tau}}\mathbb{C})\otimes |\mathrm{det}|^{1/2},0, \iota_{\tau}\circ\psi_{\zeta})\\
=\varepsilon(W(D_{\mathrm{rig}})\otimes_{L,\iota_{\tau}}\mathbb{C}, \iota_{\tau}\circ\psi_{\zeta})=\varepsilon(W(D_{\mathrm{rig}})\otimes_{L}\overline{L}_{\tau}, \psi_{\zeta})\otimes 1\in \overline{L}_{\tau}\otimes_{\overline{L}_{\tau},\iota_{\tau}}\mathbb{C}.
\end{multline*} Hence, $\varepsilon_L(W(D_{\mathrm{rig}}),\zeta)=\prod_{\tau}\varepsilon(W(D_{\mathrm{rig}})\otimes_{L}\overline{L}_{\tau}, \psi_{\zeta})\in (L_{\infty})^{\times}$ satisfies the equality 
\begin{equation}\label{25}
\varepsilon_L(W(D_{\mathrm{rig}}),\zeta)\otimes_{L_{\infty},\iota_{\tau}} 1=\varepsilon(\pi_p(D)\otimes_{L,\iota_{\tau}}\mathbb{C},\frac{1}{2},\iota_{\tau}\circ\psi_{\zeta})
\end{equation} for any $\tau$.

For $k\in \mathbb{Z}$ and a locally constant homomorphism $\eta:\mathbb{Q}_p^{\times}\rightarrow L^{\times}$, define a locally constant function $\xi_{\eta,k}:\mathbb{Q}_p^{\times}\rightarrow L$ with compact support by 
$$\xi_{\eta,k}(x):=\begin{cases} \eta(x) & \text{ if } x\in p^k\mathbb{Z}_p^{\times} \\
                                                     0 & \text{ otherwise }\end{cases}.$$ 
We remark that we have $\alpha_{\eta}^{-1}\cdot\xi_{\eta,k}\in \mathrm{LC}_{\mathrm{c}}(\mathbb{Q}_p^{\times}, L_{\infty})^{\Gamma}$ for the fixed $\bold{f}_{\eta}=\alpha_{\eta}\cdot\bold{e}_{\eta}\in 
L_{\infty}(\eta)^{\Gamma}$ since we have $\sigma_a(\alpha_{\eta})=\eta(a)^{-1}\alpha_{\eta}$ for any $a\in \mathbb{Z}_p^{\times}$.

Under these preliminaries, we have the following formula.
\begin{thm}\label{BH}$($\cite{BH} $37.3)$
For any locally constant homomorphism $\eta:\mathbb{Q}_p^{\times}\rightarrow L^{\times}$ and any $k\in \mathbb{Z}$, we have 
$$\pi_p(w)\cdot (\alpha_{\eta}^{-1}\cdot\xi_{\eta,k})=\eta(-1)\cdot\alpha^{-1}_{\eta}\cdot\varepsilon_L(W(D_{\mathrm{rig}})(\eta^{-1}),\zeta)\cdot \xi_{\eta^{-1}\cdot w_{\pi_p(D)}, -a(W(D_{\mathrm{rig}})(\eta^{-1}))-k},$$
where $a(W(D_{\mathrm{rig}})(\eta^{-1}))$ is the exponent of the Artin conductor of $W(D_{\mathrm{rig}}(\eta^{-1}))$ $($see \cite{De73} for the definition$)$.
\end{thm}    
\begin{proof}
We first remark that the right hand side in the theorem is contained in $\mathrm{LC}_{\mathrm{c}}(\mathbb{Q}_p^{\times}, L_{\infty})^{\Gamma}$ since we have 
$$\sigma_a(\varepsilon_L(W(D_{\mathrm{rig}})(\eta^{-1}),\zeta))=\mathrm{det}_LW(D_{\mathrm{rig}})(a)\cdot\eta(a)^{-2}\cdot \varepsilon_L(W(D_{\mathrm{rig}})(\eta^{-1}),\zeta)$$ for any $a\in \mathbb{Z}_p^{\times}$.
Since we have $$\varepsilon_L(W(D_{\mathrm{rig}})(\eta^{-1}),\zeta)\otimes_{L_{\infty},\iota_{\tau}} 1=\varepsilon(
(\pi_p(D)\otimes (\eta^{-1}\circ\mathrm{det}))\otimes_{L,\iota_{\tau}}\mathbb{C},\frac{1}{2},\iota_{\tau}\circ\psi_{\zeta})$$ for any $\tau$, the theorem follows from 
Theorem 37.3 of \cite{BH}.

\end{proof}

We recall that one has a canonical $B$-equivariant isomorphism 
$$\Pi(D)^{\mathrm{alg}}=\Pi(D)_{\mathrm{c}}^{\mathrm{alg}}\isom \mathrm{LP}_{\mathrm{c}}(\mathbb{Q}_p^{\times}, X^{-}_{\infty})^{\Gamma}$$ 
(under the assumption that $D$ is non-trianguline), by which we extend the action of $B$ on the right hand side to that of $G$ such that this isomorphism becomes $G$-equivariant. We denote this action by $\Pi(g)\cdot f$ for any $g\in G$ and $f\in \mathrm{LP}_{\mathrm{c}}(\mathbb{Q}_p^{\times}, X^{-}_{\infty})^{\Gamma}$.

To show the key proposition, we need the following corollary of Theorem \ref{BH}.

\begin{corollary}\label{3.12}For any pair $(r,\delta)$ such that $0\leqq r\leqq k-1$ and 
$\delta:\Gamma\rightarrow L^{\times}$ with finite image, we have
$$\Pi(w)\cdot h_{(r,\delta),m}=\frac{\delta(\sigma_{-1})\cdot r!}{(k-1-r)!}\cdot\frac{\varepsilon_L(W(D_{\mathrm{rig}}(\delta)),\zeta)}{(w_{\pi_p(D)}(p)\cdot p^{(k-1-r)})^{a(W(D_{\mathrm{rig}}(\delta)))+m}}\cdot\frac{\alpha_{\delta}^2}{\Omega}\cdot f_{(r,\delta),-a(W(D_{\mathrm{rig}}(\delta)))-m}$$

\end{corollary}
\begin{proof}
We first remark that, under the $B$-equivariant isomorphism 
\begin{multline}\label{27}
\mathrm{LC}_c(\mathbb{Q}_p^{\times},L_{\infty})^{\Gamma}
\otimes_L\mathrm{Sym}^{k-1}L^2\isom \mathrm{LP}_c(\mathbb{Q}_p^{\times}, X^{-}_{\infty})^{\Gamma}:\\
\phi_i\otimes e_1^ie_2^{k-1-i}\mapsto [x\mapsto (k-1-i)!\cdot \phi_i(x)(xt)^i]\cdot\overline{f}_2,
\end{multline}
$h_{(r,\delta),m}\in \mathrm{LC}_c(\mathbb{Q}_p^{\times},L_{\infty})^{\Gamma}$ 
corresponds to 
$$\frac{1}{(k-1-r)!}\cdot \alpha_{\delta}\cdot\xi_{\delta^{-1},m}\otimes e_1^re_2^{k-1-r}\in \mathrm{LC}_c(\mathbb{Q}_p^{\times},L_{\infty})^{\Gamma}
\otimes_L\mathrm{Sym}^{k-1}L^2.$$ 
Applying Theorem \ref{BH} to $\alpha_{\delta}\cdot\xi_{\delta^{-1},m}$ (where we regard $\delta:\Gamma\rightarrow L^{\times}$ as $\delta:\mathbb{Q}_p^{\times}\rightarrow L^{\times}$ by $\delta(p)=1$ and $\delta(a)=\delta(\sigma_a)$ for $a\in \mathbb{Z}_p^{\times}$), then
we obtain 
\begin{multline*}
\Pi(w)\cdot h_{(r,\delta),m}=
\frac{1}{(k-1-r)!}\cdot (\pi_p(w)\cdot (\alpha_{\delta}\cdot\xi_{\delta^{-1},m})\otimes (w\cdot (e_1^re_2^{k-1-r}))\\
=\frac{1}{(k-1-r)!}\cdot\delta(\sigma_{-1})\cdot\alpha_{\delta}\cdot\varepsilon_L(W(D_{\mathrm{rig}}(\delta)),\zeta)\cdot \xi_{\delta\cdot w_{\pi_p(D)}, -a(W(D_{\mathrm{rig}}(\delta)))-m}\otimes 
e_1^{k-1-r}e_2^r\\
=\delta(\sigma_{-1})\cdot\frac{r!}{(k-1-r)!}\cdot\frac{\varepsilon_L(W(D_{\mathrm{rig}}(\delta)),\zeta)}{(w_{\pi_p(D)}(p)\cdot p^{(k-1-r)})^{a(W(D_{\mathrm{rig}}(\delta)))+m}}\cdot 
 \frac{\alpha_{\delta}^2}{\Omega} 
\cdot f_{(r,\delta),-a(W(D_{\mathrm{rig}}(\delta)))-m},
\end{multline*}
where (we remark that one has $W(D_{\mathrm{rig}}(\delta))=W(D_{\mathrm{rig}})(\delta)$ and) the third equality follows from the fact that $f_{(r,\delta),m}$ corresponds to 
$$\frac{(w_{\pi_p(D)}(p)\cdot p^{(k-1-r)})^{-m}}{r!}\cdot (\alpha_{\delta}^{-1}\cdot\Omega)\cdot\xi_{\delta\cdot w_{\pi_p(D)}, m}\otimes 
e_1^{k-1-r}e_2^r$$
by the isomorphism (\ref{27}).
\end{proof}     

Finally, we prove Proposition \ref{key}.                                   

\begin{proof}(of Proposition \ref{key})
Take $x\in D^{\psi=1}$. Take $\widetilde{x}\in (D^{\natural}\boxtimes_{w_{\delta_D}}\bold{P}^1)^{g_p=1}$ such that $\mathrm{Res}_{\mathbb{Z}_p}(\widetilde{x})=x$. Then, $w\cdot \tilde{x}\in (D^{\natural}\boxtimes_{w_{\delta_D}}\bold{P}^1)^{g_p=\delta_D(p)}$ satisfies that 
$$\mathrm{Res}_{\mathbb{Z}_p}(w\cdot \widetilde{x})=w_{\delta_D}(x)\in D^{\delta_D(p)\psi=1}.$$ 
By Proposition \ref{formula1}, Proposition \ref{formula2} and Corollary \ref{3.12}, then we have 
\begin{multline*}
\beta_{(r,\delta)}(w_{\delta_D}(x))=-\frac{p-1}{p}
\cdot [w\cdot \tilde{x}, h_{(r,\delta),0}]_{\bold{P}^1}
=-\frac{p-1}{p}
\cdot \delta_D(\mathrm{det}(w))
\cdot [\tilde{x}, \Pi(w)\cdot h_{(r,\delta),0}]_{\bold{P}^1}\\
=-\frac{(p-1)\cdot\delta_D(\sigma_{-1})\cdot \delta(\sigma_{-1})\cdot r!}{p\cdot (k-1-r)!}\cdot \frac{\varepsilon_L(W(D_{\mathrm{rig}}(\delta)),\zeta)}{(w_{\pi_p(D)}(p)\cdot p^{(k-1-r)})^{a(W(D_{\mathrm{rig}}(\delta)))}}\cdot\frac{\alpha_{\delta}^2}{\Omega}\cdot 
[\tilde{x}, f_{(r,\delta),-a(W(D_{\mathrm{rig}}(\delta)))}]_{\bold{P}^1}\\
=\frac{r!}{(-1)^r}\cdot\frac{1}{(k-1-r)!}\cdot\frac{\varepsilon_L(W(D_{\mathrm{rig}}(\delta)),\zeta)}{(w_{\pi_p(D)}(p)\cdot p^{(k-1-r)})^{a(W(D_{\mathrm{rig}}(\delta)))}}\cdot\frac{\alpha_{\delta}^2}{\Omega}
\cdot \delta_D(p)^{a(W(D_{\mathrm{rig}}(\delta)))}\cdot\alpha_{(r,\delta)}(x)\\
=(p^r)^{a(W(D_{\mathrm{rig}}(\delta)))}\cdot
\frac{r!}{(-1)^r}\cdot\frac{1}{(k-1-r)!}\cdot\frac{\varepsilon_L(W(D_{\mathrm{rig}}(\delta)),\zeta)\cdot\alpha_{\delta}^2}{\Omega}
\cdot\alpha_{(r,\delta)}(x)\\
=\frac{r!}{(-1)^r}\cdot\frac{1}{(k-1-r)!}\cdot\frac{\varepsilon_L(W(
D_{\mathrm{rig}}(-r)(\delta)),\zeta)\cdot\alpha_{\delta}^2}{
\Omega}\cdot \alpha_{(r,\delta)}(x)
\end{multline*}
 for any pair $(r,\delta)$ as in Proposition \ref{key},
where the fifth equality follows from the equality $\frac{\delta_D(p)}{w_{\pi_p(D)}(p)\cdot p^{(k-1-r)}}=p^r$, and the sixth equality follows from the equality $$(p^r)^{a(W(D(\delta)))}\cdot\varepsilon_L(W(D_{\mathrm{rig}}(\delta)),\zeta)=\varepsilon_L(W(
D_{\mathrm{rig}}(-r)(\delta)),\zeta)$$
which follows from (5.5.3) of \cite{De73}.

\end{proof}

\section{A functional equation of Kato's Euler system}
Throughout this section, we fix embeddings $\iota_{\infty}:\overline{\mathbb{Q}}\hookrightarrow \mathbb{C}$ and 
$\iota_p:\overline{\mathbb{Q}}\hookrightarrow \overline{\mathbb{Q}}_p$, and fix an isomorphism 
$\iota:\mathbb{C}\isom \overline{\mathbb{Q}}_p$ such that $\iota\circ\iota_{\infty}=\iota_p$. Using this isomorphism, we identify $\Gamma(\mathbb{C}, \mathbb{Z}_l(1))\isom 
\Gamma(\overline{\mathbb{Q}}_p, \mathbb{Z}_l(1))=:\mathbb{Z}_l(1)$, and set $\zeta^{(l)}:=\{\iota(\mathrm{exp}(\frac{2\pi i}{l^n}))\}_{n\geqq 1}\in \mathbb{Z}_l(1)$ for each 
prime $l$.
Let $S$ be a finite set of primes containing $p$. 
Let $\mathbb{Q}_{S}(\subseteq \overline{\mathbb{Q}})$ be the maximal 
Galois extension of $\mathbb{Q}$ which is unramified outside $S\cup\{\infty\}$, and set $G_{\mathbb{Q}, S}:=\mathrm{Gal}(\mathbb{Q}_S/\mathbb{Q})$. Set $c\in G_{\mathbb{Q},S}$ be the restriction by $\iota_{\infty}$ of the complex conjugation. For each $\mathbb{Z}[G_{\mathbb{R}}]$-module $M$ and $k\in \mathbb{Z}$, we define a canonical  $G_{\mathbb{R}}$-equivariant map $M(k):=M\otimes_{\mathbb{Z}}\mathbb{Z}(2\pi i)^k\rightarrow M_{\mathbb{Z}_p}(k):=(M\otimes_{\mathbb{Z}}\mathbb{Z}_p)\otimes_{\mathbb{Z}_p}\mathbb{Z}_p(k)$ by $x\otimes (2\pi i)^k\mapsto x\otimes (\zeta^{(p)})^{\otimes k}$
using the basis $\zeta^{(p)}\in \mathbb{Z}_p(1)$. We set $M^{\pm}:=M^{c=\pm 1}$.

\subsection{The generalized Iwasawa main conjecture and the global $\varepsilon$-conjecture}
In this subsection, we roughly recall the generalized Iwasawa main conjecture and the global $\varepsilon$-conjecture. See the original articles \cite{Ka93a}, \cite{Ka93b} and \cite{FK06} for the precise formulations. 

Let $T$ be an $R$-representation of $G_{\mathbb{Q}, S}$. We set 
$$\mathrm{H}^i(\mathbb{Z}[1/S], T):=\mathrm{H}^i(C_{\mathrm{cont}}^{\bullet}(G_{\mathbb{Q},S}, T))$$ for $i\geqq 0$. For each $l\in S$ and $i=1,2$, we set $$\Delta_{R,i}^{(l)}(T):=\Delta_{R,i}(T|_{G_{\mathbb{Q}_l}})\text{ and }\Delta_{R}^{(l)}(T):=\Delta_{R}(T|_{G_{\mathbb{Q}_l}})$$ which are defined in $\S 2.1$. 
Set $c_T:=\mathrm{rank}_R T(-1)^{+}$ (remark that $T(-1)^{+}$ is a finite projective $R$-module since we have assumed that $2\in R^{\times}$). We define
$$\Delta_{R,1}^{S}(T):=\mathrm{Det}_R(C^{\bullet}_{\mathrm{cont}}(G_{\mathbb{Q},S}, T))^{-1},\,\,\,\,\,\Delta_{R,2}^{S}(T):=
(\mathrm{det}_R(T(-1)^{+}), c_T )^{-1}$$
and 
$$\Delta_R^{S}(T):=\Delta_{R,1}^{S}(T)\boxtimes \Delta_{R,2}^{S}(T).$$
We remark that $\Delta_R^{S}(T)$ is a graded invertible $R$-module of degree zero by the global Euler-Poincar\'e characteristic formula. 

In \cite{Ka93a}, Kato proposed the following conjecture called the generalized Iwasawa main conjecture. See Conjecture 3.2.2 of \cite{Ka93a} or Conjecture 2.3.2 of \cite{FK06} for the precise formulation, in particular, for the interpolation condition of the zeta isomorphism.

\begin{conjecture}\label{GIMC}
For any pair $(R,T)$ such that $T$ is an $R$-representation of $G_{\mathbb{Q},S}$, one can define an $R$-linear isomorphism (which we call the zeta isomorphism)
$$z^S_R(T):\bold{1}_R\isom\Delta_R^S(T)$$
which is compatible with any base change and any exact sequence and satisfies the following: for any pair $(L, V)=(R, T)$ which comes from a motive $M$ over $\mathbb{Q}$ with coefficients in $L$, $z^{S}_L(V)$ can be described by the special value of 
the $L$-function associated to $V^*$ (under the validity of the meromorphic continuation of the $L$-function and the Deligne-Beilinson conjecture for $M$).
\end{conjecture}
\begin{rem}\label{4.2.1}
In the rank one case, Kato (in $\S 3.3$ of \cite{Ka93a}) defined the zeta isomorphism using the cyclotomic units and the Stickelberger elements. In this case, the existence of the zeta isomorphism is essentially equivalent to the Iwasawa main conjecture proved by Mazur-Wiles and Rubin, and the interpolation condition of the zeta isomorphism (i.e. the relation with the 
special values of Dirichlet $L$-functions) is essentially equivalent to the $p$-part of the Bloch-Kato's Tamagawa number conjecture for Artin motives over $\mathbb{Q}$ which is proved by Bloch-Kato \cite{BK90} (in a special case), Burns-Greither \cite{BG03} and Huber-Kings \cite{HK03} (in the general case).

\end{rem}

We next recall the global $\varepsilon$-conjecture. We first need to recall 
the definition of the canonical isomorphism 
\begin{equation}\label{PT}
\Delta_{R}^S(T^*)\isom \boxtimes_{l\in S}\Delta^{(l)}_R(T)\boxtimes \Delta^S_R(T)
\end{equation}
induced by the Poitou-Tate duality. By the Poitou-Tate duality, one has a canonical quasi-isomorphism 
$$\bold{R}\mathrm{Hom}_R(C^{\bullet}_{\mathrm{cont}}(G_{\mathbb{Q},S}, T^*),R)[-2]
\isom \mathrm{Cone}(C^{\bullet}_{\mathrm{cont}}(G_{\mathbb{Q},S}, T)\rightarrow \oplus_{l\in S}C^{\bullet}_{\mathrm{cont}}(G_{\mathbb{Q}_l}, T|_{G_{\mathbb{Q}_l}}))[-1], $$
from which we obtain a canonical isomorphism 
\begin{equation}\label{eq1}
(\Delta_{R,1}^{S}(T^*)^{-1})^{\vee}
\isom \mathrm{Det}_R(\bold{R}\mathrm{Hom}_R(C^{\bullet}_{\mathrm{cont}}(G_{\mathbb{Q},S}, T^*),R))
\isom \boxtimes_{l\in S}
\Delta_{R,1}^{(l)}(T)\boxtimes \Delta_{R,1}^{S}(T).
\end{equation}
We next define the following isomorphism 
\begin{equation}\label{eq2}
(\Delta_{R,2}^S(T^*)^{-1})^{\vee}\isom (\mathrm{det}_RT, r_T)\boxtimes \Delta^S_{R,2}(T)
\isom \boxtimes_{l\in S}\Delta_{R,2}^{(l)}(T)\boxtimes \Delta_{R,2}^S(T), 
\end{equation}
where the first isomorphism is naturally induced by the isomorphism 
$$T^{+}\oplus T(-1)^{+}\isom T:(x, y)\mapsto 2x+\frac{1}{2}y\otimes \zeta^{(p)}$$
and the canonical isomorphism
$$T^{+}\isom((T^{\vee})^{\vee})^{+}
=(T^*(-1)^{\vee})^{+}\isom (T^*(-1)^{+})^{\vee}$$
(where the last map is defined by $f\mapsto f|_{T^*(-1)^{+}}$),
and the second isomorphism is induced by the isomorphism 
$$(\mathrm{det}_RT, r_T)\isom \boxtimes_{l\in S}\Delta_{R,2}^{(l)}(T)$$ defined by (the inverse of) the isomorphism
$$\otimes_{l\in S}R_{a_{l}(T)}\otimes \mathrm{det}_RT\isom \mathrm{det}_RT:(\otimes_{l\in S} x_l)\otimes y\mapsto \left(\prod_{l\in S} x_l\right)\cdot y$$ for 
$x_l\in R_{a_l(T)}$ and $y\in \mathrm{det}_RT$
(remark that one has $\otimes_{l\in S}R_{a_{l}(T)}=R$ since one has
$\prod_{l\in S}a_l(T)=1$ by the global class field theory). Finally, the isomorphism (\ref{PT}) is defined as the product of the isomorphisms (\ref{eq1}) and (\ref{eq2}) (remark that one has $(\Delta^S_{R}(T^*)^{-1})^{\vee}=\Delta^S_{R}(T^*)$ since $\Delta^S_{R}(T^*)$ is of degree zero).


 In \cite{Ka93b}, Kato proposed the following conjecture called the global $\varepsilon$-conjecture 
 (see Conjecture 1.13 of \cite{Ka93b} or Conjecture 3.5.5 of \cite{FK06}). 
 
 \begin{conjecture} For any pair $(R,T)$ such that $T$ is an $R$-representation of $G_{\mathbb{Q},S}$, the conjectural zeta isomorphism
 $z^S_R(T_0):\bold{1}_R\isom\Delta_R^S(T_0)$ for $T_0=T, T^*$ and the conjectural $p$-adic local $\varepsilon$-isomorphism $\varepsilon_{R,\zeta^{(p)}}
 ^{(p)}(T):\bold{1}_R\isom\Delta_R^{(p)}(T)$, and the $l$-adic local $\varepsilon$-isomorphism 
 $\varepsilon_{R,\zeta^{(l)}}^{(l)}(T):\bold{1}_R\isom \Delta_R^{(l)}(T)$ defined by \cite{Ya09} for 
 $l\in S\setminus \{p\}$ satisfy the equality
 \begin{equation}
z_R^S(T^*)=\boxtimes_{l\in S}\varepsilon_{R,\zeta^{(l)}}^{(l)}(T)\boxtimes z_R^S(T)
\end{equation}
under the canonical isomorphism $\Delta_{R}^S(T^*)\isom \boxtimes_{l\in S}\Delta^{(l)}_R(T)\boxtimes \Delta^S_R(T)$.

 \end{conjecture}
 \begin{rem}
 In the rank one case and when $S=\{p\}$, this conjecture is also proved by Kato in $\S$ 4.22 of \cite{Ka93b} (and it seems to be possible to prove it for the general $S$ case in a similar way). More precisely, he proved that the zeta isomorphism defined in $\S 3.3$ of \cite{Ka93a} and  the local $p$-adic $\varepsilon$-isomorphism defined in \cite{Ka93b} for the rank one case satisfies the equality in the  conjecture above.
 
 \end{rem}
 

\subsection{Statement of the main theorem on the global $\varepsilon$-conjecture}Let $k, N\geqq 1$ be positive integers.
Let $f(\tau)=\sum_{n=1}^{\infty}a_n(f) q^n\in S_{k+1}(\Gamma_1(N))^{\mathrm{new}}$ be a normalized Hecke eigen new form of level $N$, weight $k+1$, where $\tau\in \mathbb{C}$ such that 
$\mathrm{Im}(\tau)>0$, $q:=\mathrm{exp}(2\pi i\tau)$ and 
 $\Gamma_1(N):=\left\{g\in\mathrm{SL}_2(\mathbb{Z})\,\left|\,g\equiv \begin{pmatrix}1& *\\ 0 &1\end{pmatrix} \text{ mod } N\right.\right\}$. Set 
$f^*(\tau):=\sum_{n=1}^{\infty}\overline{a_n(f)}q^n$ ( $\overline{a_n(f)}$ is the complex conjugation of $a_n(f)$), which is also a Hecke eigen 
new form in $S_{k+1}(\Gamma_1(N))^{\mathrm{new}}$ by the theory of new forms.

 For each homomorphism $\delta:\mathbb{Z}_p^{\times}\rightarrow \mathbb{C}^{\times}$ with finite image (which we naturally regard as a Dirichlet character $\delta:(\mathbb{Z}/p^{n(\delta)})^{\times}\rightarrow \mathbb{C}^{\times}$ ($n(\delta)$ is the conductor of $\delta$), or a Hecke character $\delta:\mathbb{A}_{\mathbb{Q}}^{\times}/\mathbb{Q}^{\times}\rightarrow \mathbb{C}^{\times}$), define 
 $$L(f,\delta,s):=\sum_{n\geqq 1}\frac{a_n(f)\delta(n)}{n^{s}}\text{ and } L_{\{p\}}(f,\delta,s):=\sum_{n\geqq 1, (n,p)=1}\frac{a_n(f)\delta(n)}{n^{s}}$$

 These functions 
 absolutely converge when $\mathrm{Re}(s)>\frac{k}{2}+1$. The $L$-function $L(f,\delta,s)$ is analytically continued to the whole $\mathbb{C}$, and, if we denote by $\pi_f=\otimes'_{v:\text{place of }\mathbb{Q}}\pi_{f,v}$ the automorphic cuspidal representation 
 of $\mathrm{GL}_2(\mathbb{A}_{\mathbb{Q}})$ associated to $f$, then it satisfies the following functional equation 
 \begin{equation}\label{fe}
 \Gamma_{\mathbb{C}}(s)\cdot L(f,\delta,s)=\varepsilon(f,\delta,s)\cdot \Gamma_{\mathbb{C}}(k+1-s)\cdot 
 L(f^*, \delta^{-1}, k+1-s)\,\,\, (s\in \mathbb{C}),
 \end{equation}
 where we set $\Gamma_{\mathbb{C}}(s):=\frac{\Gamma(s)}{(2\pi)^s}$ and $\varepsilon(f,\delta,s)$ is 
 the global $\varepsilon$-factor associated to $\pi_f\otimes (\delta\circ\mathrm{det})$, which is defined as the product of the local $\varepsilon$-factors
 $$\varepsilon(f,\delta,s)=\varepsilon_{\infty}(f,\delta,s)\prod_{l\in S}\varepsilon_{l}(f,\delta,s),$$
 where, for $v\in S\cup\{\infty\}$, $\varepsilon_{v}(f,\delta,s)$ is the local $\varepsilon$-factor associated to the $v$-th component $\pi_{f,v}\otimes (\delta\circ\mathrm{det})$ with respect to the additive character $\psi_v:\mathbb{Q}_v\rightarrow \mathbb{C}^{\times}$ and the Haar measure $dx_v$ on $\mathbb{Q}_v$ which are uniquely characterized by 
 $\psi_{\infty}(a):=\mathrm{exp}(-2\pi i\cdot a)$ ($a\in \mathbb{R}$), $\psi_l(\frac{1}{l^n})
 =\mathrm{exp}(\frac{2\pi i}{l^n})$ ($n\in \mathbb{Z}$), $\int_{\mathbb{Z}_l}dx_l=1$ and $dx_{\infty}$ is the standard Lebesgue measure on $\mathbb{R}$. We remark that one has 
 \begin{equation}
 \varepsilon_{\infty}(f,\delta,s)=i^{k+1}.
 \end{equation}
 
 Set $F:=\mathbb{Q}(\{\iota_{\infty}^{-1}(a_n(f))\}_{n\geqq 1})\subseteq \overline{\mathbb{Q}}$, $L:=\mathbb{Q}_p(\{\iota_p(\iota_{\infty}^{-1}(a_n(f)))\}_{n\geqq 1})\subseteq \overline{\mathbb{Q}}_p$ and $S:=\{l|N\}\cup \{p\}$. Let denote by $\mathcal{O}_F$, $\mathcal{O}_L$ the rings of integers of $F$, $L$ respectively. For $f_0=f,f^*$, let 
$T_{f_0}$ be the $\mathcal{O}_L$-representation of $G_{\mathbb{Q},S}$ of rank two associated to $f_0$ which is obtained as a quotient of the \'etale cohomology (with coefficients) of a modular curve (this is denoted by $V_{\mathcal{O}_{\lambda}}(f_0)$ in $\S$ 8.3 of \cite{Ka04}). Set $V_{f_0}=T_{f_0}[1/p]$.
By the Poincar\'e duality of the \'etale cohomology of a modular curve, one has a canonical $G_{\mathbb{Q},S}$-equivariant isomorphism $$V_{f^*}(1)\isom(V_f(k))^*,$$ 
which induces a canonical isomorphism $\Delta_L^{\mathrm{Iw},S}(V_{f^*}(1))\isom\Delta_{L}^{\mathrm{Iw},S}((V_f(k))^*)$. Since the sub $\Lambda_{\mathcal{O}_L}$-module $\Delta_{\mathcal{O}_L}^{\mathrm{Iw},S}(T)$ of $\Delta_L^{\mathrm{Iw},S}(V)$ is independent of 
the choice of $G_{\mathbb{Q},S}$-stable lattice $T$ of $V$ for any $L$-representation $V$ of $G_{\mathbb{Q},S}$ (because $\Delta_{\mathcal{O}_L}^{\mathrm{Iw},S}(T)$ is of grade zero), the latter also induces a canonical isomorphism 
$$\Delta_{\mathcal{O}_L}^{\mathrm{Iw},S}(T_{f^*}(1))\isom \Delta_{\mathcal{O}_L}^{\mathrm{Iw},S}((T_{f}(k))^*).$$
Therefore, we obtain a canonical isomorphism 
\begin{equation}\label{323}
\Delta_{\mathcal{O}_L}^{\mathrm{Iw},S}(T_{f^*}(1))^{\iota}\isom \Delta_{\mathcal{O}_L}^{\mathrm{Iw},S}((T_{f}(k))^*)^{\iota}\isom \Delta_{\Lambda_{\mathcal{O}_L}(\Gamma)}^S(\bold{Dfm}(T_f(k))^*),
\end{equation}
where the second isomorphism is defined in the same way as in the last part of \S2.1.

We denote by $Q(\Lambda_{\mathcal{O}_L}(\Gamma))$ the total fraction ring of $\Lambda_{\mathcal{O}_L}(\Gamma)$. For a $\Lambda_{\mathcal{O}_L}(\Gamma)$-module or a 
graded invertible $\Lambda_{\mathcal{O}_L}(\Gamma)$-module $M$, we set
$$M_{Q}:=M\otimes_{\Lambda_{\mathcal{O}_L}(\Gamma)}Q(\Lambda_{\mathcal{O}_L}(\Gamma))$$
to simplify the notation.
Using (the $p$-th layer of) the Kato's Euler system associated to $f_0$, we define below a (candidate of the) zeta-isomorphism
$$\widetilde{z}_{\mathcal{O}_L}^{\mathrm{Iw},S}(T_{f_0}(r)):\bold{1}_{Q(\Lambda_{\mathcal{O}_L}(\Gamma))}\isom \Delta^{\mathrm{Iw}, S}
_{\mathcal{O}_L}(T_{f_0}(r))_{Q}$$
for $f_0=f, f^*$ and $r\in \mathbb{Z}$. 

Before defining this isomorphism, we state the second main theorem of this article concerning the global $\varepsilon$-conjecture, whose proof is given in the next subsection. 

\begin{thm}\label{4.5}
 Assume $V_f|_{G_{\mathbb{Q}_p}}$ is absolutely irreducible and 
 $\bold{D}_{\mathrm{cris}}(V_f(-r)(\delta))^{\varphi=1}\allowbreak=0$ for any $0\leqq r\leqq k-1$ and 
$\delta:\Gamma\rightarrow \overline{\mathbb{Q}}_p^{\times}$ with finite image. Then, one has the equality 
$$\widetilde{z}_{\mathcal{O}_L}^{\mathrm{Iw},S}(T_{f^*}(1))^{\iota}=\boxtimes_{l\in S}\left(\varepsilon_{\mathcal{O}_L}^{\mathrm{Iw}, (l)}
(T_f(k))\otimes \mathrm{id}_{Q(\Lambda_{\mathcal{O}_L}(\Gamma))}\right)\boxtimes 
\widetilde{z}_{\mathcal{O}_L}^{\mathrm{Iw},S}(T_{f}(k))$$
under the isomorphism obtained by the base change to $Q(\Lambda_{\mathcal{O}_L}(\Gamma))$ of the canonical isomorphism 
$$
\Delta^{\mathrm{Iw}, S}_{\mathcal{O}_L}(T_{f^*}(1))^{\iota}
\isom \boxtimes_{l\in S}\Delta^{\mathrm{Iw}, (l)}_{\mathcal{O}_L}(T_f(k))\boxtimes \Delta^{\mathrm{Iw},S}_{\mathcal{O}_L}(T_f(k))
$$
defined by $(\ref{PT})$ for 
$(R,T)=\left(\Lambda_{\mathcal{O}_L}(\Gamma), \bold{Dfm}(T_f(k))\right)$ and $(\ref{323})$, 
where  the isomorphism 
$$\varepsilon_{\mathcal{O}_L}^{\mathrm{Iw}, (l)}(T_f(k)):=\varepsilon^{\mathrm{Iw}}_{\mathcal{O}_L,\zeta^{(l)}}
(T_f(k)|_{G_{\mathbb{Q}_l}}):\bold{1}_{\Lambda_{\mathcal{O}_L}(\Gamma)}\isom 
\Delta_{\mathcal{O}_L}^{\mathrm{Iw}, (l)}(T_f(k))$$ is the local $\varepsilon$-isomorphism defined by Theorem \ref{3.0} $($resp. \cite{Ya09}$)$ for $l=p$ $($resp.$l\not=p$$)$ for the pair $(\Lambda_{\mathcal{O}_L}(\Gamma), \bold{Dfm}(T_f(k)|_{G_{\mathbb{Q}_l}}))$.

\end{thm}
\begin{rem}
For $M\in \mathbb{Z}_{\geqq 1}$ such that $(M,p)=1$, we set 
$\Lambda^{(M)}_{\mathcal{O}_L}:=\mathcal{O}_L[[\mathrm{Gal}(\mathbb{Q}(\mu_{Mp^{\infty}})/\mathbb{Q})]]$ and $S^{(M)}:=S\cup\{l|M\}$. For any $\mathcal{O}_L$-representation $T$ of $G_{\mathbb{Q}, S^{(M)}}$, we define a $\Lambda^{(M)}_{\mathcal{O}_L}$-representation 
$\bold{Dfm}^{(M)}(T):=T\otimes_{\mathcal{O}_L}\Lambda^{(M)}_{\mathcal{O}_L}$ in the same way as for $\bold{Dfm}(T)$. Then, it seems to be possible to generalize the theorem above for
$\bold{Dfm}^{(M)}(T_f(k)|_{G_{\mathbb{Q},S^{(M)}}})$ without any trouble. If we try to prove it for general $M$, we need to slightly generalize Theorem 12.4 and Theorem 12.5 of \cite{Ka04}, which seems to be easy but makes this article longer. Since the present article is already long enough, 
we only treat the $M=1$ case.
\end{rem}

In the rest of this subsection, we define our zeta isomorphism $\widetilde{z}_{\mathcal{O}_L}^{\mathrm{Iw},S}(T_{f_0}(r))$ using the $p$-th layer of the Kato's Euler system which we recall now. Since the definitions for $f$ and $f^*$ are the same, we only define it for $f_0=f$.

For an $\mathcal{O}_L$-representation $T$ of $G_{\mathbb{Q},S}$ (which we also regard
as a smooth $\mathcal{O}_L$-sheaf on the \'etale site over $\mathrm{Spec}(\mathbb{Z}[1/S])$), we define 
a $\Lambda_{\mathcal{O}_L}(\Gamma)$-module 
 $$\bold{H}^i(T):=\mathrm{H}^i_{\mathrm{Iw}}(\mathbb{Z}[1/p], T)
 :=\varprojlim_{n\geqq 0}\mathrm{H}^i(\mathbb{Z}[1/p, \zeta_{p^n}], T)$$ for $i\geqq 0$, where we define for $n\geqq 0$
 $$\mathrm{H}^i(\mathbb{Z}[1/p, \zeta_{p^n}], T):=\mathrm{H}^i_{\text{\'et}}(\mathrm{Spec}(\mathbb{Z}[1/p, \zeta_{p^n}]), (j_{n})_{*}(T|_{\mathrm{Spec}(\mathbb{Z}[1/S, \zeta_{p^n}])}))$$ 
 for the canonical inclusion $j_n:\mathrm{Spec}(\mathbb{Z}[1/S, \zeta_{p^n}])\hookrightarrow \mathrm{Spec}(\mathbb{Z}[1/p,\zeta_{p^n}])$. For $V=T[1/p]$, we set
 $\bold{H}^i(V):=\bold{H}^i(T)[1/p].$ 
  
 For the eigen form $f$, Kato defined in Theorem 12.5 \cite{Ka04} an $L$-linear map 
 $$V_f\rightarrow \bold{H}^1(V_f):\gamma\mapsto \bold{z}_{\gamma}^{(p)}(f)$$ which interpolates the critical values of the $L$-functions $L(f^*,\delta,s)$ for all $\delta$, whose precise meaning we explain in the next subsection. 
 By Theorem 12.4 of \cite{Ka04}, $\bold{H}^1(T_f)$ is torsion free over $\Lambda_{\mathcal{O}_L}(\Gamma)$, and $\bold{H}^1(V_f)$ is a free $\Lambda_L(\Gamma)$-module of rank one, and $\bold{H}^2(T_f)$ is a torsion $\Lambda_{\mathcal{O}_L}(\Gamma)$-module (and $\bold{H}^i(T_f)=0$ for $i\not= 1,2$). The restriction map $\bold{H}^i(T_f)\rightarrow \mathrm{H}_{\mathrm{Iw}}^i(\mathbb{Z}[1/S], T_f)$ induces an isomorphism 
 $$\bold{H}^1(T_f)\isom \mathrm{H}^1_{\mathrm{Iw}}(\mathbb{Z}[1/S], T_f)$$ 
 and an exact sequence 
 $$0\rightarrow \bold{H}^2(T_f)\rightarrow \mathrm{H}^2_{\mathrm{Iw}}(\mathbb{Z}[1/S], T_f)\rightarrow \bigoplus_{l\in S\setminus \{p\}}\mathrm{H}^2_{\mathrm{Iw}}(\mathbb{Q}_l, T_f)\rightarrow 0,$$
 which follow from (for example) the proof of Lemma 8.5 of \cite{Ka04}. Since $\mathrm{H}^2_{\mathrm{Iw}}(\mathbb{Q}_l, T)$ is a torsion $\Lambda_{\mathcal{O}_L}(\Gamma)$-module for any $l$ by Proposition A.2.3 of \cite{Pe95}, $\mathrm{H}^2_{\mathrm{Iw}}(\mathbb{Z}[1/S], T_f)$ is also a torsion $\Lambda_{\mathcal{O}_L}(\Gamma)$-module by the above exact sequence. By these facts, we obtain a canonical $Q(\Lambda_{\mathcal{O}_L}(\Gamma))$-linear isomorphism 
 \begin{equation}
 \Delta^{\mathrm{Iw},S}_{\mathcal{O}_L,1}(T_f(r))_{Q}
 \isom (\bold{H}^1(T_f(r))_{Q}, 1)
 \end{equation}
 for $r=0$. For general $r\in \mathbb{Z}$, we also define the  isomorphism above induced by 
 that for $=0$ using the canonical (not $\Lambda_{\mathcal{O}_L}(\Gamma)$-linear) isomorphism 
 $$\bold{H}^i(T_f)\isom \bold{H}^i(T_f(r)):\bold{z}\mapsto \bold{z}(r)$$ 
 which is induced by the isomorphism 
 $$T_f\otimes_{\mathcal{O}_L}\Lambda_{\mathcal{O}_L}(\Gamma)\isom 
 T_f(r)\otimes_{\mathcal{O}_L}\Lambda_{\mathcal{O}_L}(\Gamma):x\otimes y\mapsto (x\otimes \bold{e}_r)\otimes g_{\chi^r}(y)$$
 defined in the same way as in the proof of Lemma \ref{2.10.1}.

 For each $l\in S\setminus \{p\}$ and $r\in \mathbb{Z}$, we set 
 $$L^{(l)}_{\mathrm{Iw}}(T_f(r)):=\mathrm{det}_{\Lambda_{\mathcal{O}_L}(\Gamma)}(1-\varphi_l|\bold{Dfm}(T_f(r))^{I_l})
 =1-a_l(f)\cdot l^{-r}\cdot[\sigma_l]
 \in \Lambda_{\mathcal{O}_L}(\Gamma)\cap Q(\Lambda_{\mathcal{O}_L}(\Gamma))^{\times}$$
 (remark that $\bold{Dfm}(T_f(r))^{I_l}
 =T_f^{I_l}(r)\otimes_{\mathcal{O}_L}\Lambda_{\mathcal{O}_L}(\Gamma)$ is free over
 $\Lambda_{\mathcal{O}_L}(\Gamma)$), where the second equality follows from the global-local compatibility of the Langlands correspondence proved by \cite{La73}, \cite{Ca86}.

Denote the sign of $(-1)^r$ by $\mathrm{sgn}(r)\in \{\pm\}$. Set $\Lambda^{\pm}:=\{\lambda\in \Lambda_{\mathcal{O}_L}(\Gamma)|[\sigma_{-1}]\cdot \lambda =\pm\lambda\}$.
 
 Using these, we define an isomorphism $$\widetilde{z}^{\mathrm{Iw},S}_{\mathcal{O}_L}(T_f(r)):\bold{1}_{Q(\Lambda_{\mathcal{O}_L}(\Gamma))}\isom \Delta^{\mathrm{Iw},S}_{\mathcal{O}_L}(T_f(r))_{Q}$$ which corresponds to the isomorphism
    $$\theta_r(f):
 \Delta^{\mathrm{Iw},S}_{\mathcal{O}_L,2}(T_f(r))^{-1}_Q\isom (\bold{H}^1(T_f(r))_Q, 1)$$ defined as the base change to $Q(\Lambda_{\mathcal{O}_L}(\Gamma))$ of the $\Lambda_{\mathcal{O}_L}(\Gamma)$-linear morphism
 \begin{multline}
 \Theta_r(f):(\bold{Dfm}(T_f(r))(-1))^+=T_f^{\mathrm{sgn}(r-1)}(r-1)\otimes_{\mathcal{O}_L} \Lambda^+\oplus 
 T_f^{\mathrm{sgn}(r)}(r-1)\otimes_{\mathcal{O}_L} \Lambda^-\rightarrow \bold{H}^1(V_f(r)):\\
 (\gamma\otimes \bold{e}_{r-1}\otimes \lambda^+, 
 \gamma'\otimes \bold{e}_{r-1}\otimes \lambda^-)\mapsto \prod_{l\in S\setminus \{p\}}
 L^{(l)}_{\mathrm{Iw}}(T_{f^*}(1+k-r))^{\iota}\cdot (
 \lambda^+\cdot(\bold{z}^{(p)}_{\gamma}(f)(r))+ \lambda^-\cdot(\bold{z}^{(p)}_{\gamma'}(f)(r))),
 \end{multline}
 where we set $\lambda^{\iota}:=\iota(\lambda)$ for $\lambda\in Q(\Lambda_L(\Gamma))$, and the fact that the base change to $Q(\Lambda_{\mathcal{O}_L}(\Gamma))$ of this morphism is isomorphism 
 follows from Theorem 12.5 of \cite{Ka04}.

  \subsection{Proof of Theorem \ref{4.5}}
 In this subsection, we give a proof of Theorem \ref{4.5}. 
 We first precisely recall the interpolation property of the Kato's Euler system which is so called 
 called the explicit reciprocity law (Theorem 12.5 (1) of \cite{Ka04}), which is crucial in our 
 proof of the theorem. 
 
Using the comparison between the Betti and the \'etale cohomologies, one has a canonical $\mathcal{O}_F$-lattice $T_{f,\mathcal{O}_F}$ of $T_f$ which is 
 stable by the action of $G_{\mathbb{R}}(\subseteq G_{\mathbb{Q},S})$. Set $V_{f,F}:=T_{f,\mathcal{O}_F}\otimes_{\mathcal{O}_F}F$, which is a $G_{\mathbb{R}}$-stable $F$-lattice of $V_f$. Using the comparison theorem in the $p$-adic Hodge theory, 
 oen has a canonical 
 $F$-lattice $S(f)=F\cdot f$ of $\bold{D}_{\mathrm{dR}}^1(V_f)=\bold{D}_{\mathrm{dR}}^{k}(V_f)$. 

 By the theory of Eichler-Shimura, one has a canonical $F$-linear map 
 $$\mathrm{per}_f:S(f)\hookrightarrow V_{f,\mathbb{C}}:=V_{f,F}\otimes_{F,\iota_{\infty}}\mathbb{C}. $$
 
 For each pair $(r,\delta)$ such that $0\leqq r\leqq k-1$ and $\delta:\Gamma\rightarrow \overline{\mathbb{Q}}^{\times}$ a homomorphism with finite image (which we regard as a homomorphism with values in $\mathbb{C}^{\times}$ or $\overline{\mathbb{Q}}_p^{\times}$ by the fixed embeddings $\iota_{\infty}$ or $\iota_p$), we set $V_f(k-r)(\delta):=V_f(k-r)\otimes_L\overline{\mathbb{Q}}_p(\delta)$,  and define an $L$-linear map 
 \begin{equation}\label{sp}
 \bold{H}^1(V_f)\rightarrow \bold{D}^0_{\mathrm{dR}}(V_f(k-r)(\delta))=L(f\otimes\frac{1}{t^{k-r}}\bold{e}_{k-r})\otimes_{L}(\overline{\mathbb{Q}}_{p,\infty}(\delta))^{\Gamma}
 \end{equation}
 as the composites of the following morphisms
 \begin{equation*}
 \bold{H}^1(V_f)\xrightarrow{\mathrm{can}} \mathrm{H}^1_{\mathrm{Iw}}(\mathbb{Q}_p, V_f)
 \xrightarrow{\mathrm{sp}_{\chi^{k-r}\cdot\delta}} \mathrm{H}^1(\mathbb{Q}_p, V_f(k-r)(\delta))
 \xrightarrow{\mathrm{exp}^*}\bold{D}^0_{\mathrm{dR}}(V_f(k-r)(\delta)),
   \end{equation*}
 
 For $\gamma\in V_{f,F}$, we decompose $\gamma=\gamma^++\gamma^{-}$ such that $\gamma^{\pm} \in V_{f,F}^{\pm}$. For each $(r,\delta)$ as above, we denote by $\mathrm{sgn}(r,\delta)\in \{\pm\}$ the sign of $\delta(-1)(-1)^{r}$. 
 

Under these preliminaries, the interpolation property of $\bold{z}_{\gamma}^{(p)}(f)$ can be described as follows. By Theorem 12.5 (1) of \cite{Ka04}, the image of 
 $\bold{z}^{(p)}_{\gamma}(f)$ by the map ($\ref{sp}$) is contained in the sub $F$-vector space 
 $$F(f\otimes \frac{1}{t^{k-r}}\bold{e}_{k-r})\otimes_{\mathbb{Q}}(\overline{\mathbb{Q}}_{\infty}(\delta))^{\Gamma}$$ of $\bold{D}^0_{\mathrm{dR}}(V_f(k-r)(\delta))$, and is sent to 
  \begin{equation}\label{law}
  (2\pi i)^{k-r-1}\cdot L_{\{p\}}(f^*, \delta^{-1}, r+1)\cdot \gamma^{\mathrm{sgn}(k-r-1,\delta)}\in V_{f,\mathbb{C}}^{\mathrm{sgn}(k-r-1,\delta)}.
   \end{equation} by the injection map defined by the following composite
 $$\mathrm{per}_{f}^{(k-r,\delta)}:F(f\otimes \frac{1}{t^{k-r}}\bold{e}_{k-r})\otimes_{F}(\overline{\mathbb{Q}}_{\infty}(\delta))^{\Gamma}\rightarrow V_{f,\mathbb{C}}
 \rightarrow V_{f,\mathbb{C}}^{\mathrm{sgn}(k-r-1,\delta)}$$
 where the first map is defined by 
 $$(f\otimes \frac{1}{t^{k-r}}\bold{e}_{k-r})\otimes (\sum_{i\in I}
 b_i\otimes c_i\bold{e}_{\delta})\mapsto \iota_{\infty}(\sum_{i\in I}b_ic_i)\mathrm{per}_f(f)$$ 
 for $b_i\in \overline{\mathbb{Q}}, c_i\in \mathbb{Q}(\zeta_{p^{\infty}})(\subseteq \overline{\mathbb{Q}})$, 
 and the second map is the canonical projection $V_{f,\mathbb{C}}\rightarrow V_{f,\mathbb{C}}^{\pm}:x\mapsto \frac{1}{2}(x\pm c(x))$.

 
 
 From now on until the end of the article, we assume that 
 $V_f|_{G_{\mathbb{Q}_p}}$ is absolutely irreducible and 
 $\bold{D}_{\mathrm{cris}}(V_f(-r)(\delta))^{\varphi=1}\allowbreak=0$ for any $0\leqq r\leqq k-1$ and 
$\delta:\Gamma\rightarrow \overline{\mathbb{Q}}_p^{\times}$ with finite image. 
 
  Under this assumption, we next reduce Theorem \ref{4.5} to the following Theorem \ref{4.8} below concerning the equality of the two elements in $D_{f^*}(1)^{\psi=1}$ respectively defined by using $\bold{z}_{\gamma}^{(p)}(f)(k)$ and $\bold{z}_{\gamma'}^{(p)}(f^*)(1)$. We denote by $D_{f_0}$ the \'etale $(\varphi,\Gamma)$-module over $\mathcal{E}_L$ associated to $V_{f_0}|_{G_{\mathbb{Q}_p}}$ for $f_0=f,f^*$. 
 By the following canonical morphisms (for $r\in \mathbb{Z}$)
 $$\bold{H}^1(V_{f_0}(r))\xrightarrow{\mathrm{can}}  \mathrm{H}^1_{\mathrm{Iw}}(\mathbb{Q}_p, V_{f_0}(r))
 \isom D_{f_0}(r)^{\psi=1}\xrightarrow{1-\varphi} D_{f_0}(r)^{\psi=0}, $$
 we freely regard $\bold{z}_{\gamma}^{(p)}(f_0)(r)$ as an element in these modules.

Fix an $\mathcal{O}_F$-basis $\gamma^{\pm}$ of $T_{f,\mathcal{O}_F}^{\pm}$ for each $\pm$, and set 
$$\gamma:=\gamma^++\gamma^-\in T_{f,\mathcal{O}_F}\text{ and } \bold{f}_{\gamma}:=(\gamma^{\mathrm{sgn}(k)}\bold{e}_k)\wedge (\gamma^{\mathrm{sgn}(k-1)}\bold{e}_k)\in \mathrm{det}_{\mathcal{O}_F}T_{f,\mathcal{O}_F}(k).$$ 
We define the basis $\gamma_*^{\pm}$ of $V_{f^*, F}^{\pm}$ such that $\{\gamma_*^{+}, \gamma_*^{-}\}$ is the dual basis of $\{\gamma^{\mathrm{sgn}(k)}\bold{e}_k, \gamma^{\mathrm{sgn}(k-1)}\bold{e}_k\}$ under
 the canonical $F$-bilinear perfect pairing 
 $$V_{f^*,F}\times V_{f,F}(k)\rightarrow F$$
 induced by the Poincar\'e duality. We also set 
 $$\gamma_*:=\gamma_*^{+}+\gamma_*^{-}\in V_{f^*,F}.$$ 
 
 For each $l \in S\setminus \{p\}$, set 
 $$\varepsilon^{(l)}_{0,\mathcal{O}_L}(T_f(k)):=\varepsilon_{0,\mathcal{O}_L}(T_f(k)|_{G_{\mathbb{Q}_l}},\zeta^{(l)})\in (\mathcal{O}_L)^{\times}_{a_l(T_f(k))}$$
 the $\varepsilon_0$-constant associated to the triple $(\mathcal{O}_L,T_f(k)|_{G_{\mathbb{Q}_l}},\zeta^{(l)})$ defined by \cite{Ya09} (see Remark \ref{Ya}). Using the canonical isomorphism $$\otimes_{l\in S}(\mathcal{O}_L)_{a_l(T_f(k))}\isom \mathcal{O}_L:\otimes_{l\in S} x_l\mapsto \prod_{l\in S} x_l,$$ we define 
 $$\varepsilon_{0}:=\otimes_{l\in S\setminus \{p\}}\varepsilon^{(l)}_{0,\mathcal{O}_L}(T_f(k))^{-1}\in 
 (\otimes_{l\not =p\in S}(\mathcal{O}_L)_{a_l(T_f(k))^{-1}})^{\times}=(\mathcal{O}_L)^{\times}_{a_p(T_f(k))}.$$  Using this, we define a basis 
 $$\bold{e}_{\gamma}:=\bold{f}_{\gamma}\otimes \varepsilon_{0}\in \mathcal{L}_{\mathcal{O}_L}(T_f(k))=\mathrm{det}_{\mathcal{O}_L}T_f(k)\otimes_{\mathcal{O}_L} (\mathcal{O}_L)_{a_{p}(T_f(k))}.$$
 
 For $l\in S$ and an $L$-representation $V$ of $G_{\mathbb{Q},S}$ such that $V|_{G_{\mathbb{Q}_p}}$ is de Rham, we denote by
 $a^{(l)}(V)$ the exponent of the Artin conductor of $W(V|_{G_{\mathbb{Q}_l}})$ defined in 8.12 of \cite{De73}, 
 and set $L^{(l)}(V):=\mathrm{det}_L(1-\varphi_l|\bold{D}_{\mathrm{cris}}(V|_{G_{\mathbb{Q}_l}}))\in L$ and $\varepsilon_L^{(l)}(V):=\varepsilon_L(W(V|_{G_{\mathbb{Q}_l}}), \zeta^{(l)})\in L_{\infty}$, which we also regard as elements in $\mathbb{C}$ by the injection $L\subseteq \overline{\mathbb{Q}}_p\xrightarrow{\iota^{-1}}\mathbb{C}$ and the projection $L_{\infty}=L\otimes_{\mathbb{Q}}\mathbb{Q}(\zeta_{p^{\infty}})\rightarrow \mathbb{C}:a\otimes b\mapsto \iota^{-1}(a)\cdot \iota_{\infty}(b)$.

 \begin{thm}\label{4.8}
 Under the  situation above, one has the equality 
 $$\prod_{l\in S\setminus\{p\}} \frac{[\sigma_l]^{-a^{(l)}(V_f(k))}}{\mathrm{det}_L(-\varphi_l|V_f(k)^{I_l})} \cdot (w_{\delta_{D_f(k)}}(\bold{z}^{(p)}_{\gamma}(f)(k))\otimes \bold{e}_{\gamma}^{\vee}\otimes 
 \bold{e}_1)=-\bold{z}_{\gamma_*}^{(p)}(f^*)(1)$$
 
 under the canonical isomorphism 
 $$w_{\delta_{D_f(k)}}(D_f(k)^{\psi=1})\otimes_L\mathcal{L}_L(D_f(k))^{\vee}(1)\isom (D_f(k)^*)^{\psi=1}\isom D_{f^*}(1)^{\psi=1},$$ where the first isomorphism is defined in \S3 and the second one is induced by the canonical isomorphism $D_f(k)^*\isom D_{f^*}(1)$ defined by the Poincar\'e duality.

 \end{thm}
 \begin{rem}\label{4.8.5}
 Since one has 
 $$[\sigma_{-1}]\cdot \bold{z}^{(p)}_{\gamma^{\pm}}(f)=\mp\bold{z}^{(p)}_{\gamma^{\pm}}(f)$$ (and similarly for $\bold{z}^{(p)}_{\gamma^{\pm}_*}(f^*))$ by Theorem 12.5 of 
 \cite{Ka04}, the equality in 
 the  theorem above is equivalent to the equation
 $$\prod_{l\in S\setminus\{p\}} \frac{[\sigma_l]^{-a^{(l)}(V_f(k))}}{\mathrm{det}_L(-\varphi_l|V_f(k)^{I_l})} \cdot[\sigma_{-1}]\cdot (w_{\delta_{D_f(k)}}(\bold{z}^{(p)}_{\gamma'}(f)(k))\otimes \bold{e}_{\gamma}^{\vee}\otimes 
 \bold{e}_1)=\pm\bold{z}_{\gamma'_*}^{(p)}(f^*)(1)$$
 for each $(\gamma', \gamma'_*,\pm)\in \{(\gamma^{\mathrm{sgn}(k)}, \gamma_*^-,+), (\gamma^{\mathrm{sgn}(k-1)}, \gamma_*^+,-)\}$.

 \end{rem}
 
 Before proceeding to the proof of Theorem \ref{4.8}, we first prove Theorem \ref{4.5} using 
 this Theorem (as can be easily seen from the proof below, Theorem \ref{4.8} is in fact equivalent to Theorem \ref{4.5}).
 
 \begin{proof} (of Theorem \ref{4.5})
 We first explicitly describe the base change to $Q(\Lambda_{\mathcal{O}_L}(\Gamma))$ of the $l$-adic $\varepsilon$-isomorphism 
 $$\varepsilon^{\mathrm{Iw}, (l)}(T_f(k)):\bold{1}_{\Lambda_{\mathcal{O}_L}(\Gamma)}\isom  \Delta_{\mathcal{O}_L}^{\mathrm{Iw}, (l)}(T_f(k))$$ for $l\in S\setminus \{p\}$ which is defined in \cite{Ya09} 
 (and Remark \ref{Ya}). We first remark that, if we set $\varepsilon_{0, \mathcal{O}_L}
 ^{\mathrm{Iw}, (l)}(T_f(k)):=\varepsilon_{0, \Lambda_{\mathcal{O}_L}(\Gamma)}(
 \bold{Dfm}(T_f(k))|_{G_{\mathbb{Q}_l}}, \zeta^{(l)})$, then we have 
 \begin{equation}\label{39.123}
 \varepsilon_{0, \mathcal{O}_L}
 ^{\mathrm{Iw}, (l)}(T_f(k))=[\sigma_l]^{a^{(l)}(V_f(k))+\mathrm{dim}_LV_f(k)^{I_l}}\cdot \varepsilon^{(l)}_{0,\mathcal{O}_L}(T_f(k))
 \end{equation}
 
 since one has
 \begin{equation}\label{40.111}
 \varepsilon_{0, \overline{\mathbb{Q}}_p}(V_f(k)(\delta)|_{G_{\mathbb{Q}_l}},\zeta^{(l)})=\delta(\sigma_l)^{-(a^{(l)}(V_f(k))+\mathrm{dim}_LV_f(k)^{I_l})}\cdot \varepsilon^{(l)}_{0,\mathcal{O}_L}(T_f(k))
 \end{equation} for any 
 continuous homomorphism $\delta:\Gamma\rightarrow \overline{\mathbb{Q}}_p^{\times}$ by 
(5.5.1) and (8.12.1) of \cite{De73}. Since $\mathrm{H}^{i}_{\mathrm{Iw}}(\mathbb{Q}_l, T_f(k))$ is a torsion 
 $\Lambda_{\mathcal{O}_L}(\Gamma)$-module for any $i$ by Proposition A.2.3 of \cite{Pe95}, we have 
 $$\Delta^{\mathrm{Iw}, (l)}_{\mathcal{O}_L,1}(T_f(k))_Q=\bold{1}_{Q(\Lambda_{\mathcal{O}_L}(\Gamma))}.$$ Then, the base change to $Q(\Lambda_{\mathcal{O}_L}(\Gamma))$ of the isomorphism $\bold{1}_{\Lambda_{\mathcal{O}_L}(\Gamma)}\isom\Delta^{\mathrm{Iw}, (l)}_{\mathcal{O}_L, 1}(T_f(k))$ defined in Remark \ref{Ya} is explicitly defined by 
 \begin{multline}\label{40.123}
 \bold{1}_{Q(\Lambda_{\mathcal{O}_L}(\Gamma))}\isom\Delta^{\mathrm{Iw}, (l)}_{\mathcal{O}_L, 1}(T_f(k))_Q=\bold{1}_{Q(\Lambda_{\mathcal{O}_L}(\Gamma))}:\\
 1\mapsto \frac{\mathrm{det}_{\Lambda_{\mathcal{O}_L}(\Gamma)}
 (1-\varphi_l^{-1}|\bold{Dfm}(T_f(k))^{I_l})} {\mathrm{det}_{\Lambda_{\mathcal{O}_L}(\Gamma)}
 (1-\varphi_l|\bold{Dfm}(T_f^*(1))^{I_l})^{\iota}}
=\frac{\mathrm{det}_L(-\varphi_l^{-1}|V_f^{I_l}(k))}{[\sigma_l]^{\mathrm{dim}_LV_f(k)^{I_l}}}
\cdot \frac{L_{\mathrm{Iw}}^{(l)}(T_f(k))}{L^{(l)}_{\mathrm{Iw}}(T_{f^*}(1))^{\iota}}.
 \end{multline}

 Since we have 
 $$(\Lambda_{\mathcal{O}_L}(\Gamma))_{a_l(\bold{Dfm}(T))}=(\mathcal{O}_L)_{a_l(T)}\otimes_{\mathcal{O}_L}\Lambda_{\mathcal{O}_L}(\Gamma)$$ for any $\mathcal{O}_L$-representation 
 $T$ of $G_{\mathbb{Q}_l}$, (\ref{39.123}) and (\ref{40.123}) imply that the base change to $Q(\Lambda_{\mathcal{O}_L}(\Gamma))$ of 
 $\varepsilon_{\mathcal{O}_L}^{\mathrm{Iw}, (l)}(T_f(k))$ can be explicitly defined by 
  \begin{multline}\label{47.111}
\bold{1}_{Q(\Lambda_{\mathcal{O}_L}(\Gamma))}\isom\Delta^{\mathrm{Iw}, (l)}_{\mathcal{O}_L}(T_f(k))_Q
 =\Delta^{(l)}_{\mathcal{O}_L,2}(T_f(k))\otimes_{\mathcal{O}_L} Q(\Lambda_{\mathcal{O}_L}(\Gamma)):\\
 1\mapsto \varepsilon^{(l)}_{0,\mathcal{O}_L}(T_f(k))\otimes\frac{[\sigma_l]^{a^{(l)}(V_f(k))}}{\mathrm{det}_L(-\varphi_l|V_f^{I_l}(k))}
 \cdot \frac{L_{\mathrm{Iw}}^{(l)}(T_f(k))}{L^{(l)}_{\mathrm{Iw}}(T_{f^*}(1))^{\iota}}.
 \end{multline}
 
 Using this explicit expression, we prove Theorem \ref{4.5} as follows. 

 To show the theorem, it suffices to show the following diagram 
 of $Q(\Lambda_{\mathcal{O}_L}(\Gamma))$-linear isomorphisms is commutative:
 \begin{equation}\label{42.123}
\begin{CD}
\bold{Dfm}(T_f(k))(-1)^+_Q\otimes_{Q(\Lambda_{\mathcal{O}_L}(\Gamma))} (\bold{Dfm}(T_{f^*}(1))^{\iota}(-1)^+_Q)^{\vee}@> (a)>> (\mathrm{det}_{\Lambda_{\mathrm{O}_L}(\Gamma)}\bold{Dfm}(T_f(k)))_Q\\
@VV \Theta_k(f)\otimes ((\Theta_1(f^*)^{\iota})^{-1})^{\vee} V  @ AA (c) A \\
\bold{H}^1(T_f(k))_Q
\otimes_{Q(\Lambda_{\mathcal{O}_L}(\Gamma))} (\bold{H}^1(T_{f^*}(1))^{\iota}_Q)^{\vee}
@>> (b) > \mathrm{det}_{Q(\Lambda_{\mathcal{O}_L}(\Gamma))}\mathrm{H}^1_{\mathrm{Iw}}(\mathbb{Q}_p, T_f(k))_Q.
\end{CD}
\end{equation}
Here the arrows (a), (b) and (c) in the  diagram above is defined as follows. 

First, the isomorphism (a) is the base change to 
$Q(\Lambda_{\mathcal{O}_L}(\Gamma))$ of the canonical isomorphism
\begin{multline}
\bold{Dfm}(V_f(k))(-1)^+\otimes_{\Lambda_{L}(\Gamma)} (\bold{Dfm}(V_{f^*}(1))^{\iota}(-1)^+)^{\vee}\rightarrow (\mathrm{det}_{\Lambda_{L}(\Gamma)}\bold{Dfm}(V_f(k))):\\
(\lambda_1^+\cdot\gamma^{\mathrm{sgn}(k-1)}\bold{e}_{k-1}+\lambda_1^-\cdot\gamma^{\mathrm{sgn}(k)}\bold{e}_{k-1})\otimes (\lambda_2^{+}\cdot_{\iota}(\gamma^+_*)^{\vee}+\lambda_2^{-}\cdot_{\iota}(\gamma^-_*)^{\vee})\mapsto (\lambda_1^-\cdot \lambda_2^--\lambda_1^+\cdot \lambda_2^+)\cdot\bold{f}_{\gamma}
\end{multline}
for $\lambda^{\pm}_{i}\in \Lambda^{\pm}$ ($i=1,2$) (remark that we have $(\gamma^+_*)^{\vee}=\gamma^{\mathrm{sgn}(k)}\bold{e}_k, (\gamma_*^-)^{\vee}
=\gamma^{\mathrm{sgn}(k-1)}\bold{e}_k$). 

The isomorphism (b) is the isomorphism naturally induces by the short exact sequence
\begin{equation}\label{nnnn}
0\rightarrow \bold{H}^1(T_f(k))_Q\rightarrow 
\mathrm{H}^1_{\mathrm{Iw}}(\mathbb{Q}_p, T_f(k))_Q
\rightarrow (\bold{H}^1(T_{f^*}(1))^{\iota})^{\vee}_Q\rightarrow 0,
\end{equation}
which is obtained by the base change to $Q(\Lambda_{\mathcal{O}_L}(\Gamma))$ of the Poitou-Tate 
exact sequence for the pair $(\Lambda_{\mathcal{O}_L}(\Gamma), \bold{Dfm}(T_f(k)))$ (and the 
$\Lambda_{\mathcal{O}_L}(\Gamma)$-torsionness of $\mathrm{H}^2_{\mathrm{Iw}}(\mathbb{Z}[1/S], T_{f_0}(r))$ for $f_0=f,f^*$, $r\in \mathbb{Z}$, and that of $\mathrm{H}^1_{\mathrm{Iw}}(\mathbb{Q}_l, T_f(k))$ for $l\in S\setminus \{p\}$).

Finally, the isomorphism (c) 
$$(\mathrm{det}_{\Lambda_{L}(\Gamma)}\mathrm{H}^1_{\mathrm{Iw}}(\mathbb{Q}_p, V_f(k)))_Q
\isom (\mathrm{det}_{\Lambda_{L}(\Gamma)}\bold{Dfm}(V_f(k)))_Q$$
is defined by sending $(x\wedge y)\otimes \lambda$ for $x,y \in \mathrm{H}^1_{\mathrm{Iw}}(\mathbb{Q}_p, V_f(k))\isom D_f(k)^{\psi=1}$, $\lambda\in Q(\Lambda_{\mathcal{O}_L}(\Gamma))$ to 
$$
\lambda\cdot\{[\sigma_{-1}](w_{\delta_{D_f(k)}}(x)\otimes \bold{e}_{\gamma}^{\vee}\otimes \bold{e}_1), y\}_{\mathrm{Iw}}\cdot\prod_{l\in S \setminus \{p\}}\frac{[\sigma_l]^{a^{(l)}(V_f(k))}}{\mathrm{det}_L(-\varphi_l|V_f^{I_l}(k))} \cdot \frac{L_{\mathrm{Iw}}^{(l)}(T_f(k))}{L^{(l)}_{\mathrm{Iw}}(T_{f^*}(1))^{\iota}}\cdot\bold{f}_{\gamma}$$
(remark that here we use the definition of $\varepsilon^{\mathrm{Iw},(p)}_{\mathcal{O}_L}(T_f(k))$ given in 
\S3 and the explicit description (\ref{47.111}) of $\varepsilon^{\mathrm{Iw},(l)}_{\mathcal{O}_L}(T_f(k))$ for $\in S\setminus\{p\})$.

By definition of these maps, the commutativity of the diagram (\ref{42.123}) is equivalent to the equalities 
$$\prod_{l\in S \setminus \{p\}}\frac{[\sigma_l]^{a^{(l)}(V_f(k))}}{\mathrm{det}_L(-\varphi_l|V_f^{I_l}(k))}\cdot \{[\sigma_{-1}]\cdot (w_{\delta_{D_f(k)}}(\bold{z}^{(p)}_{\gamma'}(f)(k))\otimes \bold{e}_f^{\vee}\otimes \bold{e}_1), \widetilde{(\bold{z}^{(p)}_{\gamma'_*}(f^*)(1))^{\vee}}\}_{\mathrm{Iw}}
=\pm 1$$
for each $(\gamma',\gamma'_*,\pm)\in \{(\gamma^{\mathrm{sgn}(k)}, \gamma_*^{-},+), (\gamma^{\mathrm{sgn}(k-1)}, \gamma_*^+,-)\}$, where we denote by $\widetilde{y}\in \mathrm{H}^1_{\mathrm{Iw}}(\mathbb{Q}_p, T_f(k))_Q$ a lift of $y\in (\bold{H}^1(T_{f^*}(1))^{\iota}_Q)^{\vee}$ by the surjection in (\ref{nnnn}). This equality follows from Theorem 
\ref{4.8} and Remark \ref{4.8.5}.

 \end{proof}
 
 Finally, we prove Theorem \ref{4.8}.

\begin{proof} (of Theorem \ref{4.8}) In this proof, we freely use the notations which are used in 
$\S3.3$. To simplify the notation, we set $D:=D_f(k)$. We identify $D^*\isom D_{f^*}(1)$ by the canonical isomorphism induced the Poincar\'e duality.

Applying Corollary \ref{3.12.123}, it suffices to show the equality
\begin{multline}\label{32}
\prod_{l\in S\setminus\{p\}}\frac{(l^{r}\cdot\delta(\sigma_l)^{-1})^{a^{(l)}(V_f(k))}}{
\mathrm{det}_L(-\varphi_l|V_f(k)^{I_l})}\cdot \mathrm{exp}^*((w_{\delta_D}(\bold{z}^{(p)}_{\gamma}(f)(k))\otimes \bold{e}^{\vee}_{\gamma}\otimes\bold{e}_1)_{r,\delta^{-1}})\\
=-\mathrm{exp}^*((\bold{z}_{\gamma_*}^{(p)}(f^*)(1))_{r,\delta^{-1}})
\end{multline}
in $\bold{D}_{\mathrm{dR}}^0(D^*(r)(\delta^{-1}))$ for all the pairs $(r,\delta)$ such that 
$0\leqq r\leqq k-1$ and $\delta:\Gamma\rightarrow \overline{\mathbb{Q}}_p^{\times}$ with finite image.

Take any pair $(r,\delta)$ as above. We first remark that we have 
\begin{multline}
\bold{D}_{\mathrm{dR}}^0(D^*(r)(\delta^{-1}))=\bold{D}_{\mathrm{dR}}^0(V_{f^*}(1+r)(\delta^{-1}))\\
=\overline{\mathbb{Q}}_p\cdot f^*\otimes \frac{1}{t^{1+r}}\bold{e}_{1+r}\otimes \bold{f}_{\delta^{-1}}\subseteq \overline{\mathbb{Q}}_{p,\infty}\cdot 
f^*\otimes \frac{1}{t^{1+r}}\bold{e}_{1+r}\otimes \bold{e}_{\delta^{-1}},
\end{multline}
and also have 
\begin{multline}
\bold{D}_{\mathrm{dR}}^0(D(-r)(\delta))=\bold{D}^0_{\mathrm{dR}}(V_f(k)(-r)(\delta))\\
=\overline{\mathbb{Q}}_p\cdot f\otimes \frac{1}{t^{k-r}}\bold{e}_{k-r}\otimes \bold{f}_{\delta}\subseteq 
\overline{\mathbb{Q}}_{p,\infty}\cdot f\otimes \frac{1}{t^{k-r}}\bold{e}_{k-r}\otimes \bold{e}_{\delta}.
\end{multline}

By the interpolation property (\ref{law}) of $\bold{z}^{(p)}_{\gamma_*}(f^*)$, if we set 
$$\mathrm{exp}^*((\bold{z}^{(p)}_{\gamma_*}(f^*)(1))_{r,\delta^{-1}})
=:a_{(r,\delta)}\cdot f^*\otimes \frac{1}{t^{r+1}}\bold{e}_{1+r}\otimes \bold{e}_{\delta^{-1}}\in \overline{\mathbb{Q}}_{p,\infty}\cdot 
f^*\otimes \frac{1}{t^{1+r}}\bold{e}_{1+r}\otimes \bold{e}_{\delta^{-1}},$$ then we have 
$a_{(r,\delta)}\in \overline{\mathbb{Q}}_{\infty}$ and 
\begin{multline*}
\mathrm{per}^{(r+1,\delta^{-1})}_{f^*}(a_{(r,\delta)}\cdot f^*\otimes \frac{1}{t^{r+1}}\bold{e}_{1+r}\otimes \bold{e}_{\delta^{-1}})=(2\pi i)^{r}\cdot L_{\{p\}}(f,\delta, k-r)\cdot
\gamma_*^{\mathrm{sgn}(r,\delta^{-1})}\\
=:\alpha_{(r,\delta)}\cdot\gamma_*^{\mathrm{sgn}(r,\delta^{-1})}.
\end{multline*}

By (\ref{law}) for $\bold{z}^{(p)}_{\gamma}(f)$, if we set 
$$\mathrm{exp}^*((\bold{z}_{\gamma}^{(p)}(f)(k)))_{-r,\delta})
=b_{(r,\delta)}\cdot f\otimes \frac{1}{t^{k-r}}\bold{e}_{k-r}\otimes \bold{e}_{\delta}\in \overline{\mathbb{Q}}_{p,\infty}\cdot 
f\otimes\frac{1}{t^{k-r}}\bold{e}_{k-r}\otimes \bold{e}_{\delta},$$ then we have $b_{(r,\delta)}\in \overline{\mathbb{Q}}_{\infty}$ and 
\begin{multline*}
\mathrm{per}^{(k-r,\delta)}_f(b_{(r,\delta)}\cdot f\otimes \frac{1}{t^{k-r}}\bold{e}_{k-r}\otimes\bold{e}_{\delta})=(2\pi i)^{k-r-1}\cdot L_{\{p\}}(f^*, \delta^{-1}, r+1)\cdot 
\gamma^{\mathrm{sgn}(k-r-1,\delta)}\\
=:\beta_{(r,\delta)}\cdot\gamma^{\mathrm{sgn}(k-r-1,\delta)}.
\end{multline*}


By Proposition \ref{key}, 
if we set
\begin{multline}
\mathrm{exp}^*((w_{\delta_D}(\bold{z}_{\gamma}^{(p)}(f)(k))\otimes\bold{e}_{\gamma}^{\vee}\otimes\bold{e}_1)_{r,\delta^{-1}})
=:c_{(r,\delta)}\cdot f\otimes t^k\bold{e}_{\gamma}^{\vee}\otimes \frac{1}{t^{k+r+1}}\bold{e}_{k+r+1}
\otimes \bold{e}_{\delta^{-1}}\\
\in \bold{D}_{\mathrm{dR}}^0(D^*(r)(\delta^{-1}))
\subseteq \overline{\mathbb{Q}}_{p,\infty}\cdot (f\otimes\frac{1}{t^k}\bold{e}_k)\otimes t^k\bold{e}_{\gamma}^{\vee}\otimes \frac{1}{t^{r+1}}\bold{e}_{r+1}
\otimes \bold{e}_{\delta^{-1}},
\end{multline}
 then we have 
 \begin{multline*}
c_{(r,\delta)}=\frac{\delta(-1)(-1)^r}{(k-r-1)!}\cdot\frac{r!}{(-1)^r}\cdot \frac{L^{(p)}(V_f(k-r)(\delta))}{L^{(p)}(V_{f^*}(1+r)(\delta^{-1}))}\cdot\varepsilon^{(p)}_L(V_f(k-r)(\delta))\cdot b_{(r,\delta)}
\\
=:d_{(r,\delta)}\cdot b_{(r,\delta)}
\end{multline*}
(remark that we have 
$$(w_{\delta_D}(\bold{z}_{\gamma}^{(p)}(f)(k))\otimes\bold{e}_{\gamma}^{\vee}\otimes\bold{e}_1)_{r,\delta^{-1}}=\delta(-1)(-1)^r\cdot([\sigma_{-1}]\cdot(w_{\delta_D}(\bold{z}_{\gamma}^{(p)}(f)(k))\otimes\bold{e}_{\gamma}^{\vee}\otimes\bold{e}_1))_{r,\delta^{-1}}).$$

Therefore, it suffices to show the equality
\begin{multline}\label{33}
\prod_{l\in S\setminus\{p\}}\frac{(l^{r}\cdot\delta(\sigma_l)^{-1})^{a^{(l)}(V_f(k))}}
{\mathrm{det}_L(-\varphi_l|V_f(k)^{I_l})}\cdot d_{(r,\delta)}\cdot \beta_{(r,\delta)}\cdot \gamma^{\mathrm{sgn}(k-r-1,\delta)}\bold{e}_k\otimes \bold{e}_{\gamma}^{\vee}\\
=-\alpha_{(r,\delta)}\cdot \gamma^{\mathrm{sgn}(r,\delta^{-1})}_*.
\end{multline}

 Remark that $\gamma^{\mathrm{sgn}(k)}\bold{e}_k\otimes\bold{f}_{\gamma}^{\vee}$ 
(resp. $\gamma^{\mathrm{sgn}(k-1)}\bold{e}_k\otimes \bold{f}_{\gamma}^{\vee}$) is sent to 
$-\gamma_*^-$ (resp. $\gamma_*^+$) under the canonical isomorphism
$$T_{f,\mathcal{O}_F}(k)\otimes_{\mathcal{O}_F}(\mathrm{det}_{\mathcal{O}_F}(T_{f,\mathcal{O}_F}(k)))^{\vee}\isom \mathrm{Hom}_{\mathcal{O}_F}(T_{f,\mathcal{O}_F}(k), \mathcal{O}_F)\isom T_{f^*,\mathcal{O}_F},$$
where the first isomorphism is defined by 
$$x\otimes\bold{f}_{\gamma}^{\vee}\mapsto 
[y\mapsto \bold{f}_{\gamma}^{\vee}(y\wedge x)],$$ and the second isomorphism 
is defined by the Poincar\'e duality.

 Hence, we obtain 
$$\gamma^{\mathrm{sgn}(k-r-1,\delta)}\bold{e}_k\otimes \bold{e}_{\gamma}^{\vee}=(-1)^{r}\delta(-1)\cdot \varepsilon_{0}^{-1}\cdot \gamma_*^{\mathrm{sgn}(r,\delta^{-1})}.$$
Therefore, the equality (\ref{33}) is equivalent to the equality 
\begin{equation}\label{34}
\delta(-1)(-1)^{r}\cdot\prod_{l\in S\setminus\{p\}}\frac{(l^{r}\cdot\delta(\sigma_l)^{-1})^{a^{(l)}
(V_f(k))}}{\mathrm{det}_L(-\varphi_l|V_f(k)^{I_l})}\cdot \varepsilon_{0}^{-1}\cdot d_{(r,\delta)}\cdot \beta_{(r,\delta)}=-\alpha_{(r,\delta)}.
\end{equation}
Since we have 
\begin{multline*}
\frac{(l^{r}\cdot\delta(\sigma_l)^{-1})^{a^{(l)}
(V_f(k))}}{\mathrm{det}_L(-\varphi_l|V_f(k)^{I_l})}\cdot\varepsilon_{0,\mathcal{O}_L}^{(l)}(T_f(k))
=\frac{(l^{-r}\cdot\delta(\sigma_l))^{-(a^{(l)}
(V_f(k))+\mathrm{dim}_LV_f(k)^{I_l})}}{\mathrm{det}_L(-\varphi_l|V_f(k-r)(\delta)^{I_l})}\cdot\varepsilon_{0,L}^{(l)}(V_f(k))
\\
=\frac{\varepsilon^{(l)}_{0,L}(V_f(k-r)(\delta))}{\mathrm{det}_L(-\varphi_l|V_f(k-r)(\delta)^{I_l})}
=\varepsilon_L^{(l)}(V_f(k-r)(\delta)),
\end{multline*} 
the left hand side of (\ref{34}) is equal to 
\begin{equation}\label{54.000}
\frac{(-1)^r\cdot r!}{(k-r-1)!}\cdot\prod_{l\in S}\varepsilon_L^{(l)}(V_f(k-r)(\delta))\cdot 
\frac{L^{(p)}(V_f(k-r)(\delta))}{L^{(p)}(V_{f^*}(r+1)(\delta^{-1}))}\cdot (2\pi i)^{k-r-1}\cdot L_{\{p\}}(f^*,\delta^{-1},r+1).
\end{equation}

Since we have 
$$\frac{(k-r-1)!}{(2\pi)^{k-r}}\cdot L(f,\delta,k-r)=i^{k+1}\cdot\prod_{l\in S}\varepsilon_l(f,\delta,k-r)
\cdot \frac{r!}{(2\pi)^{r+1}}\cdot L(f^*,\delta^{-1},r+1)$$
by evaluating at $s=k-r$ of the functional equation (\ref{fe}) of $L(f,\delta,s)$ 
and the $\varepsilon$-and $L$-constants for $V_f$ correspond to those for $\pi_f$ 
by the global-local compatibility (\cite{La73}, \cite{Ca86} for $l\not=p$, and \cite{Sc90}, \cite{Sa97} for $l=p$), the value (\ref{54.000}) is equal to 
\begin{multline*}
\frac{(-1)^r\cdot (2\pi)^{k}\cdot L^{(p)}(V_f(k-r)(\delta))}{i^{r+2}\cdot(k-r-1)!}
\cdot \left(i^{k+1}\cdot\prod_{l\in S}\varepsilon_l(f,\delta,k-r)\cdot\frac{r!}{(2\pi)^{r+1}}\cdot L(f^*,\delta^{-1},r+1)\right)\\
=\frac{(-1)^r\cdot (2\pi)^{k}\cdot L^{(p)}(V_f(k-r)(\delta))}{i^{r+2}\cdot(k-r-1)!}
\cdot\left(\frac{(k-r-1)!}{(2\pi)^{k-r}}\cdot L(f,\delta,k-r)\right)\\
=\frac{(-1)^r}{i^{r+2}}\cdot (2\pi)^r\cdot L_{\{p\}}(f,\delta,k-r)
=-(2\pi i)^r\cdot L_{\{p\}}(f,\delta,k-r)=-\alpha_{(r,\delta)},
\end{multline*}
which shows the equality (\ref{34}), hence finishes to prove the theorem.

\end{proof}

\subsection*{Acknowledgement}

The author thanks Seidai Yasuda for reading the manuscript carefully. 
He also thanks Iku Nakamura for constantly encouraging him. 
This work is supported in part by the Grant-in-aid 
(NO. S-23224001) for Scientific Research, JSPS.

\end{document}